\documentclass[arxiv]{agn_article}

\PassOptionsToPackage{sortcites}{agn_bib}

\usepackage{framed}
\usepackage{tabularx}
\usepackage{caption}
\usepackage{lscape}
\usepackage{rotating}
\usepackage{float}
\usepackage{wrapfig}
\usepackage{amsmath,amsfonts,amssymb,amsthm}
\usepackage[T1,T2A]{fontenc}
\usepackage[bibtex]{agn_all}

\usepackage{mathrsfs}
\usepackage{graphicx}
\usepackage{multicol}
\usepackage[inline]{enumitem}
\usepackage{array,multirow}
\usepackage[normalem]{ulem}
\usepackage{arydshln}
\usepackage[super]{nth}
\usepackage{framed}
\usepackage{enumitem}
\usepackage{amsmath,amsfonts,amssymb,amsthm,wasysym}
\PassOptionsToPackage{sortcites}{agn_bib}
\theoremstyle{plain}

\theoremstyle{definition}

\theoremstyle{remark}
\newtheorem{rem}{Remark}

\newcommand{\UE}{\overline{U}_e}
\newcommand{\UEN}{\overline{U}\hspace{-0.07cm}_e\hspace{-0.08cm}^{\natural}}
\newcommand{\UHN}{\overline{U}\hspace{-0.07cm}_h\hspace{-0.08cm}^{\natural}}
\newcommand{\UJN}{\overline{U}\hspace{-0.07cm}_{h_j}\hspace{-0.2cm}^{\natural}}

\newcommand{\os}{\omega}

\newcommand{\skw}{{\rm skew}}

\def\dd{\displaystyle}

\renewcommand\qedsymbol{$\blacksquare$}

\usepackage{agn_authors}
\title{ A geometrically nonlinear Cosserat (micropolar) curvy shell model via Gamma convergence}
\author{ Maryam Mohammadi Saem\thanks{Maryam Mohammadi Saem,  \ \  Lehrstuhl f\"{u}r Nichtlineare Analysis und Modellierung, Fakult\"{a}t f\"{u}r
		Mathematik, Universit\"{a}t Duisburg-Essen,  Thea-Leymann Str. 9, 45127 Essen, Germany, email: maryam.mohammadi-saem@uni-due.de} \quad and \quad Ionel-Dumitrel Ghiba\thanks{ Ionel-Dumitrel Ghiba,  \  {Faculty} of Mathematics,  Alexandru Ioan Cuza University of Ia\c si,  Blvd.
		Carol I, no. 11, 700506 Ia\c si,
		Romania; and  Octav Mayer Institute of Mathematics of the
		Romanian Academy, Ia\c si Branch,  700505 Ia\c si, email:  dumitrel.ghiba@uaic.ro} \quad
	and  \\    Patrizio Neff\,\thanks{Patrizio Neff,  \ \ Head of Lehrstuhl f\"{u}r Nichtlineare Analysis und Modellierung, Fakult\"{a}t f\"{u}r
		Mathematik, Universit\"{a}t Duisburg-Essen,  Thea-Leymann Str. 9, 45127 Essen, Germany, email: patrizio.neff@uni-due.de}
}
\date{\today}
\begin{document}
\maketitle

 \begin{abstract}
 	\noindent
 Using $\Gamma$-convergence arguments, we construct a nonlinear membrane-like Cosserat shell model on a curvy reference configuration starting from a geometrically nonlinear, physically linear three-dimensional isotropic Cosserat model. Even if the theory is of order $O(h)$ in the shell thickness $h$, by comparison to the membrane shell models proposed in classical nonlinear elasticity,  beside the change of metric, the membrane-like Cosserat shell model is still capable to capture the transverse shear deformation and the  {Cosserat}-curvature due to remaining Cosserat effects.  
 We formulate the limit problem  by scaling both unknowns, the deformation and the microrotation tensor, and by expressing the parental three-dimensional Cosserat energy with respect to a fictitious flat configuration. The model obtained via $\Gamma$-convergence is similar to the membrane  {(no  $O(h^3)$ flexural terms, but still depending on the Cosserat-curvature)} Cosserat shell model derived via a derivation approach but these two models do not coincide. Comparisons to other shell models are also included.
 \end{abstract}
 \\[.65em]
 {\parskip=-0.5mm \tableofcontents}
 
 \section{Introduction}
 If a three-dimensional elastic body is very thin in one direction, it has special load-bearing capacities. Due to the geometry, it is always tempting to try to come up with simplified equations for this situation. The ensuing theory is subsumed under the name shell theory. We speak of a flat shell problem  if the reference configuration is flat, i.e., the undeformed configuration  is given by $\Omega_h=\omega\times\big[-\frac{h}{2},\frac{h}{2}\big]$, with $\omega\subset \mathbb{R}^2$ and $h\ll1$, and of a shell (or curvy shell) if the reference configuration is curvy, in the sense that the undeformed configuration is given by  $\Omega_\xi=\Theta(\Omega_h)$, with $\Theta$ a $C^1$-diffeomorphism $\Theta\col\mathbb{R}^3\rightarrow \mathbb{R}^3$.
  
  There are many different ways to mathematically describe the response of shells and of obtaining two-dimensional field equations. One method is called the \textit{derivation approach}. The idea of this method is reducing the dimension of a given 3 dimensional model to 2 dimensions through physically reasonable constitution assumptions on the kinematics \cite{Koiter69}. The last author has introduced this derivation procedure based on the geometrically nonlinear Cosserat model in his habilitation thesis \cite{neff2004geometricallyhabil,neff2004geometrically}. The other approach is the \textit{intrinsic approach} which from the beginning views the shell as a two-dimensional surface and refers to methods from differential geometry \cite{Altenbach-Erem-11,Altenbach-Erem-Review,green1965general}. The \textit{asymptotic method} seeks, by using the formal expansion of the three-dimensional solution in power series in terms of a small thickness parameter to establish two-dimensional equations. Moreover, the \textit{direct approach} \cite{Naghdi69} assumes that the shell is a two-dimensional medium which has additional extrinsic directors in the concept of a restricted Cosserat surface (\cite{antman2005nonlinear, cohen1966nonlinear1, cohen1966nonlinear, cohen1989mathematical, cosserat1909theorie, ericksen1957exact, green1965general, rubin2013cosserat,birsan2019refined,boyer2017poincare}). Of course, the intrinsic approach is related to the direct approach. More information regarding to this method can be found in \cite{neff2004geometrically, neff2005geometrically, neff2005gamma, neff2007geometrically}.

  One of the most famous shell theories is the \textit{Reissner-Mindlin membrane-bending model} which is an extension of the \textit{Kirchhoff-Love} membrane-bending model \cite{anicic1999formulation} (the Koiter model \cite{anicic2019existence}). The kinematic assumptions in this theory are that straight lines normal to the reference mid-surface remain straight and normal to the mid-surface after deformation. The Reissner-Mindlin theory can be applied for thick plates and it does not require the cross-section to be perpendicular to the axial axes after deformation, i.e. it includes transverse shear. A serious drawback of both these theories is that a geometrically nonlinear, physically linear membrane-bending model is typically not well-posed (\cite{Neff2022}) and needs specific modifications \cite{anicic2018polyconvexity,anicic2019existence} to re-establish  well-posedness.
  
  There is another powerful tool that one can use to perform the dimensional reduction namely $\Gamma$- \textit{convergence}. In this case, a given 3D model is  {dimensionally reduced} via physically reasonable assumptions on the scaling of the energy.

In this regard, one of the first advances in finite elasticity was the derivation of a nonlinear membrane model (energy scaling with $h$) which is given in \cite{le1996membrane}. After that, the idea of $\Gamma$-convergence was developed in \cite{friesecke2003derivation, friesecke2002foppl, friesecke2002theorem, friesecke2006hierarchy}, where different scalings on the applied forces are considered, see also \cite{braun2016existence,schmidt2008linear}.
 
 A notorious property of the $\Gamma$-limit model based on classical elasticity is its de-coupling of the limit into either a membrane-like (scaling with $h$) or bending-like problem (scaling with $h^3$), see e.g. \cite{bartels2022nonlinear,hornung2014derivation}. 
 
 In this paper we will use the idea of $\Gamma$-convergence to deduce our two-dimensional curvy shell model from a 3-dimensional geometrically nonlinear Cosserat model (\cite{neff2010reissner}). This work is a challenging extension of the Cosserat membrane $\Gamma$-limit for flat shells, which was previously obtained by Neff and Chelminski in \cite{neff2007geometrically1}{,} to the situation of shells with initial curvature.
 
  The Cosserat model was introduced in 1909 by the Cosserat brothers \cite{cosserat1909theorie,cosserat1991note,cosserat1908theorie}. They imposed a principal of least action, combining the classical deformation $\varphi\col \Omega\subset \R^3\to \R^3$ and an independent triad of orthogonal directors, the microrotation $ {\overline{R}}\col \Omega\subset \R^3\to \SO(3)$. Invariance of the energy under superposed rigid body motions (left-invariance under $\SO(3)$) allowed them to conclude the suitable form of the energy as  {$W=W( {\overline{R}}^T\D \varphi, {\overline{R}^T\partial_{x_1} \overline{R}}, {\overline{R}^T\partial_{x_2} \overline{R}}, {\overline{R}^T\partial_{x_3} \overline{R}})$}. The balance of force equation appears by taking variations w.r.t $\varphi$ and balance of angular momentum follows from taking variations of $ {\overline{R}}\in \SO(3)$. Here, as additional structural assumption we will assume material isotropy, i.e., right-invariance of the energy under $\SO(3)$. In addition we will only consider a physically linear version of the model (quadratic energy in suitable strains)  which allows a complete and definite representation of the energy, see eq. (\ref{energy refeconf}).
  
  	
 In the geometric description of shells the normal to the midsurface and the tangent plane appear naturally and the Darboux-Frenet-frame can be used. The underlying Cosserat model immediately generalizes this concept in that the additional microrotation field $R$ can replace the Darboux-Frenet frame. The third column of the microrotation matrix $R$ generalizes the normal in a Kirchhoff-Love model and the director in a Reissner-Mindlin model. Note that the Cosserat model allows for global minimizers \cite{neff2004geometricallyhabil}.
 
 Concerning now the thin shell $\Gamma$-limit, we choose the nonlinear scaling and concentrate on a $O(h)$-model, i.e. the membrane response. 
 Since, however, the $3$-D Cosserat model already features curvature terms (derivatives of the microrotations), these terms "survive" the $\Gamma$-limit procedure and scale with $h$, while dedicated bending- {like} terms scaling with $h^3$ do not appear\!\! {\footnote{Observe that the surviving Cosserat curvature is not related to the change of curvature tensor, which measures the change of mean curvature and Gau\ss \ curvature of the surface, see  Acharya \cite{acharya2000nonlinear},   Anicic and Leg\'{e}r \cite{anicic1999formulation} as well as the recent work by Silhavy \cite{vsilhavycurvature} and \cite{GhibaNeffPartIII,GhibaNeffPartIV,GhibaNeffPartV,GhibaNeffPartVI}).}.} \
 
 The major difficulty  compared to the flat shell $\Gamma$-limit in \cite{neff2007geometrically1} is therefore the incorporation of the curved reference configuration. This problem is solved by introducing a multiplicative decomposition of the appearing fields into elastic and (compatible) permanent parts. The permanent parts encode the geometry of the curved surface given by $\Theta$. In this way, we are able to avoid completely the use of the intrinsic geometry of the curved shell.

 The related Cosserat shell model in \cite{GhibaNeffPartI,GhibaNeffPartII} is obtained by the derivation approach. There, the 2-dimensional model depends on the deformation of the midsurface $m\col \omega\to \R^3$ and the microrotation of the shell $\overline{Q}_{e,s}\col \omega\to \SO(3)$ for $\omega\subset \R^2$, the same as here. The resulting reduced energy contains a membrane part, membrane-bending part and bending-curvature part, while the Cosserat $\Gamma$-limit model obtained in this paper contains only the membrane energy and the curvature energy separately.
The membrane part is a combination of the shell energy and transverse shear energy and the curvature part includes the 2-dimensional Cosserat-curvature energy of the shell.

 	\begin{figure}[h!]
 		\begin{center}
 			\includegraphics[scale=1.1]{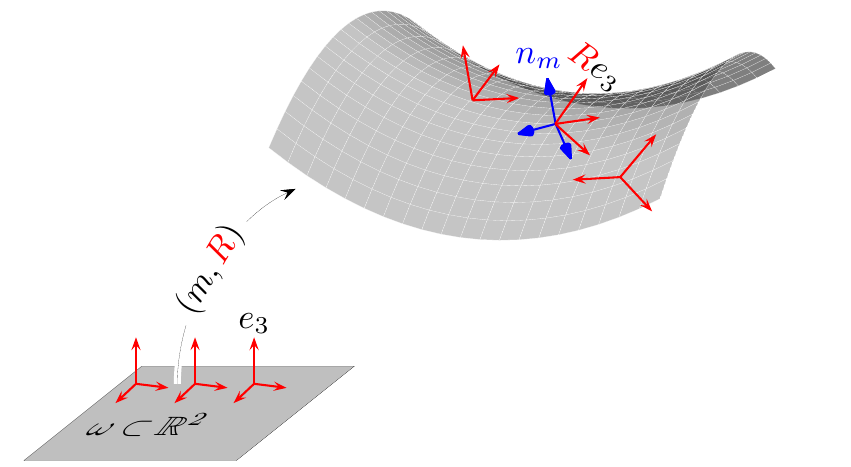}
 			
 			\caption{\footnotesize The mapping $m\col \os\subset \R^2\to\R^3$ describes the deformation of a flat midsurface $\os\subset \R^2$. The Frenet-Darboux frame (in blue, tri\`{e}dre cach\'{e}) is tangent to the midsurface $m$. The independent orthogonal frame mapped by $R\in \SO(3)$ is the tri\`{e}dre mobile (in red, not necessary tangent to the midsurface). Both fields $m$ and $R$ are coupled in the variational problem. This picture describes the situation of a flat Cosserat shell.}
 			\label{Fig1}       
 		\end{center}
 	\end{figure}

The present paper consists of 6 sections. After some notations in Section \ref{nota}, we start  {by} introducing the three dimensional isotropic nonlinear Cosserat model on the curved reference configuration $\Omega_\xi$  {formulated in terms of the} deformation $\varphi_\xi$ and microrotation $\overline{R}_\xi$. Then we transfer the problem to a variational problem defined on the fictitious flat configuration $\Omega_h$. For this goal, the diffeomorphism $\Theta\col \R^3\to \R^3$ will help us to transfer the deformation from $\Omega_h$ to $\Omega_c$ (the deformed configuration), $\Theta$ encodes the geometry of the curved reference configuration. For applying $\Gamma$-convergence arguments we need to transform our problem from $\Omega_h$ to a domain with fixed thickness $\Omega_1$. This action depends on the type of scaling of the variables, which is introduced in Section \ref{construc}. Next, we propose the admissible sets  {on which the $\Gamma$-convergence will be studied}.  We also obtain the family of functionals which are depending on the thickness $h$. From Section \ref{rescaled} on, we start to discuss the construction of the $\Gamma$-limit for the family of functionals $I_h$. After lengthy calculations, in Subsections \ref{homogen mem}  and \ref{homogen curv} we get the homogenized membrane and curvature energies. The main result of this work is presented in Section \ref{result gamma}, where we prove Theorem (\ref{Theorem homo}) on the $\Gamma$-limit. In Section \ref{sec load} we extend the $\Gamma$-limit theorem to  {the situation when  external loads are present}. Finally, in Section \ref{comparison}, we compare our model with other models{: a Cosserat flat shell model  obtained via $\Gamma$-convergence, a Cosserat shell model obtained via the derivation approach, a 6-parameter shell model, a Cosserat shell model up to $O(h^5)$, the Reissner-Mindlin membrane bending model and Aganovic and Neff's model.}

\section{Notation}\label{nota}
 Let $a,b\in \mathbb{R}^3$. We denote the scalar product on $\mathbb{R}^3$ with $\iprod{a,b}_{\mathbb{R}^3}$ and the associated vector norm with $\norm{a}^2_{\mathbb{R}^3}=\iprod{a,a}_{\mathbb{R}^3}$. The set of real-valued $3\times3$ second order tensors is denoted by $\mathbb{R}^{3\times3}$, where the elements are shown in capital letters. 
  The standard Eucliden scalar product on $\mathbb{R}^{3\times 3}$ is given by $\iprod{X,Y}_{\mathbb{R}^{3\times3}}=\tr (XY^\text{T})$, and the associated norm is $\norm{X}^2=\iprod{X,X}_{\mathbb{R}^{3\times3}}$. If $\id_3$ denotes the identity matrix in $\mathbb{R}^{3\times3}$, then we have $\tr(X)=\iprod{X,\id_3}$. For an arbitrary matrix $X\in \mathbb{R}^{3\times3}$, we define $\text{sym}(X)=\frac{1}{2}(X+X^{\text{T}})$ and $\text{skew}(X)=\frac{1}{2}(X-X^{\text{T}})$ as the symmetric and skew-symmetric parts, respectively and the deviatoric part is defined as $\dev X= X-\frac{1}{n}\tr(X)\id_n$, for all $X\in \R^{n\times n}$. We let ${\rm Sym}(n)$ and ${\rm Sym}\hspace{-0.05cm}^+\hspace{-0.05cm}(n)$ denote the symmetric and positive definite symmetric tensors, respectively. We consider the decomposition $X=\text{sym}(X)+\text{skew}(X)$ and the spaces
 \begin{align*}
 \text{GL}(3)&:=\{X\in \mathbb{R}^{3\times3}\;|\; \det X\neq 0\}\,,\hspace{2.5cm} \;\; \text{GL}^+(3):=\{X\in \mathbb{R}^{3\times3}\;|\; \det X> 0\}\,,\\
  \text{\rm SO}(3)&:=\{X\in \mathbb{R}^{3\times3}\;|\; X^TX=\id_3\,,\;\det X=1\}\,,\,\,\hspace{1cm} \mathfrak{so}(3):=\{A\in \mathbb{R}^{3\times3}\;|\; A^T=-A\}\,,\\
  \mathfrak{sl}(n)&:=\{X\in \R^{n\times n}\,|\, \tr(X)=0\}\,,\,\,\.\hspace{3cm}\text{O}(3):=\{X\in \text{GL}(3)\;|\; X^TX=\id_3\}\,.
  \end{align*}
  The \textit{canonical identification} of $\mathfrak{so}(3)$ and $\R^3$ is denoted by $\axl A\in \R^3$, for $A\in \mathfrak{so}(3)$. We have the following identities
  \begin{align}
  \axl\underbrace{\matr{0&\alpha&\beta\\-\alpha&0&\gamma\\-\beta&-\gamma&0}}_{=A}:=\matr{-\gamma\\ \beta\\-\alpha}\,,\qquad \qquad \qquad \norm{A}^2_{\R^{3\times 3}}=2\norm{\axl A}^2_{\R^3}\,.
  \end{align}
 We use
the \emph{orthogonal Cartan-decomposition  of the Lie-algebra} $\gl(3)$ of all three by three matrices with real components
\begin{align}
\mathfrak{gl}(3)&=\{\mathfrak{sl}(3)\cap {\rm Sym}(3)\}\oplus\mathfrak{so}(3) \oplus\mathbb{R}\!\cdot\! \id,\qquad\qquad
X=\dev \sym X+ \skw X+\frac{1}{3}\tr(X) \id\,\quad \forall \ X\in \mathfrak{gl}(3).
\end{align}
 A matrix having the  three  column vectors $A_1,A_2, A_3\in \R^3$ will be written as 
	$
	(A_1\,|\, A_2\,|\,A_3)$.

 Let $\Omega\subset \mathbb{R}^3$ be a bounded domain with Lipschitz boundary $\partial\Omega$ and $\Gamma\subset \partial\Omega$ be a smooth subset of the boundary of $\Omega$. In the two dimensional case, we assume that $\omega\subset \mathbb{R}^2$ with Lipschitz boundary $\partial\omega$ and $\gamma$ is also a smooth subset of $\partial\omega$. \\
 Assume that $\varphi\in C^1(\Omega,\mathbb{R}^3)$, then for the vector $x=(x_1,x_2,x_3)\in \mathbb{R}^3$ one can write $\nabla_x \varphi=(\partial_{x_1}\varphi|\partial_{x_2}\varphi|\partial_{x_3}\varphi)$. The standard volume element is $dx\.dy\.dz=dV=d\omega\, dz$. The mapping $m\col \omega\subset \R^2\to \R^3$ is the deformation of the midsurface and $\nabla m:=\nabla_{(x_1,x_2)} m$, is its gradient. We may write $m(x_1,x_2)=(x_1,x_2,0)+v(x_1,x_2)$, where $v\col \R^2\to \R^3$ is the displacement of the midsurface. 
\\For $ {1}\leq p<\infty$, we consider the \textit{Lebesgue spaces}   $
 {\rm L}^p(\Omega)=\{f:\Omega\to \R\;|\;\norm{f}_{{\rm L}^p(\Omega)}<\infty\}$ and their corresponding norms $  \norm{f}_{{\rm L}^p(\Omega)}:=\Big(\int_\Omega|f|^pdx\Big)^{\frac{1}{p}}\,.
$
 For $p\in [1,\infty]$, we define the \textit{Sobolev spaces}
 $
 {\rm W}^{1,p}(\Omega)=\{u\in {\rm L}^p(\Omega)\;|\; Du\in {\rm L}^p(\Omega)\}\,
 $,
  { $
 \norm{u}_{{\rm W}^{1,p}(\Omega)}^p=\norm{u}^p_{{\rm L}^p( \Omega)}+\norm{\D u}^p_{{\rm L}^p( \Omega)}$, where $\D u$ is the weak derivative of $u$. }
 In the case $p=2$, we set ${\rm H}^1(\Omega)={\rm W}^{1,2}(\Omega)$, where
 $
 {\rm H}^1(\Omega)=\{\varphi\in {\rm L}^2(\Omega) \; |\; \nabla \varphi\in {\rm L}^2(\Omega)\}\,$. For the energy function $W$ we define $\D W$  {as the  Fr\'echet derivative  of $W$ and $\D^2 W(F).(H,H)$ denotes the bilinear form of second derivatives.} 

 \section{The geometrically nonlinear three dimensional Cosserat model}\label{three model}

 \subsection{The variational problem defined on the thin curved reference configuration}
 Let us consider an elastic material which  in its reference configuration fills the three dimensional \textit{shell-like thin} domain $\Omega_\xi\subset\mathbb{R}^3$, i.e., we  assume that there exists a $C^1$-diffeomorphism $\Theta\col\mathbb{R}^3\rightarrow \mathbb{R}^3$ with $\Theta(x_1,x_2,x_3):=(\xi_1,\xi_2,\xi_3)$ such that $\Theta(\Omega_h)=\Omega_\xi$ and $\omega_\xi=\Theta(\omega\times\{0\})$, where  $\Omega_h\subset \mathbb{R}^3$ with $\Omega_h=\omega\times\big[-\frac{h}{2},\frac{h}{2}\big]$, and $\omega\subset \mathbb{R}^2$  a bounded domain with Lipschitz boundary $\partial\omega$. The scalar $0<h\ll1$ is called \textit{thickness} of the shell, while the domain $\Omega_h$ is called \textit{fictitious flat Cartesian configuration} of the body. We consider the following diffeomorphism $\Theta\col\mathbb{R}^3\rightarrow \mathbb{R}^3$ which is used to describe the curved surface of the shell
	\begin{align}
	\Theta(x_1,x_2,x_3)=y_0(x_1,x_2)+x_3\,n_0(x_1,x_2)\,,
	\end{align}
	where $y_0\col\omega\rightarrow \mathbb{R}^3$ is a $C^2(\omega)$-function  and $n_0=\frac{\partial_{x_1}y_0\times \partial_{x_2}y_0}{\norm{\partial_{x_1}y_0\times \partial_{x_2}y_0}}$ is the unit normal vector on $\omega_\xi$. 
	Remark that
\begin{align}
\nabla_x \Theta(x_3)&\,=\,(\nabla y_0|n_0)+x_3(\nabla n_0|0) \, \  \forall\, x_3\in \left(-\frac{h}{2},\frac{h}{2}\right),
\ \ 
\nabla_x \Theta(0)\,=\,(\nabla y_0|\,n_0),\ \ [\nabla_x \Theta(0)]^{-T}\, e_3\,=n_0,
\end{align}
and  $\det\nabla_x \Theta(0)=\det(\nabla y_0|n_0)=\sqrt{\det[ (\nabla y_0)^T\nabla y_0]}$ represents the surface element.

\begin{figure}[h!]
	\begin{center}
		\includegraphics[scale=1.6]{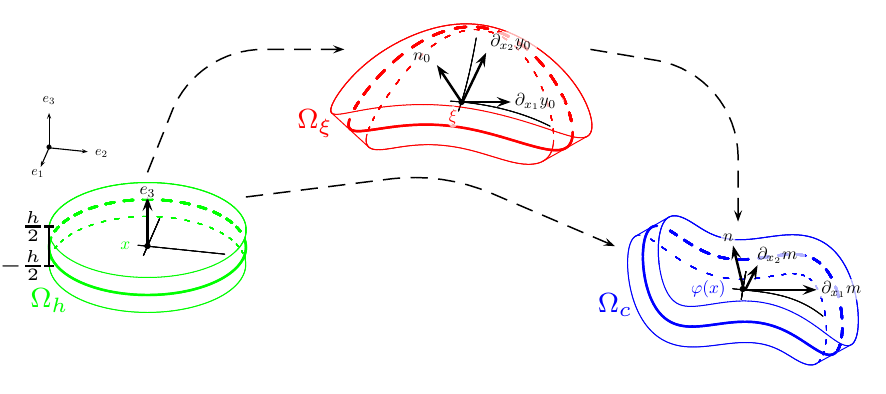}
		\put(-320,165){\footnotesize{$\Theta  ,  Q_0={\rm polar}(\nabla \Theta(0))$}} 
		\put(-320,175){\footnotesize{$\nabla\Theta(0)=(\nabla y_0|n_0)$}}
		\put(-275,82){\footnotesize{$\varphi  , \overline{ R}$}}
		\put(-65,125){\footnotesize{$\varphi_\xi  ,   \overline{ R}_\xi$}}
		\caption{\footnotesize Kinematics of the 3D-Cosserat model. In each point $\xi\in \Omega_\xi$ of the curvy reference  configuration, there is the deformation $\varphi_\xi\col \Omega_\xi\to \R^3$ and the microrotation $\overline{ R}_\xi\col \Omega_\xi\to \SO(3)$. We introduce a fictitious  flat configuration $\Omega_h$ and refer all fields to that configuration. This introduces a multiplicative split of the total deformation $\varphi\col\Omega_h\to \R^3$ and total rotation $\overline{ R}\col \Omega_h\to \SO(3)$ into ``elastic" parts ($\varphi_\xi\col \Omega_\xi\to \R^3$ and $\overline{ R}_\xi\col \Omega_\xi\to \SO(3)$) and compatible ``plastic" parts (given by $\Theta:\Omega_h\to \Omega_\xi$ and $Q_0:\Omega_h\to {\rm SO}(3)$). The "intermediate" configuration $\Omega_\xi$ is compatible by construction.}
		\label{Fig1}       
	\end{center}
\end{figure}
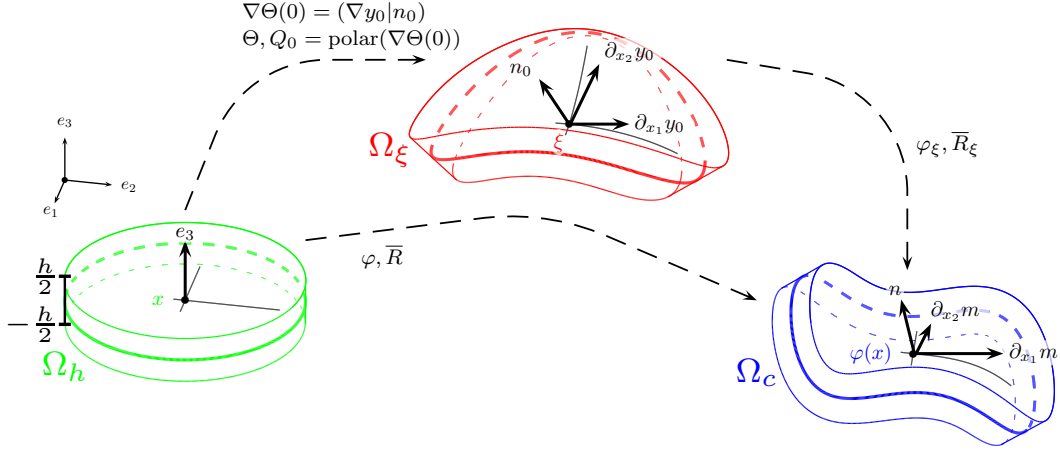

In the following we {identify} the  {\it Weingarten map (or shape operator)} {on $y_0(\omega)$} {with its associated matrix}  by 
$
{\rm L}_{y_0}\,=\, {\rm I}_{y_0}^{-1} {\rm II}_{y_0},
$
where ${\rm I}_{y_0}\coloneqq [{\nabla  y_0}]^T\,{\nabla  y_0}\in \mathbb{R}^{2\times 2}$ and  ${\rm II}_{y_0}\coloneqq\,-[{\nabla  y_0}]^T\,{\nabla  n_0}\in \mathbb{R}^{2\times 2}$ are  the matrix representations of the {\it first fundamental form (metric)} and the  {\it  second fundamental form} of the surface  {$y_0(\omega)$}, respectively. 
Then, the {\it Gau{\ss} curvature} ${\rm K}$ of the surface {$y_0(\omega)$} is determined by
$
{\rm K} \,=\,{\rm det}\,{{\rm L}_{y_0}}\, 
$
and the {\it mean curvature} $\,{\rm H}\,$ through
$
2\,{\rm H}\, :={\rm tr}({{\rm L}_{y_0}}) \, .
$ We denote  the principal curvatures of the surface by  ${\kappa_1}$ and ${\kappa_2}$.

	We note that $\det\nabla\Theta(x_3)= 1-2\, {\rm{H}}\,x_3+{\rm{K}}\, x_3^2=(1-\kappa_1\,x_3)(1-\kappa_2\, x_3)>0$. 
	 Therefore, $1-2\, {\rm{H}}\,x_3+{\rm{K}}\, x_3^2>0$, $\forall \, x_3\in \left[-\nicefrac{h}{2},\nicefrac{h}{2}\right]$ if and only if $1>\kappa_1\,x_3$ and $1>\kappa_2\, x_3$, for all $x_3\in \left[-\nicefrac{h}{2},\nicefrac{h}{2}\right]$. These conditions are equivalent with  $|\kappa_1|\, \frac{h}{2}<1$ and $|\kappa_2|\, \frac{h}{2}<1$, i.e., ~equivalent with  
	\begin{align}\label{ch5in}
	h\,\max \{\sup_{(x_1,x_2)\in {\omega}}|{\kappa_1}|,\sup_{(x_1,x_2)\in {\omega}}|{\kappa_2}|\}<2\,.\end{align}

We assume that after a deformation process given by  the function $\varphi_\xi:\Omega_\xi\to \mathbb{R}^3$, the curvy reference configuration $\Omega_\xi$ is mapped to the deformed configuration $\Omega_c=\varphi_\xi(\Omega_\xi)$. 

 In the Cosserat theory, each point of the reference body is  endowed with three independent orthogonal directors, i.e., with a matrix  $\overline{R}_\xi:\Omega_\xi\rightarrow \text{SO(3)}$ called the \textit{microrotation} tensor. Let us remark that while the tensor $ {\rm polar}(\nabla_\xi\varphi_\xi)\in \text{SO(3)}$ of the polar decomposition of $F_\xi:= \nabla_\xi\varphi_\xi={\rm polar}(\nabla_\xi\varphi_\xi)\sqrt{(\nabla_\xi\varphi_\xi)^T\nabla_\xi\varphi_\xi}$ is not independent of $\varphi_\xi$, the tensor $\overline{R}_\xi$ in the Cosserat theory is independent of $\nabla\varphi_\xi$. In other words, in general, $\overline{R}_\xi\neq {\rm polar}(\nabla_\xi\varphi_\xi)$.

In a geometrical nonlinear and physically linear Cosserat elastic 3D model, the deformation $\varphi_\xi$ and the microrotation $\overline{R}_\xi$ are the solutions of the following \textit{nonlinear minimization problem} on $\Omega_\xi$:
 \begin{align}\label{energy model}
 I(\varphi_\xi,F_\xi, \overline{R}_\xi, \alpha_\xi)=\int_{\Omega_\xi} \Big[W_{\rm mp}(\overline{U}_\xi)+W_{\rm curv}(\alpha_\xi)\Big]\; dV_\xi -\Pi(\varphi_\xi,\overline{R}_\xi) \mapsto \min. \qquad\text{w.r.t}\quad (\varphi_\xi,\overline{R}_\xi)\,,\
 \end{align}
 where
 \begin{align}\label{energy refeconf}
 F_\xi:\,=\,&\nabla_\xi\varphi_\xi\in\mathbb{R}^{3\times 3}\, \qquad \qquad \qquad \qquad \qquad \textrm{(the deformation gradient)},  \notag\\
 \overline U _\xi:\,=\,&\dd\overline{R}^T_\xi F_\xi\in\mathbb{R}^{3\times 3} \ \ \  \ \, \, \  \qquad \qquad \qquad \qquad \textrm{(the non-symmetric Biot-type stretch tensor)},  \notag\\
 \alpha_\xi:\,=\,&\overline{R}_\xi^T\, \Curl_\xi \,\overline{R}_\xi\in\mathbb{R}^{3\times 3}\qquad \qquad\qquad  \quad  \textrm{(the second order  dislocation density tensor \cite{Neff_curl08})}\, ,  \\
  W_{\rm{mp}}(\overline U _\xi):\,=\,&\mu\,\lVert \dev\,\text{sym}(\overline U _\xi-\id_3)\rVert^2+\mu_{\rm c}\,\lVert \text{skew}(\overline U _\xi-\id_3)\rVert^2+
 \frac{\kappa}{2}\,[{\rm tr}(\text{sym}(\overline U _\xi-\id_3))]^2\ \ \, \textrm{(physically linear)}\, ,  \notag\\
  W_{\rm{curv}}( \alpha_\xi):\,=\,&\mu\,{L}_{\rm c}^2\left( a_1\,\lVert \dev\,\text{sym}\, \alpha_\xi\rVert^2+a_2\,\lVert \text{skew}\, \alpha_\xi\rVert^2+  a_3\,
 [{\rm tr}(\alpha_\xi)]^2\right)\,\,\qquad \qquad\quad \textrm{(quadratic curvature energy)}, \notag
 \end{align}
 and $dV(\xi)$ denotes the  volume element in the $\Omega_\xi$-configuration.  The total stored energy can be seen by $W=W_{\rm mp}+W_{\rm curv}$, with $W_{\rm mp}$ as strain energy and $W_{\rm curv}$ as curvature energy. Clearly, $W$ depends on the deformation gradient $F_\xi=\nabla_\xi\varphi_\xi$ and the microrotation $\overline{R}_\xi$. The parameters $\mu\,$ and $\lambda$ are the \textit{Lam\'e constants}
 of classical isotropic elasticity, $\kappa=\displaystyle\frac{2\,\mu\,+3\,\lambda}{3}$ is the \textit{infinitesimal bulk modulus}, $\mu_{\rm c}> 0$ is the \textit{Cosserat couple modulus} and $L_c>0$ is the \textit{internal length} and responsible for \textit{size effects} in the sense that smaller samples are relatively stiffer than larger samples.  If not stated otherwise, we assume that $\mu\,>0$, $\kappa>0$, $\mu_{\rm c}> 0$. We also assume that $a_1>0, a_2>0$ and $a_3>0$, which assures the \textit{coercivity}  and \textit{convexity of the curvature} energy \cite{neff2007geometrically1}.

 The \textit{external loading potential}  $\Pi(\varphi_\xi,\overline{R}_\xi)$ is given by
 \begin{align*}
 \Pi(\varphi_\xi,\overline{R}_\xi)=\Pi_f(\varphi_\xi)+\Pi_c(\overline{R}_\xi)\,,
 \end{align*}
 where
 \begin{align*}
 \Pi_f(\varphi_\xi)&:=\int_{\Omega_\xi}\iprod{f,u_\xi}\, dV_\xi =\text{potential of external applied body forces}\; f\,,\\
 \Pi_c(\overline{R}_\xi)&:=\int_{\Gamma_\xi}\iprod{c,\overline{R}_\xi}\,dS_\xi =\text{potential of external applied boundary couple forces}\; c\,,
 \end{align*}
 with $u_\xi=\varphi_\xi-\xi$  the displacement vector. We will assume that the external loads satisfy in regularity condition:
 \begin{align}\label{regulartiycon}
 f\in {\rm L}^2(\Omega_\xi,\R^3)\,,\qquad\quad c\in {\rm L}^2(\Gamma_{\xi},\R^3)\,,\qquad\quad \overline{ R}_\xi\in {\rm L}^2(\Omega_\xi,\R^3)\,.
 \end{align}
  For simplicity, we consider only Dirichlet-type boundary conditions on $\Gamma_\xi=\gamma_\xi\times \big[-\frac{h}{2},\frac{h}{2}\big]$,  $\gamma_\xi \subset \partial \omega_\xi$,  i.e., we assume that $\varphi_\xi=\varphi^d_\xi$ on $\Gamma_\xi$,
 where $\varphi^d_\xi$ is a given function on $\Gamma_\xi$. 
 
 In \cite{neff2004geometrically} existence of minimizers is shown for positive Cosserat couple modulus $\mu_c>0$. The case $\mu_c=0$ can be handled as well with a slight modification of the curvature energy.
 The form of the curvature energy  $W_{\rm curv}$ is not that originally considered in \cite{Neff_plate04_cmt}.  Indeed, Neff \cite{Neff_plate04_cmt} used a curvature energy expressed in terms of the \textit{third order curvature tensor} $\mathfrak{A}_\xi=(\overline{R}_\xi^T\nabla (\overline{R}_\xi. e_1)\,|\,\overline{R}_\xi^T\nabla (\overline{R}_\xi. e_2)\,|\,\overline{R}_\xi^T\nabla (\overline{R}_\xi. e_3))$. The new form of the energy based on the \textit{second order dislocation density tensor} $\alpha_\xi$ simplifies considerably
 the representation by allowing to use the orthogonal decomposition
 \begin{align}\overline{R}_\xi^T\, \Curl_\xi \,\overline{R}_\xi=\alpha_\xi=\dev\, \sym \,\alpha_\xi+\skw\, \alpha_\xi+ \frac{1}{3}\,\tr(\alpha_\xi)\id_3.\end{align} Moreover, it yields an equivalent control of spatial derivatives of rotations \cite{Neff_curl08} and allows us to write the curvature energy   in a fictitious Cartesian configuration in terms of the so-called \textit{wryness tensor}  \cite{Neff_curl08,Pietraszkiewicz04}
 \begin{align}
 \Gamma_\xi&:= \Big(\mathrm{axl}(\overline{R}_\xi^T\,\partial_{\xi_1} \overline{R}_\xi)\,|\, \mathrm{axl}(\overline{R}_\xi^T\,\partial_{\xi_2} \overline{R}_\xi)\,|\,\mathrm{axl}(\overline{R}_\xi^T\,\partial_{\xi_3} \overline{R}_\xi)\,\Big)\in \mathbb{R}^{3\times 3},
 \end{align}
 since (see \cite{Neff_curl08})  the following close relationship between the \textit{wryness tensor}
 and the \textit{dislocation density tensor} holds
 \begin{align}\label{Nye1}
 \alpha_\xi\,=\,-\Gamma_\xi^T+\tr(\Gamma_\xi)\, \id_3\,, \qquad\textrm{or equivalently},\qquad \Gamma_\xi\,=\,-\alpha_\xi^T+\frac{1}{2}\tr(\alpha_
 \xi)\, \id_3\,.
 \end{align}
 For infinitesimal strains this formula is well-known under
 the name Nye's formula, and $-\Gamma$ is also called Nye's curvature tensor \cite{neff2008curl}. Our choice of the \textit{second order dislocation density tensor} $\alpha_\xi$  has some further implications, e.g., the coupling between the membrane part, the membrane-bending part, the bending-curvature part and the curvature part of the energy of the shell model is transparent and will coincide with shell-bending curvature tensors elsewhere considered \cite{Eremeyev06}.
 
 Within our assumptions on the constitutive coefficients, together with  the orthogonal Cartan-decomposition  of the Lie-algebra
 $
 \mathfrak{gl}(3)$ and with the definition
 \begin{align}\label{e78}
 {W}_{\mathrm{mp}}( X)
 \coloneqq &\, {W}_{\mathrm{mp}}^{\infty}( \sym \,X) +  \mu_{\rm c} \lVert  \mathrm{skew}   \,X\rVert ^2 \quad \ \forall \, X\in\mathbb{R}^{3\times 3},\\
 {W}_{\mathrm{mp}}^{\infty}( S)=&\,\mu\,\norm{S}^2 + \frac{\lambda}{2}[\tr(S)]^2\qquad \qquad \, \forall \, S\,\in{\rm Sym}(3),\notag
 \end{align}
 it follows that there exist positive constants  $c_1^+, c_2^+, C_1^+$ and $C_2^+$  such that for all $X\in \mathbb{R}^{3\times 3}$ the following inequalities hold
 \begin{align}\label{pozitivdef}
 C_1^+ \lVert S\rVert ^2&\geq\, {W}_{\mathrm{mp}}^{\infty}( S) \geq\, c_1^+ \lVert  S\rVert ^2 \qquad \qquad \qquad \qquad \qquad \qquad\ \forall \, S\,\in{\rm Sym}(3),\notag\\
\nonumber C_1^+ \lVert \sym\,X\rVert ^2+\mu_{\rm c}\,\lVert \skew\,X\rVert ^2&\geq\, W_{\mathrm{mp}}(  X) \geq\, c_1^+ \lVert  \sym\,X\rVert ^2+\mu_{\rm c}\,\lVert \skew\,X\rVert ^2 \,\,\,\quad \quad\forall \, X\in\mathbb{R}^{3\times 3},\\
 C_2^+ \lVert X \rVert^2  &\geq\, W_{\mathrm{curv}}(  X )
 \geq\,  c_2^+ \lVert X \rVert^2\,\,\qquad \qquad \qquad \qquad \quad \qquad  \forall \, X\in\mathbb{R}^{3\times 3}.
 \end{align}
 Here,  $c_1^+$ and $C_1^+$ denote respectively the  smallest and the largest eigenvalues of the quadratic form ${W}_{\mathrm{mp}}^{\infty}( X)$. Hence, they are independent of $\mu_{\rm c}$. Both ${W}_{\mathrm{mp}}$ and $W_{\mathrm{curv}}$ are   quadratic, convex and coercive functions of $\overline U _\xi$ and $\alpha_\xi$, respectively.
 
 The regularity condition of the external loads allows us to conclude that
 \begin{align}
 |\Pi_f(\varphi_\xi)|=|\int_{\Omega_\xi}\iprod{f,u_\xi}\,dV_\xi|\leq \norm{f}_{{\rm L}^2(\Omega_\xi)}\norm{u_\xi}_{{\rm L}^2(\Omega_\xi)}\,,
 \end{align}
 which implies that
 \begin{align}
 |\Pi_f(\varphi_\xi)|=|\int_{\Omega_\xi}\iprod{f,u_\xi}dV_\xi|\leq \norm{f}_{{\rm L}^2(\Omega_\xi)}\norm{u_\xi}_{{ {{\rm H}^{1}}}(\Omega_\xi)}\,.
 \end{align}
 Similarly we have
 \begin{align}
 |\Pi_c(\overline{ R}_\xi)|=|\int_{\Gamma_\xi}\iprod{c,\overline{ R}_\xi}dS_\xi|\leq \norm{c}_{{\rm L}^2(\Gamma_\xi)}\norm{\overline{ R}_\xi}_{{\rm L}^2(\Gamma_\xi)}\,.
 \end{align}
 Note that $\norm{\overline{ R}_\xi}^2=3$. By using the fact that $\norm{\overline{ R}_\xi}_{{\rm L}^2(\Gamma_\xi)}^2=(3\,\text{area}\,\Gamma_\xi)$, we get
 \begin{align}
 |\Pi(\varphi_\xi,\overline{ R}_\xi)|\leq \norm{f}_{{\rm L}^2(\Omega_\xi)}\norm{u_\xi}_{{ {{\rm H}^{1}}}(\Omega_\xi)}+\norm{c}_{{\rm L}^2(\Gamma_\xi)}(3\,\text{area}\,\Gamma_\xi)^{\frac{1}{2}}\,.
 \end{align}
 This boundedness  will be later used in the subject of $\Gamma$-convergence.
 \subsection{Transformation of the problem from $\Omega_\xi $ to the fictitious flat configuration $\Omega_h$}
 The first step in our  shell model is to transform the problem to a variational problem defined on the fictitious flat configuration $\Omega_h=\omega\times\big[-\frac{h}{2},\frac{h}{2}\big]$. This process is going to be done with the help of the diffeomorphism $\Theta$. To this aim, we define the mapping 
 \begin{align*}
 \varphi\col \Omega_h\rightarrow \Omega_c\,,\quad\quad \varphi(x_1,x_2,x_3)=\varphi_\xi(\Theta(x_1,x_2,x_3))\,.
 \end{align*}
 The function $\varphi$ maps $\Omega_h$ (fictitious flat Cartesian configuration) into $\Omega_c$ (deformed current configuration).  Moreover, we consider  the \textit{elastic microrotation}  $\overline{Q}_e\col\Omega_h\rightarrow \text{SO(3)}$ similarly defined by
 \begin{align}\label{elastic rot}
 \overline{Q}_e(x_1,x_2,x_3):=\overline{R}_\xi(\Theta(x_1,x_2,x_3))\,,
 \end{align}
 and the \textit{elastic Biot-type stretch tensor}  $\UE \col\Omega_h\rightarrow\mathbb{R}^{3\times 3}$ is then given by
 \begin{align}\label{elastic tens}
 \UE(x_1,x_2,x_3):=\overline{U}_\xi(\Theta(x_1,x_2,x_3))\,.
 \end{align}
 We also have the polar decomposition $\nabla_x\Theta=Q_0\,U_0$, where
 \begin{align}
 Q_0=\text{polar}(\nabla_x\Theta)=\text{polar}([\nabla_x\Theta]^{-T})\in \text{SO(3)}\quad\text{and}\quad U_0\in \text{Sym}^+(3)\,.
 \end{align}
 Now by using (\ref{elastic rot}), we define the {\it total microrotation} tensor
 \begin{align}
 \overline{R}\col\Omega_h\rightarrow \text{SO(3)},\qquad \overline{R}(x_1,x_2,x_3)=\overline{Q}_e(x_1,x_2,x_3)\,Q_0(x_1,x_2,x_3)\,.
 \end{align}
 By applying the chain rule for $\varphi$ one  obtains
 \begin{align}
 \nabla_x\varphi(x_1,x_2,x_3)=\nabla_\xi\varphi_\xi(\Theta(x_1,x_2,x_3))\,\nabla_x\Theta(x_1,x_2,x_3)\,,
 \end{align}
 or equivalently the {\it  multiplicative decomposition}
 \begin{align}
 F_\xi(\Theta(x_1,x_2,x_3))=F(x_1,x_2,x_3)\,[\nabla_x\Theta(x_1,x_2,x_3)]^{-1}\,.
 \end{align}
 Finally the\textit{ elastic non-symmetric stretch tensor} expressed on $\Omega_h$ can now be expressed as
 \begin{align}
 \UE=\overline{Q}_e^TF[\nabla_x\Theta]^{-1}=Q_0\overline{R}^TF[\nabla_x\Theta]^{-1}\,.
 \end{align}
 Note that 
 $
 \partial_{x_k}\overline{Q}_e=\sum_{i=1}^3 \partial_{\xi_i}\overline{R}_\xi\,\partial_{x_k}\xi_i,$  $\partial_{\xi_k}\overline{R}_\xi=\sum_{i=1}^3 \partial_{x_i}\overline{Q}_e\,\partial_{\xi_k}x_i
 $
 and
 \begin{align}
 \overline{R}_\xi^T\,\partial_{\xi_k} \overline{R}_\xi\,=\,&\sum_{i=1}^3(\overline{Q}_e^T \partial_{x_i}\overline{Q}_e)\,\partial_{\xi_k}x_i=\sum_{i=1}^3\
 \big(\overline{Q}_e^T\,\partial_{x_i} \overline{Q}_e\big)([\nabla_x \Theta]^{-1})_{ik}\,,\\
 \mathrm{axl}\big(\overline{R}_\xi^T\,\partial_{\xi_k} \overline{R}_\xi\big)\,=\,&\sum_{i=1}^3\
 \mathrm{axl}\big(\overline{Q}_e^T\,\partial_{x_i} \overline{Q}_e\big)([\nabla_x \Theta]^{-1})_{ik}.\notag
 \end{align}
 Thus, we have from the chain rule
 \begin{align}\label{crg}
 \Gamma_\xi\,=\,&
 \Big( \sum_{i=1}^3{\rm axl}\
 \Big(\overline{Q}_e^T\,\partial_{x_i} \overline{Q}_e\Big)([\nabla_x \Theta]^{-1})_{i1}\,\Big|\,  \sum_{i=1}^3{\rm axl}\
 \big(\overline{Q}_e^T\,\partial_{x_i} \overline{Q}_e\big)([\nabla_x \Theta]^{-1})_{i2}\,\Big|\, \sum_{i=1}^3{\rm axl}\
 \big(\overline{Q}_e^T\,\partial_{x_i} \overline{Q}_e\big)([\nabla_x \Theta]^{-1})_{i3}\Big)\vspace{1.5mm}\notag\\
 \,=\,&\Big(\mathrm{axl}(\overline{Q}_e^T\,\partial_{x_1} \overline{Q}_e)\,|\, \mathrm{axl}(\overline{Q}_e^T\,\partial_{x_2} \overline{Q}_e)\,|\,\mathrm{axl}(\overline{Q}_e^T\,\partial_{x_3} \overline{Q}_e)\,\Big)[\nabla_x \Theta]^{-1}.
 \end{align}
 We recall again  Nye's formula
 \begin{align}
 \alpha_\xi=-\Gamma_{\xi}^T+\tr(\Gamma_{\xi})\.\id_3\,,\qquad\text{or}\qquad \Gamma_{\xi}=-\alpha_\xi^T+\frac{1}{2}\tr(\alpha_\xi)\.\id_3\,.
 \end{align}
 Define 
 \begin{align}
 \Gamma\hspace{-0,05cm}_e:=\Big(\mathrm{axl}(\overline{Q}_e^T\,\partial_{x_1} \overline{Q}_e)\,|\, \mathrm{axl}(\overline{Q}_e^T\,\partial_{x_2} \overline{Q}_e)\,|\,\mathrm{axl}(\overline{Q}_e^T\,\partial_{x_3} \overline{Q}_e)\,\Big), \qquad\alpha_e:=\overline{Q}_e^T\Curl_x\,\overline{Q}_e.
 \end{align}
 Using   Nye's formula for $\alpha_e$ and $\Gamma\hspace{-0,05cm}_e$, 
 we deduce (see \cite{GhibaNeffPartI})
 \begin{align}
 \alpha_\xi\,=\,&\,[\nabla_x \Theta]^{-T}
 \alpha_e-\frac{1}{2}\tr(\alpha_
 e)\, [\nabla_x \Theta]^{-T}-\tr(\, [\nabla_x \Theta]^{-T}\alpha_e)\, \id_3+\frac{1}{2}\tr(\alpha_
 e)\, \tr([\nabla_x \Theta]^{-1})\, \id_3\vspace{1.5mm}\notag\\
 \,=\,&\,[\nabla_x \Theta]^{-T}
 \alpha_e-\tr(\alpha_e^T\,[\nabla_x \Theta]^{-1})\, \id_3-\frac{1}{2}\tr(\alpha_
 e)\, \Big([\nabla_x \Theta]^{-T}-\, \tr([\nabla_x \Theta]^{-1})\, \id_3\Big).
 \end{align}
 However, we will not use this formula to rewrite the curvature energy in the fictitious Cartesian configuration $\Omega_h$, since it is easier to use (from \eqref{Nye1})
 \begin{align}
 \sym \,\alpha_\xi\,=\,&-\sym\,\Gamma_\xi+ \tr(\Gamma_\xi)\, \id_3\,=\,-\sym(\Gamma\hspace{-0,05cm}_e\,[\nabla_x \Theta]^{-1})+ \tr(\Gamma\hspace{-0,05cm}_e\,[\nabla_x \Theta]^{-1})\, \id_3, \vspace{1.5mm}\notag\\
 {\rm dev}\,\sym \,\alpha_\xi\,=\,&-{\rm dev}\,\sym\, \Gamma_\xi\,=\,-{\rm dev}\,\sym (\Gamma\hspace{-0,05cm}_e\,[\nabla_x \Theta]^{-1}),\vspace{1.5mm}\\
 {\skw}
 \,\alpha_\xi\,=\,&-\skw\,\Gamma_\xi\,=\,-\skw(\Gamma\hspace{-0,05cm}_e\,[\nabla_x \Theta]^{-1}),\vspace{1.5mm}\notag\\
 \tr(\alpha_\xi)\,=\,&-\tr(\Gamma_\xi)+ 3\,\tr(\Gamma_\xi)\,=\,2\,\tr(\Gamma_\xi)\,=\,2\,\tr(\Gamma\hspace{-0,05cm}_e\,[\nabla_x \Theta]^{-1}),\notag
 \end{align}
 for expressing the curvature energy in terms of $\Gamma\hspace{-0,05cm}_e\,[\nabla_x \Theta]^{-1}$ as
 \begin{align}\label{curalpha}
 W_{\rm{curv}}( \alpha_\xi)\,=\,&\mu\,{L}_{\rm c}^2\left( a_1\,\lVert \dev\,\text{sym} (\Gamma\hspace{-0,05cm}_e\,[\nabla_x \Theta]^{-1})\rVert^2+a_2\,\lVert \text{skew}(\Gamma\hspace{-0,05cm}_e\,[\nabla_x \Theta]^{-1})\rVert^2+4\,a_3\,
 [{\rm tr}(\Gamma\hspace{-0,05cm}_e\,[\nabla_x \Theta]^{-1})]^2\right).
 \end{align}
 Note that using \
 \begin{align}\label{curvfuraxl}
 \overline{Q}_e^T\,\partial_{x_i} \overline{Q}_e\,=\,Q_0\,\overline{R}^T\,\partial_{x_i} (\overline{R}\,Q_0^T)\,=\,Q_0 (\overline{R}^T\,\partial_{x_i} \overline{R})\,Q_0^T-Q_0(Q_0^T\partial_{x_i} Q_0)\,Q_0^T,\quad i\,=\,1,2,3\,,
 \end{align}
 and the invariance (\cite{GhibaNeffPartI}, relation (3.12))
 \begin{align}
 \mathrm{axl}(Q\, A\, Q^T)\,=\,Q\,\mathrm{axl}( A)\qquad  \forall\, Q\in {\rm SO}(3) \quad \text{and}\quad \forall\, A\in \mathfrak{so}(3),
 \end{align}
 we obtain  the following form of the {wryness tensor}  defined on $\Omega_h$
 \begin{align}\label{qiese}
 \Gamma(x_1,x_2,x_3) :\,=\,&\,\Gamma_\xi(\Theta(x_1,x_2,x_3))\,=\,\Gamma\hspace{-0,05cm}_e\,[\nabla_x \Theta]^{-1}\notag\\\,=\,&\,Q_0\Big[ \Big(\mathrm{axl}(\overline{R}^T\,\partial_{x_1} \overline{R})\,|\, \mathrm{axl}(\overline{R}^T\,\partial_{x_2} \overline{R})\,|\,\mathrm{axl}(\overline{R}^T\,\partial_{x_3} \overline{R})\,\Big)\\&\ \ \quad - \Big(\mathrm{axl}(Q_0^T\,\partial_{x_1} Q_0)\,|\, \mathrm{axl}(Q_0^T\,\partial_{x_2} Q_0)\,|\,\mathrm{axl}(Q_0^T\,\partial_{x_3} Q_0)\,\Big)\Big] \,[\nabla_x \Theta]^{-1}.\notag
 \end{align}
 Now the minimization problem on the curved reference configuration $\Omega_\xi$ is transformed to the fictitious flat Cartesian configuration $\Omega_h$ as follows
\begin{align}\label{energy model.fictituos}
  I=\int_{\Omega_h} \Big[W_{\rm mp}(\UE)+\widetilde{W}_{\rm curv}(\Gamma)\Big]\det(\nabla_x\Theta)\;dV-\widetilde{\Pi}(\varphi,\overline{Q}_e)\ \mapsto \min. \qquad\text{w.r.t}\quad (\varphi,\overline{Q}_e)\,,
\end{align}
where
\begin{align}
\nonumber W_{\rm mp}(\UE)&=\;\mu\,\norm{\sym (\UE-\id_3)}^2+\mu_c\,\norm{\skew (\UE-\id_3)}^2 + \frac{\lambda}{2}[\tr(\sym(\UE-\id_3))]^2\\
\nonumber  &= \mu\,\norm{\dev\sym(\UE - \id_3)}^2+\mu_c\,\norm{\skew(\UE - \id_3)}^2+\frac{\kappa}{2}[\tr(\sym(\UE-\id_3))]^2\,,\\
 \widetilde{W}_{\rm curv}(\Gamma)&\,=\,\mu\,{L}_{\rm c}^2 \left( a_1\,\lVert \dev \,\text{sym} \,\Gamma\rVert^2+a_2\,\lVert \text{skew} \,\Gamma\rVert^2+4\,a_3\,
 [{\rm tr}(\Gamma)]^2\right)\\
 \nonumber&=\,\mu\,{L}_{\rm c}^2 \left( b_1\,\lVert\sym \,\Gamma\rVert^2+b_2\,\lVert \skew \,\Gamma\rVert^2+\,b_3\,
 [{\rm tr}(\Gamma)]^2\right)\,,
 \end{align}
where $b_1=a_1, b_2=a_2$, $b_3=\frac{12a_3-a_1}{3}$ and $\widetilde{\Pi}(\varphi,\overline{Q}_e)=\widetilde{\Pi}_f(\varphi)+\widetilde{\Pi}_c(\overline{Q}_e)$, with the following forms
 \begin{align}
\widetilde {\Pi}_f(\varphi)&:=\Pi_f(\varphi_\xi)=\int_{\Omega_\xi}\iprod{f,u_\xi}\,dV_\xi =\int_{\Omega_h}\iprod{\widetilde{f},\widetilde{u}}\,dV\,,\notag\\
 \widetilde{\Pi}_c(\overline{Q}_e)&:=\Pi_c(\overline{R}_\xi)=\int_{\Gamma_\xi}\iprod{c,\overline{R}_\xi} \,dS_\xi=\int_{\Gamma_h}\iprod{\widetilde{c},\overline{Q}_e}\, dS\,,
 \end{align}
with $\widetilde{u}(x_i)=\varphi(x_i)-\Theta(x_i)$  the displacement vector, $\overline{R}=\overline{Q}_e\,Q_0$ the total microrotation,  the vector fields $ \widetilde f $ and  $ \widetilde c $  can be determined in terms of $   f $ and $   c $, respectively, for instance (see \cite[Theorem 1.3.-1 ]{Ciarlet98b}) 
\begin{align}
\widetilde{f}(x)&=f(\Theta(x))\det(\nabla_x\Theta),\qquad\qquad \  \widetilde{c}(x)=c\.(\Theta(x))\det(\nabla_x\Theta).
\end{align} 
 Note that regarding to the regularity condition (\ref{regulartiycon}), the following regularity conditions will hold as well
 \begin{align}\label{regu2}
 \widetilde{f}\in {\rm L}^2(\Omega_h,\R^3)\,,\qquad \qquad\widetilde{c}\in {\rm L}^2(\Gamma_h,\R^3)\,,\qquad\qquad  \overline{Q}_e\in {\rm L}^2(\Gamma_h,\R^3)\,.
 \end{align}
The Dirichlet-type boundary conditions (in the sense of the traces) on $\Gamma_\xi=\gamma_\xi\times \big[-\frac{h}{2},\frac{h}{2}\big]$, $\gamma_\xi\subset \partial \omega_\xi$, read on the boundary $\Gamma_h=\gamma\times   \big[-\frac{h}{2},\frac{h}{2}\big]$, $\gamma=\Theta^{-1}(\gamma_\xi)\subset \partial \omega$, as $\varphi=\varphi^h_d$ on $\Gamma_h$, where $\varphi^h_d=\Theta^{-1}(\varphi^h_\xi)$.

\section{Construction of the family of functionals $I_{h_j}$}\label{construc}
 
 \subsection{Nonlinear scaling for the deformation gradient and the microrotation}
 In order to apply  the methods of $\Gamma$-convergence,  the first step is to transform our problem further from $\Omega_h$ to a \textit{domain} with fixed thickness  $\Omega_1=\omega\times[-\frac{1}{2},\frac{1}{2}]\subset\mathbb{R}^3,\;\omega\subset \mathbb{R}^2$. For this goal, scaling of the variables (dependent/independent) would be the first step. However, it is important to know which kind of scaling is suitable for our variables. In this paper we  introduce only the \textit{nonlinear scaling}, although in linear models, a concept of \textit{linear} scaling is used as well (\cite{neff2007geometrically1, neff2010reissner}). For a vector field $z\col\Omega_h\rightarrow \mathbb{R}^3$ we consider the  \textit{nonlinear} scaling   $z^\natural\col\Omega_1\rightarrow \mathbb{R}^3$, where only the independent variables will be scaled
 \begin{align}\label{nonlinear sca}
 \nonumber x_1=\eta_1\.,\quad\qquad x_2=&\;\eta_2\., \qquad \quad x_3 = h\.\eta_3\.,\\
 z^\natural\Big(x_1,x_2,\frac{1}{h}x_3\Big):=& \;z(x_1,x_2,x_3)\,, \qquad \text{nonlinear scaling}\,.
 \end{align}
 Consequently, the gradient of $z(x)=(z_1(x),z_2(x),z_3(x))$ with respect to $x=(x_1,x_2,x_3)$ can be expressed in terms of the derivative of $z^\natural$ with respect to $\eta=(\eta_1,\eta_2,\eta_3)$
 \begin{align}
 \nabla_x z(x_1,x_2,x_3) =&\; \Big(\partial_{\eta_1}z^\natural(\eta_1,\eta_2,\eta_3)\;|\;\partial_{\eta_2}z^\natural(\eta_1,\eta_2,\eta_3)\;|\;\frac{1}{h}\partial_{\eta_3}z^\natural(\eta_1,\eta_2,\eta_3) \Big):= \nabla_\eta^h z^\natural(\eta)\,.
 \end{align}
 For more details about scaling of the variable we refer to \cite{neff2010reissner}. In all our computations the mark $\cdot^\natural$ indicates the nonlinear scaling.

 In a first step we will apply the nonlinear scaling to the deformation. For $\Omega_1=\omega\times\Big[-\displaystyle\frac{1}{2},\frac{1}{2}\Big]\subset\mathbb{R}^3$, $\omega\subset \mathbb{R}^2$, we define the scaling transformations
 \begin{align}
 \nonumber\zeta\col&\; \eta\in\Omega_1\mapsto \mathbb{R}^3\,, \qquad \zeta(\eta_1,\eta_2,\eta_3):=(\eta_1,\eta_2,h\,\eta_3)\,,\\
 \zeta^{-1}\col&\;x\in\Omega_h\mapsto\mathbb{R}^3\,, \qquad\zeta^{-1}(x_1,x_2,x_3):=(x_1,x_2,\frac{x_3}{h})\,,
 \end{align}
 with $\zeta(\Omega_1)=\Omega_h$. By using the relation (\ref{nonlinear sca}) and above transformations we obtain the formula for the transformed deformation $\varphi$ as 
 \begin{align}
 \nonumber\varphi(x_1,x_2,x_3)&=\varphi^\natural(\zeta^{-1}(x_1,x_2,x_3))\, \quad \forall x\in \Omega_h\.; \quad \varphi^\natural(\eta)=\varphi(\zeta(\eta))\quad \forall \eta\in \Omega_1\,,\\
 \nabla_x\varphi(x_1,x_2,x_3)&=\begin{pmatrix}\vspace{0.2cm}
 \partial_{\eta_1}\varphi_{1}^\natural(\eta)& \partial_{\eta_2}\varphi_{1}^\natural(\eta)&\displaystyle\frac{1}{h}\partial_{\eta_3}\varphi_{1}^\natural(\eta)\\\vspace{0.2cm}
 \partial_{\eta_1}\varphi_{2}^\natural(\eta)& \partial_{\eta_2}\varphi_{2}^\natural(\eta)&\displaystyle\frac{1}{h}\partial_{\eta_3}\varphi_{2}^\natural(\eta)\\\vspace{0.2cm}
 \partial_{\eta_1}\varphi_{3}^\natural(\eta)& \partial_{\eta_2}\varphi_{3}^\natural(\eta)&\displaystyle\frac{1}{h}\partial_{\eta_3}\varphi_{3}^\natural(\eta)\\
 \end{pmatrix}=\nabla_\eta^h\varphi^\natural(\eta)=F^\natural_h\..
 \end{align}
 Now we will do the same process for the microrotation tensor $\overline{Q}_e^\natural\col\Omega_1\rightarrow \text{SO(3)}$ 
 \begin{align}
 \overline{Q}_e(x_1,x_2,x_3)&=\overline{Q}_e^\natural(\zeta^{-1}(x_1,x_2,x_3))\, \quad \forall x\in \Omega_h\,;\; \;\overline{Q}_e^\natural(\eta)=\overline{Q}_e(\zeta(\eta))\,,\; \quad\forall \eta\in \Omega_1\,,\notag
 \end{align}
 as well as for $\nabla_x \Theta(x)$, the  matrices of its  polar decomposition $\nabla_x\Theta(x)=Q_0(x)U_0(x)$, in the sense that
 \begin{align}
 (\nabla_x\Theta)^\natural(\eta) =(\nabla_x \Theta)(\zeta(\eta)), \qquad Q_0^\natural(\eta)=Q_0(\zeta(\eta)), \qquad U_0^\natural(\eta)= U_0(\zeta(\eta))\,.
 \end{align} 
   We also define  $\overline{R}^\natural\col\Omega_1\rightarrow \text{SO(3)}$ 
 \begin{align}
 \overline{R}(x_1,x_2,x_3)&=\overline{R}^\natural(\zeta^{-1}(x_1,x_2,x_3))\, \quad \forall x\in \Omega_h\,;\; \qquad\overline{R}^\natural(\eta)=\overline{R}(\zeta(\eta))\,,\; \quad\forall \eta\in \Omega_1\,.\notag
 \end{align}
 With this, the non-symmetric stretch tensor expressed in a point of $\Omega_1$ is given by
  \begin{align}
 \UEN =\overline{Q}_e^{\natural,T}F^\natural_h [(\nabla_x\Theta)^\natural ]^{-1}=\overline{Q}_e^{\natural,T}\nabla_\eta^h\varphi^\natural(\eta)[(\nabla_x\Theta)^\natural ]^{-1}\,.
 \end{align}
 Since for $\eta_3=0$ their values expressed in terms of $(\eta_1, \eta_2,0)$ and $(x_1,x_2,0)$ coincide, we will omit the sign $\cdot ^\natural$ and we will understand from the context the variables into discussion, i.e.,  
 \begin{align}
 (\nabla_x \Theta)(0)&:=(\nabla y_0\,|n_0)=(\nabla_x \Theta)^\natural(\eta_1,\eta_2,0)\equiv (\nabla_x \Theta)(x_1,x_2,0), \notag\\ Q_0(0)&:=Q_0^\natural(\eta_1,\eta_2,0)\equiv Q_0(x_1,x_2,0), \qquad\qquad\qquad U_0(0):=U_0^\natural(\eta_1,\eta_2,0)\equiv U_0(x_1,x_2,0).\notag
 \end{align}
 Therefore, we have \begin{align}
 \overline{Q}_{e}^\natural (\eta)=\overline{R}^{\natural}(\eta)(Q_0^\natural(\eta))^T,\qquad  \UEN(\eta)=\overline{Q}_{e}^{\natural,T}(\eta)F_h^\natural(\eta)[(\nabla_x\Theta)^\natural ]^{-1}=Q_0^\natural(\eta)\overline{R}^{\natural,T}(\eta)F^\natural_h(\eta)[(\nabla_x\Theta)^\natural ]^{-1}\,,
\end{align}and
\begin{align}
\Gamma^\natural_{h}& =\Big(\mathrm{axl}(\overline{Q}_{e,h}^{\natural,T}\,\partial_{\eta_1} \overline{Q}_{e,h}^\natural)\,|\, \mathrm{axl}(\overline{Q}_{e,h}^{\natural,T}\,\partial_{\eta_2} \overline{Q}_{e,h}^\natural)\,|\,\frac{1}{h}\mathrm{axl}(\overline{Q}_{e,h}^{\natural,T}\,\partial_{\eta_3} \overline{Q}_{e,h}^\natural)\,\Big)[(\nabla_x\Theta)^\natural ]^{-1}
.
\end{align}

 \subsection{Transformation of the problem from $\Omega_h$ to a fixed domain $\Omega_1$}
 The next step, in order to apply the $\Gamma$-convergence technique, is to transform the minimization problem onto the \textit{fixed domain} $\Omega_1$, which is independent from the thickness $h$.
According to the results from the previous subsection, we have the following minimization problem on $\Omega_1$
 \begin{align}\label{energy model.fictituos.ond omega1}
 \nonumber &I^\natural_h(\varphi^\natural,\nabla_\eta^h\varphi^\natural,\overline{Q}_e^\natural,\Gamma^\natural_{h})
  =\int_{\Omega_1} \Big({W}_{\rm mp}(\UHN)+{\widetilde{W}}_{\rm curv}(\Gamma^\natural_{h})\Big)\det(\nabla_\eta\zeta(\eta))\det((\nabla_x\Theta)^\natural )\;dV_\eta-{\Pi}^\natural_h(\varphi^\natural,\overline{Q}_e^\natural)\\
 & =\underbrace{\int_{\Omega_1}\; h \;\Big[\Big({{W}}_{\rm mp}(\UHN)+{\widetilde{W}}_{\rm curv}(\Gamma^\natural_{h})\Big)\det((\nabla_x\Theta)^\natural )\Big]\;dV_\eta}_{:=J^\natural_h(\varphi^\natural,\nabla_\eta^h\varphi^\natural,\overline{Q}_e^\natural,\Gamma^\natural_{h})}-{\Pi}^\natural_h(\varphi^\natural,\overline{Q}_e^\natural)
 \mapsto \min\;\text{w.r.t}\; (\varphi^\natural,\overline{Q}_e^\natural)\,,
 \end{align}
 where 
 \begin{align}\label{boun co}
 \nonumber {{W}}_{\rm mp}(\UHN)&=\;\mu\,\norm{\sym (\UHN-\id_3)}^2+\mu_c\,\norm{\skew (\UHN-\id_3)}^2 + \frac{\lambda}{2}[\tr(\sym(\UHN-\id_3))]^2\,,\\
 \widetilde{W}_{\rm curv}(\Gamma^\natural_{h})&\,=\,\mu\, {L}_{\rm c}^2 \left( a_1\,\lVert \dev \,\text{sym} \,\Gamma^\natural_{h}\rVert^2+a_2\,\lVert \text{skew} \,\Gamma^\natural_{h}\rVert^2+\,a_3\,
 [{\rm tr}(\Gamma^\natural_{h})]^2\right)\,,
 \end{align}
 with  ${\Pi}^\natural_h(\varphi^\natural,\overline{Q}_e^\natural)={\Pi}_f^\natural(\varphi^\natural)+{\Pi}_{c}^\natural(\overline{Q}_e^\natural)$, 
  \begin{align}\label{load}
 {\Pi}_f^\natural(\varphi^\natural)&:=\widetilde{\Pi}_f(\varphi)=\int_{\Omega_h}\iprod{\widetilde{f},\widetilde{u}}\,dV=\int_{\Omega_1} \iprod{\widetilde{f}^\natural,\widetilde{u}^\natural}\,\det(\nabla_\eta\zeta(\eta))\,dV_\eta=h\,\int_{\Omega_1} \iprod{\widetilde{f}^\natural,\widetilde{u}^\natural}\,dV_\eta\,,\notag\\
{\Pi}_c^\natural(\overline{Q}_e^\natural)&:=\widetilde{\Pi}_c(\overline{Q}_e)=\int_{\Gamma_h}\iprod{\widetilde{c},\overline{Q}_e} dS=\int_{\Gamma_1} \iprod{\widetilde{c}^\natural,\overline{Q}_e^\natural}\,\det(\nabla_\eta\zeta(\eta))\,dS_\eta=h\,\int_{\Gamma_1} \iprod{\widetilde{c}^\natural,\overline{Q}_e^\natural}\,dS_\eta\,\,,
 \end{align}
 with $\widetilde{f}^\natural(\eta)=\widetilde{f}(\zeta(\eta))$, $\widetilde{u}^\natural(\eta)=\widetilde{u}(\zeta(\eta))$,
 $\widetilde{c}^\natural(\eta)=\widetilde{c}(\zeta(\eta))$ and $\overline{Q}_e^\natural(\eta)=\overline{Q}_e(\zeta(\eta))$. Here we recall that regarding to the regularity condition (\ref{regu2}), it holds
\begin{align}
\widetilde{f}^\natural\in {\rm L}^2(\Omega_1,\R^3)\,,\qquad\quad\widetilde{c}^\natural\in {\rm L}^2(\Gamma_1,\R^3)\,,\qquad\quad \overline{ Q}^\natural\in {\rm L}^2(\Gamma_1,\R^3)\,.
\end{align}
Therefore, we may write
\begin{align}
\nonumber|\Pi_f^\natural(\varphi^\natural)|&=|h\int_{\Omega_1}\iprod{\widetilde{f}^\natural,\widetilde{u}^\natural}dV_\eta|\leq h\norm{\widetilde{f}^\natural}_{{\rm L}^2(\Omega_1)}\norm{\widetilde{u}^\natural}_{{\rm L}^2(\Omega_1)}\,,\\
|\Pi_c^\natural(\overline{ Q}_e^\natural)|&=|h\int_{\Gamma_1}\iprod{\widetilde{c}^\natural,\overline{ Q}_e^\natural}dS_\eta|\leq h\norm{\widetilde{c}^\natural}_{{\rm L}^2(\Gamma_1)}\norm{\overline{ Q}_e^\natural}_{{\rm L}^2(\Gamma_1)}\,,
\end{align}
and consequently
\begin{align}
|\Pi_h^\natural(\varphi^\natural,\overline{ Q}_e^\natural)|\leq h\Big[\norm{\widetilde{f}^\natural}_{{\rm L}^2(\Omega_1)}\norm{\widetilde{u}^\natural}_{{\rm L}^2(\Omega_1)}+\norm{\widetilde{c}^\natural}_{{\rm L}^2(\Gamma_1)}\norm{\overline{ Q}_e^\natural}_{{\rm L}^2(\Gamma_1)}\Big]\,.
\end{align}

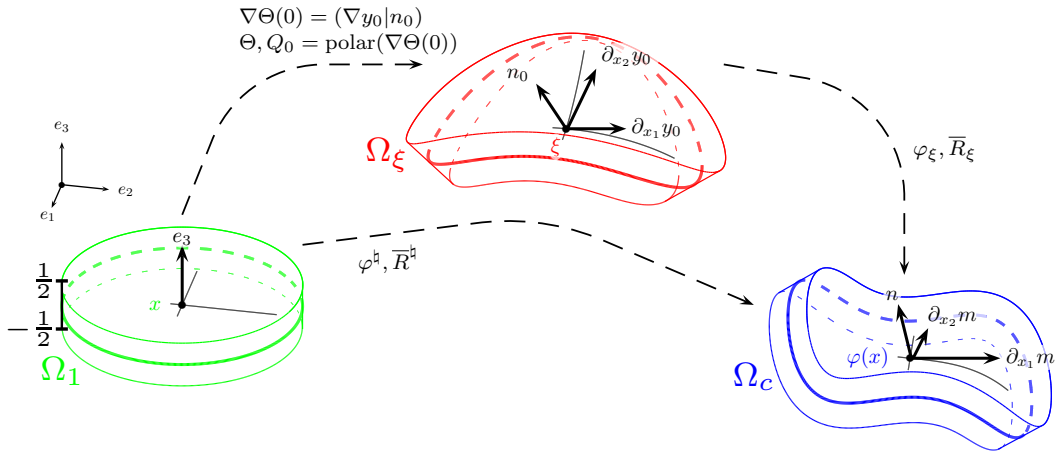
\begin{figure}[h!]
	\begin{center}
		\includegraphics[scale=1.6]{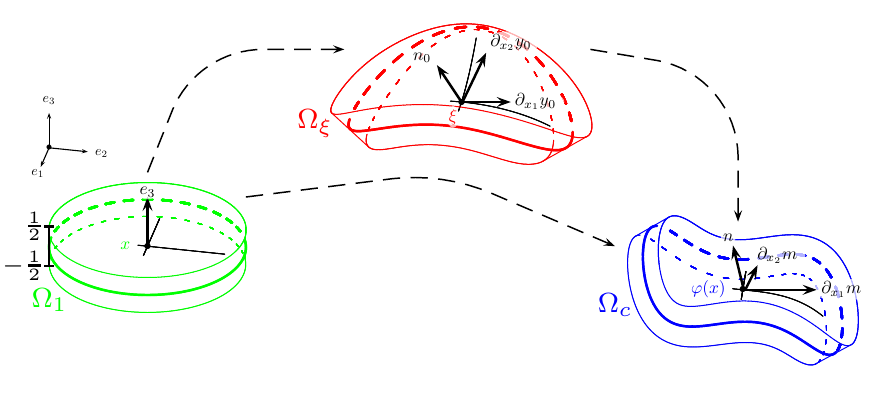}
		\put(-320,165){\footnotesize{$\Theta  ,  Q_0={\rm polar}(\nabla \Theta(0))$}}
		\put(-320,175){\footnotesize{$\nabla\Theta(0)=(\nabla y_0|n_0)$}} 
		\put(-275,82){\footnotesize{$\varphi^\natural  , \overline{ R}^\natural$}}
		\put(-65,125){\footnotesize{$\varphi_\xi  ,   \overline{ R}_\xi$}}
		\caption{\footnotesize The complete picture of the involved domains. $\Omega_1$ is the fictitious flat domain with unit thickness, $\Omega_\xi$ denotes the curved reference configuration, $\Omega_c$ is the current deformed configuration. Again, the reference configuration $\Omega_\xi$ takes on the role of a compatible intermediate configuration in the multiplicative decomposition. }
		\label{Fig2}       
	\end{center}
\end{figure}
The Dirichlet-type boundary conditions (in the sense of the trace) on $\Gamma_h=\gamma\times   \big[-\frac{h}{2},\frac{h}{2}\big]$, $\gamma=\Theta^{-1}(\gamma_\xi)\subset \partial \omega$, read on the boundary $\Gamma_1=\gamma\times   \big[-\frac{1}{2},\frac{1}{2}\big]$ as $\varphi^\natural=\varphi^\natural_d$ on $\Gamma_1$, where $\varphi^\natural_d=\Theta^{-1}(\varphi^h_d)$.  

\section{Equi-coercivity and compactness of the family of  energy functionals}\label{equi-comp1}

\subsection{The set of admissible solutions}
Due to the scaling, we have obtained a family of functionals 
\begin{align}
J^\natural_h(\varphi^\natural,\nabla_\eta^h\varphi^\natural,\overline{Q}_e^\natural,\Gamma_h^\natural)=\int_{\Omega_1}\; h \;\Big[\Big({{W}}_{\rm mp}(\UHN)+\widetilde{W}_{\rm curv}(\Gamma^\natural_{h})\Big)\det((\nabla_x\Theta)^\natural )\Big]\;dV_\eta\,,
\end{align} depending on the thickness $h$. The next step is to  prepare a suitable space $X$ on  which the existence of $\Gamma$-convergence will be studied. As already mentioned, for applying the $\Gamma$-limit  techniques   we need  {to work with a separable and metrizable space  $X$}. Since working in ${\rm H}^{1}(\Omega_1,\mathbb{R}^3)\times { {{\rm H}^{1}}}(\Omega_1,\text{SO(3)})$ means to consider the weak  {topology}, which does not give rise to a metric space, we introduce the following spaces:
\begin{align}\label{sets}
\nonumber X&:=\{(\varphi^\natural,\overline{Q}_e^\natural)\in {\rm L}^2(\Omega_1,\mathbb{R}^3)\times {\rm L}^{2}(\Omega_1, \text{SO(3)})\}\,,\\
\nonumber X'&:=\{(\varphi^\natural,\overline{Q}_e^\natural)\in {\rm H}^{1}(\Omega_1,\mathbb{R}^3)\times { {{\rm H}^{1}}}(\Omega_1, \text{SO(3)})\}\,,\\
X_\omega&:=
\{(\varphi,\overline{Q}_e)\in {\rm L}^2(\omega,\mathbb{R}^3)\times {\rm L}^{2}(\omega,\text{SO(3)})\}\,,\\
X'_\omega&:=\{(\varphi,\overline{Q}_e)\in {\rm H}^{1}(\omega,\mathbb{R}^3)\times { {{\rm H}^{1}}}(\omega,\text{SO(3)})\}\,.\notag
\end{align}
We also consider the following admissible sets
\begin{align}\label{admi set}
\nonumber\mathcal{S}':=&\;\{(\varphi,\overline{Q}_e)\in {\rm H}^{1}(\Omega_1,\mathbb{R}^3)\times { {{\rm H}^{1}}}(\Omega_1,\text{SO(3)})\,\big|\,\, \varphi|_{\Gamma_1}(\eta)=\varphi^\natural_d(\eta)\}\,,\\
\mathcal{S}'_\omega:=&\;\{(\varphi,\overline{Q}_e)\in {\rm H}^{1}(\omega,\mathbb{R}^3)\times { {{\rm H}^{1}}}(\omega,\text{SO(3)})\,\big|\,\, \varphi|_{\partial \omega}(\eta_1,\eta_2)=\varphi^\natural_d(\eta_1,\eta_2,0)\}\,.
\end{align}
By the  {embedding} theorem (\cite{ciarlet1987mathematical}, Theorem 6.1-3), the  {embedding} $X'\subset X$ is true and clearly\footnote{Since $\infty>\int_{\omega}|\varphi|^2 \,dx\,dy=\int_{\omega}\int_{-1/2}^{1/2}|\varphi|^2 \,dz\, dx\,dy=\int_{\Omega_1}|\varphi|^2\, dV$, which means any element from $X_\omega$, belongs to $X$ as well.} $X_\omega\subset X$, $X'_\omega\subset X'$.

The functionals in our analysis are obtained by extending the functionals $J_h$ (respectively $I_h$) to the entire space $X$ and to take their averages over the thickness,  through
\begin{align}\label{energyinnewarea 1}
\mathcal{I}_h^\natural(\varphi^\natural,  \nabla^h_\eta \varphi^\natural,\overline{Q}_e^\natural,\Gamma_h^\natural) &=
\begin{cases}
\displaystyle\frac{1}{h}\,I_h^\natural(\varphi^\natural,  \nabla^h_\eta \varphi^\natural, \overline{Q}_e^\natural,\Gamma_h^\natural)\quad &\text{if}\;\; (\varphi^\natural,\overline{Q}_e^\natural)\in \mathcal{S}',
\\
+\infty\qquad &\text{else in}\; X.
\end{cases}\\
&=
\begin{cases}
\displaystyle\frac{1}{h}\,J_h^\natural(\varphi^\natural,  \nabla^h_\eta \varphi^\natural, \overline{Q}_e^\natural,\Gamma_h^\natural)-\frac{1}{h} {\Pi}^\natural_h(\varphi^\natural,\overline{Q}_e^\natural)\quad &\text{if}\;\; (\varphi^\natural,\overline{Q}_e^\natural)\in \mathcal{S}',
\\
+\infty\qquad &\text{else in}\; X.\notag
\end{cases}
\end{align}

The main aim of the current paper is to find the $\Gamma$-limit of the family of functional $\mathcal{I}_h^\natural(\varphi^\natural,  \nabla^h_\eta \varphi^\natural,\overline{Q}_e^\natural,\Gamma_h^\natural)$, i.e., to obtain an energy functional expressed only in terms of the weak limit of a  subsequence of  $(\varphi_{h_j}^\natural,\overline{Q}_{e,h_j}^\natural)\in X$, when $h_j$ goes to zero. In other words, as we will see, to construct an energy function depending only on quantities defined on the midsurface of the shell-like domain, see Figure \ref{Fig3} .

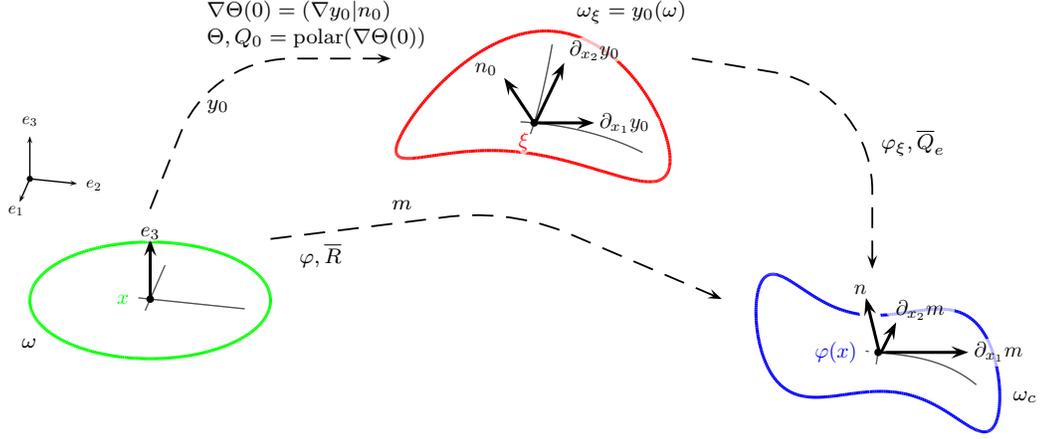
\begin{figure}[h!]
	\hspace*{-2cm}
	\begin{center}
		\includegraphics[scale=1.6]{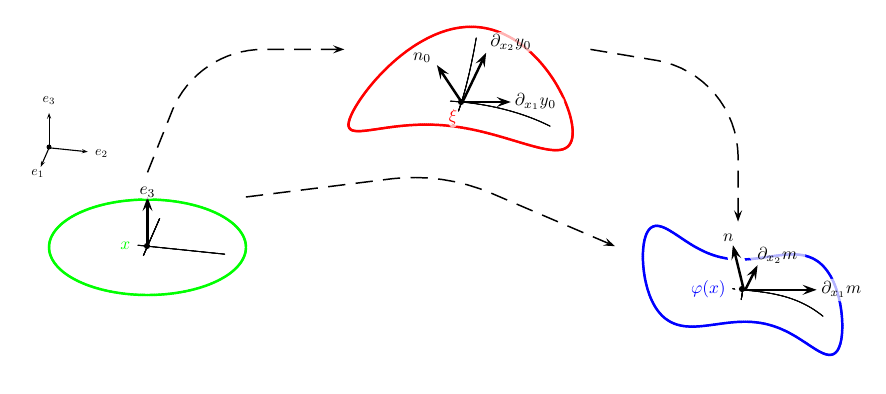}
		\put(-320,165){\footnotesize{$\Theta  ,  Q_0={\rm polar}(\nabla \Theta(0))$}} 
		\put(-320,175){\footnotesize{$\nabla\Theta(0)=(\nabla y_0|n_0)$}}
		\put(-320,140){\footnotesize{$y_0$}} 
		\put(-285,82){\footnotesize{$\varphi  , \overline{ R}$}}
		\put(-250,102){\footnotesize{$m$}}
		\put(-65,125){\footnotesize{$\varphi_\xi,\overline{ Q}_e $}}
		\put(-390,50){\footnotesize{$\omega$}}  
		\put(-15,30){\footnotesize{$\omega_c$}}  
		\put(-180,175){\footnotesize{$\omega_\xi=y_0(\omega)$}}  
		\caption{\footnotesize Kinematics of the dimensionally reduced Cosserat shell model. All fields are referred to two-dimensional surfaces. The geometry of the curved surface $\omega_\xi$ is fully encoded by the map $\Theta$. Instead of the elastic deformation starting from $\omega_\xi$, the total deformation $m$ from the fictitious flat midsurface $\omega$ is considered, likewise for the total rotation $\overline{ R}$.}
		\label{Fig3}       
	\end{center}
\end{figure}
As a first step we consider 
the functionals
\begin{align}\label{energyJ}
\mathcal{J}_h^\natural(\varphi^\natural,  \nabla^h_\eta \varphi^\natural,\overline{Q}_e^\natural,\Gamma_h^\natural) 
&=
\begin{cases}
\displaystyle\frac{1}{h}\,J_h^\natural(\varphi^\natural,  \nabla^h_\eta \varphi^\natural, \overline{Q}_e^\natural,\Gamma_h^\natural)\quad &\text{if}\;\; (\varphi^\natural,\overline{Q}_e^\natural)\in \mathcal{S}',
\\
+\infty\qquad &\text{else in}\; X.
\end{cases}
\end{align}

\subsection{Equi-coercivity and compactness of  the family $\mathcal{J}_h^\natural$ }\label{equi-comp}

\begin{theorem}
	Assume  that the initial configuration  is defined by  a continuous injective mapping $\,y_0:\omega\subset\mathbb{R}^2\rightarrow\mathbb{R}^3$  which admits an extension to $\overline{\omega}$ into  $C^2(\overline{\omega};\mathbb{R}^3)$ such that $\det[\nabla_x\Theta(0)] \geq\, a_0 >0$ on $\overline{\omega}$, where $a_0$ is a positive constant, and  assume that the boundary data satisfies the conditions
	\begin{equation}\label{25}
	\varphi^\natural_d=\varphi_d\big|_{\Gamma_1} \text{(in the sense of traces) for} \ \varphi_d\in {\rm H}^1(\Omega_1,\mathbb{R}^3).
	\end{equation}
	Consider a sequence $(\varphi_{h_j}^\natural,\overline{Q}_{e,h_j}^\natural)\in X$, such that the energy functionals $\mathcal{J}^\natural_{h_j}(\varphi_{h_j}^\natural,\overline{Q}_{e,h_j}^\natural)$ are bounded as $h_j\to 0$. Let the constitutive parameters satisfy 
	\begin{align}
	\mu\,>0, \quad\quad\quad \kappa>0, \quad\quad\quad \mu_{\rm c}> 0,\quad\quad\quad a_1>0,\quad\quad\quad a_2>0,\quad\quad\quad a_3> 0.
	\end{align} 
	Then  the sequence $(\varphi_{h_j}^\natural,\overline{Q}_{e,h_j}^\natural)$ admits a subsequence which is weakly convergent to $(\varphi_{0}^\natural,\overline{Q}_{e,0}^\natural)\in X_\omega$.
\end{theorem}
\begin{proof} Consider the sequence $(\varphi_{h_j}^\natural,\overline{Q}_{e,h_j}^\natural)\in X$, such that the energy functionals $\mathcal{J}^\natural_{h_j}(\varphi_{h_j}^\natural,\overline{Q}_{e,h_j}^\natural)$ are bounded as $h_j\to 0$. Obviously this implies that $(\varphi_{h_j}^\natural,\overline{Q}_{e,h_j}^\natural)\in \mathcal{S}'$ for all $h_j$. 
 We have
		\begin{align}\label{newequi1}
		 2\Big(\norm{\UJN-\id_3}^2+\norm{\id_3}^2\Big)&\geq (\norm{\UJN-\id_3}+\norm{\id_3})^2\geq \norm{\UJN}^2= \norm{\overline{Q}_e^{\natural,T}\nabla_\eta^{h_j}\varphi_{h_j}^\natural[(\nabla_{x}\Theta)^\natural(\eta)]^{-1}}^2\\
		 \nonumber&=\iprod{\overline{Q}_e^{\natural,T}\nabla_\eta^{h_j}\varphi_{h_j}^\natural[(\nabla_{x}\Theta)^\natural(\eta)]^{-1},\overline{Q}_e^{\natural,T}\nabla_\eta^{h_j}\varphi_{h_j}^\natural[(\nabla_{x}\Theta)^\natural(\eta)]^{-1}}=\lVert\nabla_\eta^{h_j}\varphi_{h_j}^\natural[(\nabla_{x}\Theta)^\natural(\eta)]^{-1}\rVert^2.
		\end{align}
	Thus, we deduce with (\ref{newequi1})$_1$
	\begin{align}\label{newequi2} \norm{\UJN-\id_3}^2&\geq\frac{1}{2}\lVert\nabla_\eta^{h_j}\varphi_{h_j}^\natural[(\nabla_{x}\Theta)^\natural(\eta)]^{-1}\rVert^2-3.
\end{align}
But
\begin{align}
 \lVert\nabla_\eta^{h_j}\varphi_{h_j}^\natural\rVert=\lVert\nabla_\eta^{h_j}\varphi_{h_j}^\natural[(\nabla_{x}\Theta)^\natural(\eta)]^{-1}[(\nabla_{x}\Theta)^\natural(\eta)]\rVert
 \leq\lVert\nabla_\eta^{h_j}\varphi_{h_j}^\natural[(\nabla_{x}\Theta)^\natural(\eta)]^{-1}\rVert\cdot\lVert(\nabla_{x}\Theta)^\natural(\eta)]\rVert
\end{align}
and we obtain
\begin{align}\label{alt}
\lVert\nabla_\eta^{h_j}\varphi_{h_j}^\natural[(\nabla_{x}\Theta)^\natural(\eta)]^{-1}\rVert\geq\lVert\nabla_\eta^{h_j}\varphi_{h_j}^\natural\rVert \frac{1}{\lVert(\nabla_{x}\Theta)^\natural(\eta)\rVert}.
\end{align}
From the formula $[(\nabla_x \Theta)^\natural(\eta)]\,=\,(\nabla y_0|n_0)+h_j\eta_3(\nabla n_0|0)$ 
we get
\begin{align}
\notag\lVert(\nabla_{x}\Theta)^\natural(\eta)\rVert&\leq \lVert(\nabla y_0|n_0)\|+h_j\, |\eta_3|\,\lVert(\nabla n_0|0) \rVert\leq \lVert(\nabla y_0|n_0)\|+h_j\, \lVert(\nabla n_0|0) \rVert\\&< \,\lVert(\nabla y_0|n_0)\|+ \lVert(\nabla n_0|0) \rVert.
\end{align}
since $h_j\ll 1$.
Thus
\begin{align}
\frac{1}{\lVert(\nabla_{x}\Theta)^\natural(\eta)\rVert}\geq \frac{1}{\lVert(\nabla y_0|n_0)\|+\lVert(\nabla n_0|0) \rVert}.
\end{align}
Moreover, since 
${y}_0\in C^2(\overline{\omega};\mathbb{R}^3)$, it follows  for $h_j$ small enough that there exists $c_1>0$ such that
$
\frac{1}{\lVert(\nabla_{x}\Theta)^\natural(\eta)\rVert}\geq c_1.
$
Therefore, from \eqref{newequi2} and \eqref{alt}, we get that there exist
$c_1,c_2>0$ such that\\
\begin{align}\label{newequi3} \norm{\UJN-\id_3}^2&\geq\frac{c_1}{2}\lVert\nabla_\eta^{h_j}\varphi_{h_j}^\natural\rVert^2-c_2.
\end{align}
From the hypothesis we have
		\begin{align}\label{comp1}
		 \infty>\mathcal{J}^\natural_{h_j}(\varphi_{h_j}^\natural,\overline{Q}_{e,h_j}^\natural)&\geq\int_{\Omega_1}\Big( {W}_{\rm mp}(\UJN)+\widetilde{W}_{\rm curv}(\Gamma^\natural_{h_j})\Big)\det((\nabla_x\Theta)^\natural)\,dV_\eta \\&\geq \int_{\Omega_1} {W}_{\rm mp}(\UJN)\det((\nabla_x\Theta)^\natural)\,dV_\eta\geq \min(c_1^+\,,\mu_{\rm c}\,)\int_{\Omega_1} \norm{\UJN-\id_3}^2 \det((\nabla_x\Theta)^\natural)dV_\eta,\notag
		\end{align}
		where $c_1^+$  denotes the  smallest  eigenvalue of the quadratic form ${W}_{\mathrm{mp}}^{\infty}( X)$.
		
		Let us recall that ${\rm det}(\nabla_x \Theta(x_3))\,=\,{\rm det}(\nabla y_0|n_0)\Big[1-2\,x_3\,{\rm H}\,+x_3^2 \, {\rm K}\Big]=\,{\rm det}(\nabla y_0|n_0)(1-\kappa_1\,x_3)(1-\kappa_2\, x_3)$, where $\text{H {,} K}$ are the mean curvature and Gauß curvature, respectively. But $(1-\kappa_1\,x_3)(1-\kappa_2\, x_3)>0$, $\forall \, x_3\in \left[-\nicefrac{h_j}{2},\nicefrac{h_j}{2}\right]$ if and only if $h_j$ satisfies the hypothesis  {\eqref{ch5in}}. Therefore, there exists a constant $c>0$ such that\,
		\begin{align}
		{\rm det}(\nabla_x \Theta(x_3))\,\geq c\,{\rm det}(\nabla y_0|n_0)\quad  \forall \ x_3\in \left[-\nicefrac{h}{2},\nicefrac{h}{2}\right].
		\end{align}Due to  the hypothesis $\det[\nabla_x\Theta(0)] \geq\, a_0 >0 $ this implies that 
		there exists a constant $c>0$ such that\,
		\begin{align}\label{detb}
		{\rm det}(\nabla_x \Theta(x_3))\,\geq c\quad  \forall \ x_3\in \left[-\nicefrac{h_j}{2},\nicefrac{h_j}{2}\right],
		\end{align}
	which means that $	{\rm det}(\nabla_x \Theta(x_3)^\natural)\,\geq c\quad  \forall \ x_3\in \left[-\nicefrac{1}{2},\nicefrac{1}{2}\right]$.
	
		Hence, from \eqref{comp1}, \eqref{newequi3} and \eqref{detb}, it follows that for small enough $h_j$ there exist constants $c_1>0$ and $c_2>0$ such that
		\begin{align}\label{bi}
		\infty>\mathcal{J}^\natural_{h_j}(\varphi_{h_j}^\natural,\overline{Q}_{e,h_j}^\natural)&\geq c_1\,\int_{\Omega_1} \lVert\nabla_\eta^{h_j}\varphi_{h_j}^\natural\rVert^2 dV_\eta-c_2\\ &\geq c_1\,\int_{\Omega_1}\left(\norm{\partial_{\eta_1}\varphi_{h_j}^\natural}^2+\norm{\partial_{\eta_2}\varphi_{h_j}^\natural}^2+\frac{1}{h_j^2}\norm{\partial_{\eta_3}\varphi_{h_j}^\natural}^2\right)dV_\eta-c_2.	\notag
		\end{align}
	%
		Furthermore, due to the hypothesis on $h_j$, it is clear that there exists $c>0$ such that
		\begin{align}
		\norm{\partial_{\eta_1}\varphi_{h_j}^\natural}^2+\norm{\partial_{\eta_2}\varphi_{h_j}^\natural}^2+\frac{1}{h_j^2}\norm{\partial_{\eta_3}\varphi_{h_j}^\natural}^2\geq \,c\, \left(\norm{\partial_{\eta_1}\varphi_{h_j}^\natural}^2+\norm{\partial_{\eta_2}\varphi_{h_j}^\natural}^2+\norm{\partial_{\eta_3}\varphi_{h_j}^\natural}^2\right),
		\end{align}
		which implies the existence of $c_1,c_2>0$ such that
		\begin{align}\label{bi2}
		\infty>\mathcal{J}^\natural_{h_j}(\varphi_{h_j}^\natural,\overline{Q}_{e,h_j}^\natural)\geq c_1\,\int_{\Omega_1}\underbrace{\left(\norm{\partial_{\eta_1}\varphi_{h_j}^\natural}^2+\norm{\partial_{\eta_2}\varphi_{h_j}^\natural}^2+\norm{\partial_{\eta_3}\varphi_{h_j}^\natural}^2\right)}_{=:\|\nabla_\eta^{h_j} \varphi_{h_j}^\natural\|^2}dV_\eta-c_2.	
		\end{align}
		We also obtain, applying the Poincar\'e--inequality \cite{neff2015poincare}, that  there exists  a constant $C>0$ such that
		\begin{align}\label{pin}
		\nonumber\lVert \nabla_\eta^{h_j} \varphi_{h_j}^\natural\rVert ^2_{{\rm L}^2(\omega)}&=\norm{\nabla_\eta^{h_j} \varphi_{h_j}^\natural-\nabla_\eta^{h_j} \varphi_d+\nabla_\eta^{h_j} \varphi_d}_{{\rm L}^2(\omega)}^2\\
		\nonumber&\geq\, (\lVert  \nabla_\eta^{h_j} (\varphi_{h_j}^\natural-\varphi_d)\rVert_{{\rm L}^2(\omega)}-\lVert  \nabla_\eta^{h_j} \varphi_d\rVert_{{\rm L}^2(\omega)})^2\\
		&=\norm{\nabla_\eta^{h_j} (\varphi_{h_j}^\natural-\varphi_d)}_{{\rm L}^2(\omega)}^2-2\norm{\nabla_\eta^{h_j} (\varphi_{h_j}^\natural-\varphi_d)}_{{\rm L}^2(\omega)}\norm{\nabla_\eta^{h_j} \varphi_d}_{{\rm L}^2(\omega)}+\norm{\nabla_\eta^{h_j} \varphi_d}_{{\rm L}^2(\omega)}^2\\
		\nonumber&\geq\, C\,\lVert  \varphi_{h_j}^\natural-\varphi_d\rVert _{{\rm H}^1(\omega)}^2-2\,\lVert  \varphi_{h_j}^\natural-\varphi_d\rVert _{{\rm H}^1(\omega)} \lVert  \nabla_\eta^{h_j} \varphi_d\rVert_{{\rm L}^2(\omega)}+\lVert  \nabla_\eta^{h_j} \varphi_d\rVert_{{\rm L}^2(\omega)}^2 \\
		\nonumber&\geq\, C\,\lVert  \varphi_{h_j}^\natural-\varphi_d\rVert _{{\rm H}^1(\omega)}^2-\frac{1}{\varepsilon}\,\lVert  \varphi_{h_j}^\natural-\varphi_d\rVert ^2_{{\rm H}^1(\omega)}-\varepsilon \lVert  \nabla_\eta^{h_j} \varphi_d\rVert ^2_{{\rm L}^2(\omega)}+\lVert  \nabla_\eta^{h_j} \varphi_d\rVert_{{\rm L}^2(\omega)}^2 \,\qquad \forall \, \varepsilon>0,
		\end{align}
		where we have used Young's and Poincar\'{e}'s inequality. Therefore, by choosing $\varepsilon>0$ small enough,  \eqref{pin} ensures the existence of  constants $c_1>0$ and $c {_2}\in\mathbb{R}$  such that 
		\begin{align}\label{36}
	\nonumber\lVert  \nabla_\eta^{h_j} \varphi_{h_j}^\natural\rVert ^2_{{\rm L}^2(\omega)}&\geq\, c_1\lVert \varphi_{h_j}^\natural-\varphi_d\rVert_{{\rm H}^1(\omega)}^2 -c_2\geq\, \frac{c_1}{2}2\,(\lVert  \varphi_{h_j}^\natural\rVert_{{\rm H}^1(\omega)}-\lVert \varphi_d\rVert_{{\rm H}^1(\omega)})^2 -c_2\\&\geq\, \frac{c_1}{2}\lVert  \varphi_{h_j}^\natural\rVert_{{\rm H}^1(\omega)}^2+\frac{c_1}{2}\lVert \varphi_d\rVert_{{\rm H}^1(\omega)}^2 -c_2.
		\end{align}
	Thus, there exists $c_1>0$ and $c {_2}\in\mathbb{R}$  such that
		\begin{align}\label{36b}
	\lVert  \nabla_\eta^{h_j} \varphi_{h_j}^\natural\rVert ^2_{{\rm L}^2(\omega)}&\geq\, \frac{c_1}{2}\lVert  \varphi_{h_j}^\natural\rVert_{{\rm H}^1(\omega)}^2-c_2,
	\end{align}
	which implies the  uniform bound for $\varphi_{h_j}^\natural$ in $\mathcal{S}'$.
		On the other hand, since
	\begin{align}
	\norm{\partial_{\eta_1}\varphi_{h_j}^\natural}^2+\norm{\partial_{\eta_2}\varphi_{h_j}^\natural}^2+\frac{1}{h_j^2}\norm{\partial_{\eta_3}\varphi_{h_j}^\natural}^2\geq \frac{1}{h_j^2}\norm{\partial_{\eta_3}\varphi_{h_j}^\natural}^2,
	\end{align}
	from \eqref{bi} it results that $\frac{1}{h_j^2}\norm{\partial_{\eta_3}\varphi_{h_j}^\natural}^2$ is bounded, i.e., there is $c>0$, such that
	\begin{align}
	\norm{\partial_{\eta_3}\varphi_{h_j}^\natural}_{{\rm L}^2(\Omega)}\leq c\,h_j.
	\end{align}
	This means that $\partial_{\eta_3}\varphi_{h_j}^\natural\to 0$ strongly in ${\rm L}^2(\Omega)$, when $h_j\to 0$.
	
	Hence, considering  $(\varphi_{h_j}^\natural,\overline{Q}_{e,h_j}^\natural)\in X$, such that the energy functionals $\mathcal{J}^\natural_{h_j}(\varphi_{h_j}^\natural,\overline{Q}_{e,h_j}^\natural)$ are bounded, it follows that  any limit point  $\varphi_{0}^\natural$  of $\varphi_{h_j}^\natural$ for the weak topology of ${\rm L}^2(\Omega_1, \mathbb{R}^3)$ (which exists due to its uniform boundedness in ${\rm H}^{1}(\omega,\mathbb{R}^3)$) satisfies 
		\begin{align}\label{phi0}
	\partial_{\eta_3}\varphi_{0}^\natural=0 \quad \Rightarrow\quad \varphi_{0}^\natural\in {\rm H}^{1}(\omega,\mathbb{R}^3).
	\end{align}

		Similar arguments for the curvature energy imply that there exists $c>0$ such that
		\begin{align}
		 \infty >&\, \mathcal{J}^\natural_{h_j}(\varphi_{h_j}^\natural,\overline{Q}_{e,h_j}^\natural)\geq \int_{\Omega_1}\widetilde{W}_{\rm curv}(\Gamma^\natural_{h_j})\det((\nabla_x\Theta)^\natural)\,dV_\eta\geq \int_{\Omega_1}c\,\; \norm{\Gamma^\natural_{h_j}}^{2}\det((\nabla_x\Theta)^\natural)\,dV_\eta\\
	\nonumber	=&\,c\int_{\Omega_1}\dynnorm{\Big(\mathrm{axl}(\overline{Q}_{e,h_j}^{\natural,T}\,\partial_{\eta_1} \overline{Q}_{e,h_j}^\natural)\,\Big|\, \mathrm{axl}(\overline{Q}_{e,h_j}^{\natural,T}\,\partial_{\eta_2} \overline{Q}_{e,h_j}^\natural)\,\Big|\,\frac{1}{h_j}\mathrm{axl}(\overline{Q}_{e,h_j}^{\natural,T}\,\partial_{\eta_3} \overline{Q}_{e,h_j}^\natural)\,\Big)[(\nabla_x\Theta)^\natural(\eta) ]^{-1}}^{2}\det((\nabla_x\Theta)^\natural)\,dV_\eta\,.
	\end{align}
In the next step, as in the deduction of \eqref{newequi1}--\eqref{bi}, it will be shown that for $a_1,a_2,a_3>0$ there exists $c>0$ such that 
		\begin{align}\label{gradQ}
		\nonumber \infty &>  c\int_{\Omega_1}\left(\norm{\mathrm{axl}(\overline{Q}_{e,h_j}^{\natural,T}\,\partial_{\eta_1} \overline{Q}_{e,h_j}^\natural)}^2+\norm{\mathrm{axl}(\overline{Q}_{e,h_j}^{\natural,T}\,\partial_{\eta_2} \overline{Q}_{e,h_j}^\natural)}^2+\frac{1}{h_j^2}\norm{\mathrm{axl}(\overline{Q}_{e,h_j}^{\natural,T}\,\partial_{\eta_3} \overline{Q}_{e,h_j}^\natural)}^2\right)\,dV_\eta\,\\
		&=   {\frac{c}{2}}\int_{\Omega_1}\left(\norm{\overline{Q}_{e,h_j}^{\natural,T}\,\partial_{\eta_1} \overline{Q}_{e,h_j}^\natural}^2+\norm{\overline{Q}_{e,h_j}^{\natural,T}\,\partial_{\eta_2} \overline{Q}_{e,h_j}^\natural}^2+\frac{1}{h_j^2}\norm{\overline{Q}_{e,h_j}^{\natural,T}\,\partial_{\eta_3} \overline{Q}_{e,h_j}^\natural}^2\right)\,dV_\eta\,
		\\
		&=   {\frac{c}{2}}\int_{\Omega_1}\left(\norm{\partial_{\eta_1} \overline{Q}_{e,h_j}^\natural}^2+\norm{\partial_{\eta_2} \overline{Q}_{e,h_j}^\natural}^2+\frac{1}{h_j^2}\norm{\partial_{\eta_3} \overline{Q}_{e,h_j}^\natural}^2\right)\,dV_\eta\,.\notag
 		\end{align}
 		
	With the same argument as in the strain part,  we deduce
		\begin{align}
	 \infty &>   c\int_{\Omega_1}\left(\norm{\partial_{\eta_1} \overline{Q}_{e,h_j}^\natural}^2+\norm{\partial_{\eta_2} \overline{Q}_{e,h_j}^\natural}^2+\norm{\partial_{\eta_3} \overline{Q}_{e,h_j}^\natural}^2\right)\,dV_\eta\,,
	\end{align}
	where $c>0$. 
	Hence, it follows  that
	$  \partial_{\eta_i} \overline{Q}_{e,h_j}^\natural$ is bounded in ${\rm L}^2(\Omega_1,\mathbb{R}^{3\times 3})$, for $i=1,2,3$. Since $\overline{Q}_{e,h_j}^\natural\in {\rm SO}(3)$, we have $\lVert \overline{Q}_{e,h_j}^\natural\rVert ^2=3$ and therefore $\overline{Q}_{e,h_j}^\natural$ is bounded in ${\rm L}^2(\Omega_1,\mathbb{R}^{3\times 3})$. Hence, we can infer that the sequence $ \overline{Q}_{e,h_j}^\natural$
	is bounded in ${ {{\rm H}^{1}}}(\Omega_1,\SO(3))$, independently from $h_j$.

	Therefore, there is a subsequence from $\overline{Q}_{e,h_j}^\natural $ which is weakly convergent (without relabeling) to $\overline{Q}_{e,0}^\natural$. That is
	\begin{align} \overline{Q}_{e,h_j}^\natural \rightharpoonup \overline{Q}_{e,0}^\natural\quad\text{in} \quad { {{\rm H}^{1}}}(\Omega_1,\SO(3))\,.
	\end{align}
	In addition, from \eqref{gradQ}, we also obtain that there exists $c>0$ such that $	c\,h_j> \norm{\partial_{\eta_3} \overline{Q}_{e,h_j}^\natural}_{{\rm L}^2(\Omega_1,\SO(3))}$. This means that $\partial_{\eta_3} \overline{Q}_{e,h_j}^\natural\to 0$ strongly in ${\rm L}^2(\Omega_1 ,\SO(3))$, when $h_j\to 0$.
	 Hence, considering  $(\varphi_{h_j}^\natural,\overline{Q}_{e,h_j}^\natural)\in X$, such that the energy functionals $\mathcal{J}^\natural_{h_j}(\varphi_{h_j}^\natural,\overline{Q}_{e,h_j}^\natural)$ are bounded, it follows that  any limit point  $\overline{Q}_{e,0}^\natural$  of $\overline{Q}_{e,h_j}^\natural$ for the weak topology of $X$  satisfies 
	\begin{align}\label{Q0}
	\partial_{\eta_3}\overline{Q}_{e,0}^\natural =0 \quad \Rightarrow\quad \overline{Q}_{e,0}^\natural\in { {{\rm H}^{1}}}(\omega,\SO(3)).
	\end{align}
	From \eqref{phi0}, \eqref{Q0} and due to the continuity of the trace operator we obtain that  considering  $(\varphi_{h_j}^\natural,\overline{Q}_{e,h_j}^\natural)\in X$, such that the energy functionals $\mathcal{J}^\natural_{h_j}(\varphi_{h_j}^\natural,\overline{Q}_{e,h_j}^\natural)$ are bounded, it follows that  any limit point  $(\varphi_{0}^\natural,\overline{Q}_{e,0}^\natural)$   for the weak topology of $X$  belongs to $\mathcal{S}'_\omega$ (since actually, such a sequence belongs to $\mathcal{S}'$).	
\end{proof}

Since the embedding $X'\subset X$ is compact, it follows that the set of the sequence of energies due to the scaling is a subset of $X'$, and hence, we have obtained that the family of energy functionals 
$
J^\natural_h$ is equi-coercive with respect to $X$.

\section{The construction of the $\Gamma$-limit $J_0$ of the rescaled energies}\label{rescaled}

In this section we construct the $\Gamma$-limit of the rescaled energies
\begin{align}\label{energyinnewarea}
\mathcal{J}_h^\natural(\varphi^\natural,  \nabla^h_\eta \varphi^\natural,\overline{Q}_e^\natural,\Gamma_h^\natural) 
&=
\begin{cases}
\displaystyle\frac{1}{h}\,J_h^\natural(\varphi^\natural,  \nabla^h_\eta \varphi^\natural, \overline{Q}_e^\natural,\Gamma_h^\natural)\quad &\text{if}\;\; (\varphi^\natural,\overline{Q}_e^\natural)\in \mathcal{S}',
\\
+\infty\qquad &\text{else in}\; X,
\end{cases}
\end{align}
with 
\begin{align}
J_h^\natural(\varphi^\natural,  \nabla^h_\eta \varphi^\natural, \overline{Q}_e^\natural,\Gamma_h^\natural)=\int_{\Omega_1}\; h \;\Big[\Big({{W}}_{\rm mp}(\UHN)+\widetilde{W}_{\rm curv}(\Gamma^\natural_h)\Big)\det((\nabla_x\Theta)^\natural )\Big]\;dV_\eta.
\end{align}

\subsection{Auxiliary optimization problem}\label{auxi}
For  $\varphi^\natural:\Omega_1\rightarrow \mathbb{R}^3$ and $\overline{Q}_{e}^\natural:\Omega_1\rightarrow{\rm SO}(3)$ we associate the non fully dimensional reduced elastic shell stretch  tensor
\begin{align}
\overline{U}_{\varphi^\natural,\overline{Q}_{e}^{\natural}}:=\overline{Q}_{e}^{\natural, T}(\nabla_{(\eta_1,\eta_2)} \varphi^\natural|0)[(\nabla_x\Theta)^\natural ]^{-1}\,,
\end{align}
and the non fully dimensional reduced elastic shell strain  tensor 
\begin{align}
\mathcal{E}_{\varphi^\natural,\overline{Q}_{e}^{\natural} }:=(\overline{Q}_{e}^{\natural, T}\nabla_{(\eta_1,\eta_2)} \varphi^\natural-(\nabla y_0)^\natural |0)[(\nabla_x\Theta)^\natural ]^{-1}=\overline{U}_{\varphi^\natural,\overline{Q}_{e}^{\natural}}-((\nabla y_0)^\natural|0)[(\nabla_x\Theta)^\natural ]^{-1}\,.
\end{align} 
Here, "non-fully" means that the introduced quantities still depend on $\eta_3$ and $h$, because the elements $\nabla_{(\eta_1,\eta_2)}\varphi^\natural$ still depend on $\eta_3$ and $\overline{ Q}^{\natural , T}$ depends on $h$.\\
For reaching our goal we need to solve the following optimization problem: for  $\varphi^\natural:\Omega_1\rightarrow \mathbb{R}^3$ and $\overline{Q}_{e}^\natural:\Omega_1\rightarrow{\rm SO}(3)$, we determine a vector $d^*\in \R^3$ through
\begin{align}\label{hom inf}
 {W}_{\rm mp}^{{\rm hom},\natural}(  \mathcal{E}_{\varphi^\natural,\overline{Q}_{e}^{\natural} })= {W}_{\rm mp}\Big(\overline{Q}_{e}^{\natural,T}(\nabla_{(\eta_1,\eta_2)} \varphi^\natural|d^*)[(\nabla_x\Theta)^\natural ]^{-1}\Big):=\inf_{c\in \mathbb{R}^3} {W}_{\rm mp}\Big(\overline{Q}_{e}^{\natural,T}(\nabla_{(\eta_1,\eta_2)} \varphi^\natural|c)[(\nabla_x\Theta)^\natural ]^{-1}\Big).
\end{align}
The motivation for this optimization problem is to minimize the effect of the derivative in the $\eta_3$-direction in the local energy $ {W}_{\rm mp}$.  Due to the coercivity and continuity of the energy $ {W}_{\rm mp}$, it is clear that this function is well defined and the infimum is attained. Note that  $\varphi^\natural$ and  $\overline{Q}_{e}^\natural$ depend on $\eta_3$ and $h$. Hence $ {W}_{\rm mp}^{{\rm hom},\natural}(  \mathcal{E}_{\varphi^\natural,\overline{Q}_{e}^{\natural} })$ depends on $\eta_3$ and $h$.
While it is not immediately clear why $ {W}_{\rm mp}\Big(\overline{Q}_{e}^{\natural,T}(\nabla_{(\eta_1,\eta_2)} \varphi^\natural|d^*)[(\nabla_x\Theta)^\natural ]^{-1}\Big)$ can be expressed as a function of $ \mathcal{E}_{\varphi^\natural,\overline{Q}_{e}^{\natural} }$, this aspect will be clarified in the rest of this subsection.

 {
We do some lengthy but straightforward calculations in Appendix \ref{Calculhom} and after using the fact that \break $[\nabla_x\Theta]^{-T}e_3=n_0$ and $[(\nabla_x\Theta)^\natural]^{-T}e_3=n_0$, as well, we obtain the minimizer $d^*$ from \eqref{hom inf} as
\begin{align}
 d^*&=\Big(1-\frac{\lambda}{2\,\mu\,+\lambda}\iprod{  \mathcal{E}_{\varphi^\natural,\overline{Q}_{e}^{\natural} } ,\id_3}\Big) \overline{Q}_{e}^\natural n_0+\frac{\mu_c\,-\mu\,}{\mu_c\,+\mu\,}\;\overline{Q}_{e}^\natural  \mathcal{E}_{\varphi^\natural,\overline{Q}_{e}^{\natural} }^T n_0.
\end{align}
In terms of $\overline{Q}_{e}^\natural=\overline{R}^\natural Q_0^{\natural,T}$ we obtain the following expression for $d^*$
\begin{align}\label{formula b}
\nonumber d^*&=\Big(1-\frac{\lambda}{2\,\mu\,+\lambda}\iprod{(Q_0^{\natural}\overline{R}^{\natural,T}\nabla_{(\eta_1,\eta_2)} \varphi^\natural-(\nabla y_0)^\natural |0)[(\nabla_x\Theta)^\natural]^{-1},\id_3}\Big) \overline{R}^{\natural}Q_0^{\natural,T}n_0\\
&\qquad+\frac{\mu_c\,-\mu\,}{\mu_c\,+\mu\,}\;\overline{R}^{\natural}Q_0^{\natural,T}\Big((Q_0^{\natural}\overline{R}^{\natural,T}\nabla_{(\eta_1,\eta_2)} \varphi^\natural-(\nabla y_0)^\natural |0)[(\nabla_x\Theta)^\natural]^{-1}\Big)^Tn_0.
\end{align}}

 {Inserting $d^*$ in the strain energy 
$ {{W}}_{\rm mp}(\UHN)=\;\mu\,\norm{\sym (\UHN-\id_3)}^2+\mu_c\,\norm{\skew (\UHN-\id_3)}^2 + \frac{\lambda}{2}[\tr(\sym(\UHN-\id_3))]^2\,
$}
and using (\ref{symter}), (\ref{skew}) and (\ref{trace}), we obtain the explicit form of the homogenized energy for the membrane part
\begin{align}
\nonumber  {W}_{\rm mp}^{{\rm hom},\natural}(  \mathcal{E}_{\varphi^\natural,\overline{Q}_{e}^{\natural} })&=\mu\,\norm{\sym\mathcal{E}_{\varphi^\natural,\overline{Q}_{e}^{\natural} }}^2 +\frac{\mu\,}{2}\frac{(\mu_c\,-\mu\,)^2}{(\mu_c\,+\mu\,)^2}\norm{\mathcal{E}_{\varphi^\natural,\overline{Q}_{e}^{\natural} }^Tn_0}^2+\frac{\mu\,(\mu_c\,-\mu\,)}{(\mu_c\,+\mu\,)}\norm{\mathcal{E}_{\varphi^\natural,\overline{Q}_{e}^{\natural} }^Tn_0}^2\\ &\quad+\frac{\mu\,\lambda^2}{(2\,\mu\,+\lambda)^2}\tr(\mathcal{E}_{\varphi^\natural,\overline{Q}_{e}^{\natural} })^2+\mu_c\,\norm{\skew(\overline{Q}_{e}^{\natural,T}(\nabla_{(\eta_1,\eta_2)} \varphi^\natural|0)[(\nabla_x\Theta)^\natural ]^{-1})}^2\\&\quad +\frac{\mu_c\,}{2}\frac{(\mu_c\,-\mu\,)^2}{(\mu_c\,+\mu\,)^2}\norm{\mathcal{E}_{\varphi^\natural,\overline{Q}_{e}^{\natural} }^Tn_0}^2-\frac{\mu_c\,(\mu_c\,-\mu\,)}{(\mu_c\,+\mu\,)}\norm{\mathcal{E}_{\varphi^\natural,\overline{Q}_{e}^{\natural} }^Tn_0}^2+\frac{2\,\mu\,^2\lambda}{(2\,\mu\,+\lambda)^2}\tr(\mathcal{E}_{\varphi^\natural,\overline{Q}_{e}^{\natural} })^2,\notag
\end{align}
and finally
\begin{align}\label{final hom}
  {W}_{\rm mp}^{{\rm hom},\natural}(  \mathcal{E}_{\varphi^\natural,\overline{Q}_{e}^{\natural} }) &=W_{\mathrm{shell}}(  \mathcal{E}_{\varphi^\natural,\overline{Q}_{e}^{\natural} })   - \frac{(\mu_c\,-\mu\,)^2}{2(\mu_c\,+\mu\,)}\norm{\mathcal{E}_{\varphi^\natural,\overline{Q}_{e}^{\natural} }^Tn_0}^2,
\end{align}
where
\begin{align}
W_{\mathrm{shell}}(  X) & = \mu\,\lVert  \mathrm{sym} \,X\rVert^2 +  \mu_{\rm c}\,\lVert \mathrm{skew} \,X\rVert^2 +\,\dfrac{\lambda\,\mu\,}{\lambda+2\,\mu\,}\,\big[ \mathrm{tr}   \,X\,\big]^2.\notag
\end{align}
Using the orthogonal decomposition in the tangential plane and in the normal direction, gives
\begin{align}\label{orthog decom}
X=X^\parallel+X^\perp, \quad\qquad  X^\parallel\coloneqq{\rm A}_{y_0} \,X, \quad \qquad X^\perp\coloneqq(\id_3-{\rm A}_{y_0}) \,X,
\end{align} and  we deduce that for all $X\,=\,(*|*|0)\cdot[	\nabla_x \Theta(0)]^{-1}$ we have the following split in the expression of the considered quadratic forms
\begin{align}
W_{\mathrm{shell}}(  X) & =\,   \mu\,\lVert  \mathrm{sym}   \,X^\parallel\rVert^2 +  \mu_{\rm c}\,\lVert \mathrm{skew}  \,X^\parallel\rVert^2+\,   \frac{\mu\,+\mu_{\rm c}}{2}\,\lVert X^\perp\rVert^2 +\,\dfrac{\lambda\,\mu\,}{\lambda+2\,\mu\,}\,\big[ \mathrm{tr}   (X)\big]^2.
\end{align}
Moreover, using that for all $X\,=\,(*|*|0)\,[	\nabla_x \Theta(0)]^{-1}$,  it holds that
\begin{align}\label{symA1-A}
\tr(X^\perp)= \tr\big( (\id_3-{\rm A}_{y_0}) X \big) = \tr(X)-\tr({\rm A}_{y_0} X ) = \tr(X)-\tr(X\,{\rm A}_{y_0} )
=0,
\end{align}
we obtain 
\begin{align}\label{expW}
W_{\mathrm{shell}}\big(    \mathcal{E}_{\varphi^\natural,\overline{Q}_{e}^{\natural}} \big)=&\, \mu\,\lVert  \mathrm{sym}\,   \,\mathcal{E}_{\varphi^\natural,\overline{Q}_{e}^{\natural} }^{\parallel}\rVert^2 +  \mu_{\rm c}\,\lVert\mathrm{skew}\,   \,\mathcal{E}_{\varphi^\natural,\overline{Q}_{e}^{\natural} }^{\parallel}\rVert^2 +\,\dfrac{\lambda\,\mu\,}{\lambda+2\,\mu\,}\,\big[ \mathrm{tr}     (\mathcal{E}_{\varphi^\natural,\overline{Q}_{e}^{\natural} }^{\parallel})\big]^2+\frac{\mu\, +  \mu_{\rm c}}{2}\,\lVert \mathcal{E}_{\varphi^\natural,\overline{Q}_{e}^{\natural} }^{\perp}\rVert^2\\
=&\, \mu\,\lVert  \mathrm{sym}\,   \,\mathcal{E}_{\varphi^\natural,\overline{Q}_{e}^{\natural} }^{\parallel}\rVert^2 +  \mu_{\rm c}\,\lVert \mathrm{skew}\,   \,\mathcal{E}_{\varphi^\natural,\overline{Q}_{e}^{\natural} }^{\parallel}\rVert^2 +\,\dfrac{\lambda\,\mu\,}{\lambda+2\,\mu\,}\,\big[ \mathrm{tr}    (\mathcal{E}_{\varphi^\natural,\overline{Q}_{e}^{\natural} }^{\parallel})\big]^2+\frac{\mu\, +  \mu_{\rm c}}{2}\,\lVert \mathcal{E}_{\varphi^\natural,\overline{Q}_{e}^{\natural} }^T\,n_0\rVert^2.\notag
\end{align}

Therefore, the  homogenized energy for the membrane part is
\begin{align}\label{hommpnat}
\nonumber  {W}_{\rm mp}^{{\rm hom},\natural}(  \mathcal{E}_{\varphi^\natural,\overline{Q}_{e}^{\natural} })&=
\, \mu\,\lVert  \mathrm{sym}\,   \,\mathcal{E}_{\varphi^\natural,\overline{Q}_{e}^{\natural} }^{\parallel}\rVert^2 +  \mu_{\rm c}\,\lVert \mathrm{skew}\,   \,\mathcal{E}_{\varphi^\natural,\overline{Q}_{e}^{\natural} }^{\parallel}\rVert^2 +\,\dfrac{\lambda\,\mu\,}{\lambda+2\,\mu\,}\,\big[ \mathrm{tr}    (\mathcal{E}_{\varphi^\natural,\overline{Q}_{e}^{\natural} }^{\parallel})\big]^2\\&\qquad +\frac{\mu\, +  \mu_{\rm c}}{2}\,\lVert \mathcal{E}_{\varphi^\natural,\overline{Q}_{e}^{\natural} }^T\,n_0\rVert^2  - \frac{(\mu_c\,-\mu\,)^2}{2(\mu_c\,+\mu\,)}\norm{\mathcal{E}_{\varphi^\natural,\overline{Q}_{e}^{\natural} }^Tn_0}^2\notag\\
&=
\, \mu\,\lVert  \mathrm{sym}\,   \,\mathcal{E}_{\varphi^\natural,\overline{Q}_{e}^{\natural} }^{\parallel}\rVert^2 +  \mu_{\rm c}\,\lVert \mathrm{skew}\,   \,\mathcal{E}_{\varphi^\natural,\overline{Q}_{e}^{\natural} }^{\parallel}\rVert^2 +\,\dfrac{\lambda\,\mu\,}{\lambda+2\,\mu\,}\,\big[ \mathrm{tr}    (\mathcal{E}_{\varphi^\natural,\overline{Q}_{e}^{\natural} }^{\parallel})\big]^2 +\frac{2\,\mu\, \,  \mu_{\rm c}}{\mu_c\,+\mu\,}\norm{\mathcal{E}_{\varphi^\natural,\overline{Q}_{e}^{\natural} }^Tn_0}^2\\
&=W_{\mathrm{shell}}\big(   \mathcal{E}_{\varphi^\natural,\overline{Q}_{e}^{\natural} }^{\parallel} \big)+\frac{2\,\mu\, \,  \mu_{\rm c}}{\mu_c\,+\mu\,}\lVert \mathcal{E}_{\varphi^\natural,\overline{Q}_{e}^{\natural} }^{\perp}\rVert^2.\notag
\end{align}

\subsection{Homogenized membrane energy}\label{homogen mem}

Now, we will be able to propose the form of the homogenized membrane energy. To each pair $(m, \overline{Q}_{e,0})$, where $m:\omega\to \mathbb{R}^3$,  $\overline{Q}_{e,0}:\omega\to {\rm SO}(3)$, we  associate the \textit{elastic shell strain tensor}
\begin{align}
\mathcal{E}_{m,s} :=(\overline{Q}^{T}_{e,0}\nabla m-\nabla y_0 |0)[\nabla_x\Theta(0)]^{-1}\,,
\end{align} 
and we define the homogenized energy
\begin{align}
 {W}_{\rm mp}^{\rm hom}
(\mathcal{E}_{m,s}):=\inf_{\widetilde{d}\in \mathbb{R}^3} {W}_{\rm mp}\Big(\overline{Q}^{T}_{e,0}(\nabla m|\widetilde{d})[(\nabla_x\Theta)(0)]^{-1}\Big)=\inf_{\widetilde{d}\in \mathbb{R}^3} {W}_{\rm mp}\Big(\mathcal{E}_{m,s}-(0|0|\widetilde{d})[(\nabla_x\Theta)(0)]^{-1}\Big).
\end{align}
Direct calculations as in the previous subsection (\ref{auxi}) show us that the infimum is attained for
\begin{align}\label{formula d}
 \widetilde{d}^*&=\Big(1-\frac{\lambda}{2\,\mu\,+\lambda}\iprod{  \mathcal{E}_{m,s} ,\id_3}\Big) \overline{Q}_{e,0} n_0+\frac{\mu_c\,-\mu\,}{\mu_c\,+\mu\,}\;\overline{Q}_{e,0} \mathcal{E}_{m,s}^T n_0\,,
\end{align}
and
\begin{align}\label{hominf_2D}
 {W}_{\rm mp}^{\rm hom}
(\mathcal{E}_{m,s})&=
\, \mu\,\lVert  \mathrm{sym}\,   \,\mathcal{E}_{m,s}^{\parallel}\rVert^2 +  \mu_{\rm c}\,\lVert \mathrm{skew}\,   \,\mathcal{E}_{m,s}^{\parallel}\rVert^2 +\,\dfrac{\lambda\,\mu\,}{\lambda+2\,\mu\,}\,\big[ \mathrm{tr}    (\mathcal{E}_{m,s}^{\parallel})\big]^2 +\frac{2\,\mu\, \,  \mu_{\rm c}}{\mu_c\,+\mu\,}\norm{\mathcal{E}_{m,s}^Tn_0}^2\\
&=W_{\mathrm{shell}}\big(   \mathcal{E}_{m,s}^{\parallel} \big)+\frac{2\,\mu\, \,  \mu_{\rm c}}{\mu_c\,+\mu\,}\lVert \mathcal{E}_{m,s}^{\perp}\rVert^2,\notag
\end{align}
where 
\begin{align}
W_{\mathrm{shell}}\big(   \mathcal{E}_{m,s}^{\parallel} \big)=
\, \mu\,\lVert  \mathrm{sym}\,   \,\mathcal{E}_{m,s}^{\parallel}\rVert^2 +  \mu_{\rm c}\,\lVert \mathrm{skew}\,   \,\mathcal{E}_{m,s}^{\parallel}\rVert^2 +\,\dfrac{\lambda\,\mu\,}{\lambda+2\,\mu\,}\,\big[ \mathrm{tr}    (\mathcal{E}_{m,s}^{\parallel})\big]^2.
\end{align}

Note that $ {W}_{\rm mp}^{{\rm hom},\natural}
(  \mathcal{E}_{\varphi^\natural,\overline{Q}_{e}^{\natural} })$ constructed in (\ref{hommpnat})  depends  on $\eta_3$ and $h$, while $ {W}_{\rm mp}^{\rm hom}
(\mathcal{E}_{m,s})$ in (\ref{hominf_2D}) does not depend on $\eta_3$ and $h$, since $\overline{Q}_{e,0}$ and $[(\nabla_x\Theta)(0)]$ do not depend on  $\eta_3$ and $h$.

\subsection{Homogenized curvature energy}\label{homogen curv}
We define the homogenized curvature energy as
\begin{align}\label{inf curv}
\nonumber \widetilde{W}^{{\rm hom},\natural}_{\rm curv}( \mathcal{K}_{e}^\natural):&=\widetilde{W}_{\rm curv}\Big(\mathrm{axl}(\overline{Q}_e^{\natural,T}\,\partial_{\eta_1} \overline{Q}_e^\natural)\,|\, \mathrm{axl}(\overline{Q}_e^{\natural,T}\,\partial_{\eta_2} \overline{Q}_e^\natural)\,|\,\axl\.(A^*)\,\Big)[(\nabla_x\Theta)^\natural ]^{-1}\\
&=\inf_{A\in \mathfrak{so}(3)}\widetilde{W}_{\rm curv}\Big(\mathrm{axl}(\overline{Q}_e^{\natural,T}\,\partial_{\eta_1} \overline{Q}_e^\natural)\,|\, \mathrm{axl}(\overline{Q}_e^{\natural,T}\,\partial_{\eta_2} \overline{Q}_e^\natural)\,|\,\axl\.(A)\,\Big)[(\nabla_x\Theta)^\natural ]^{-1},
\end{align}
where
\begin{align}
\mathcal{K}_{e}^\natural & :\,=\,  \Big(\mathrm{axl}(\overline{Q}_{e}^{\natural,T}\,\partial_{\eta_1} \overline{Q}_{e}^\natural)\,|\, \mathrm{axl}(\overline{Q}_{e}^{\natural,T}\,\partial_{\eta_2} \overline{Q}_{e}^\natural)\,|0\Big)[(\nabla_x\Theta)^\natural \,]^{-1}\notag\,,
\end{align}
represents a not fully reduced elastic shell bending-curvature tensor, in the sense that it still depends on $\eta_3$ and $h$, since  $\overline{Q}_{e}^\natural=\overline{Q}_{e}^\natural(\eta_1,\eta_2,\eta_3)$. Therefore, $\widetilde{W}_{\rm curv}^{{\rm hom},\natural}( \mathcal{K}_{e}^\natural)$ given by the above definitions still depends on $\eta_3$ and $h$.

As in the case of the homogenized membrane part in (\ref{boun co}), from which we obtained the unknown $d^*$,  one can explicitly determine the infinitesimal microrotation $A^*\in \mathfrak{so}(3)$ as well. Ghiba et.~al, in \cite{Ghiba2022} obtained the homogenized quadratic curvature energy (see Appendix \ref{homcurghiba}  {for its explicit form}). Presently, it is enough to see that $\widetilde{W}_{\rm curv}^{\rm hom}$ is uniquely defined and has the other requirements like remaining convex in its argument and having the same growth as $\widetilde{W}_{\rm curv}$. Therefore,
\begin{align}\label{last hom curv}
\widetilde{W}_{\rm curv}\Big(\Gamma^\natural_h\Big)\geq \widetilde{W}_{\rm curv}^{{\rm hom},\natural}(\mathcal{K}^\natural _e),
\end{align}
i.e.,
\begin{align}
\widetilde{W}_{\rm curv}\Big(\big(\mathrm{axl}(\overline{Q}_{e,h}^{\natural,T}\,\partial_{\eta_1} \overline{Q}_{e,h}^\natural)\,|\, \mathrm{axl}(\overline{Q}_{e,h}^{\natural,T}\,&\partial_{\eta_2} \overline{Q}_{e,h}^\natural)\,|\,\frac{1}{h}\mathrm{axl}(\overline{Q}_{e,h}^{\natural,T}\,\partial_{\eta_3} \overline{Q}_{e,h}^\natural)\,\big)[(\nabla_x\Theta)^\natural ]^{-1}\Big)\\
\nonumber&\geq \widetilde{W}_{\rm curv}^{{\rm hom},\natural}\Big(\big(\mathrm{axl}(\overline{Q}_e^{\natural,T}\,\partial_{\eta_1} \overline{Q}_e^\natural)\,|\, \mathrm{axl}(\overline{Q}_e^{\natural,T}\,\partial_{\eta_2} \overline{Q}_e^\natural)\,|0\,\big)[(\nabla_x\Theta)^\natural ]^{-1}\Big)\,,
\end{align} 
where this relation will help us in  {S}ubsection \ref{liminf} to show the $\liminf$ condition for the curvature energy. 

In order to construct the $\Gamma$-limit, we have to define a homogenized curvature energy. This energy will be expressed in terms of 
the elastic shell bending-curvature tensor
\begin{align}
\mathcal{K}_{e,s} & :\,=\,  \Big(\mathrm{axl}(\overline{Q}^{T}_{e,0}\,\partial_{x_1} \overline{Q}_{e,0})\,|\, \mathrm{axl}(\overline{Q}^{T}_{e,0}\,\partial_{x_2} \overline{Q}_{e,0})\,|0\Big)[\nabla_x\Theta (0)\,]^{-1}\not\in {\rm Sym}(3) \quad \text{\it\  elastic shell bending--curvature tensor\,,}\notag
\end{align}
which will be defined for any $\overline{Q}_{e,0}:\omega\to {\rm SO}(3)$. 
For   $\overline{Q}_{e,0}:\omega\rightarrow{\rm SO}(3)$, we set
\begin{align} \widetilde{W}^{{\rm hom}}_{\rm curv}(\mathcal{K}_{e,s}):&=\widetilde{W}_{\rm curv}^*\Big(\mathrm{axl}(\overline{Q}^{T}_{e,0}\,\partial_{\eta_1} \overline{Q}_{e,0})\,|\, \mathrm{axl}(\overline{Q}^{T}_{e,0}\,\partial_{\eta_2} \overline{Q}_{e,0})\,|\,\axl\.(A^*)\,\Big)[(\nabla_x \Theta)^\natural(0)]^{-1}\\
&=\inf_{A\in \mathfrak{so}(3)}\widetilde{W}_{\rm curv}\Big(\mathrm{axl}(\overline{Q}^{T}_{e,0}\,\partial_{\eta_1} \overline{Q}_{e,0})\,|\, \mathrm{axl}(\overline{Q}^{T}_{e,0}\,\partial_{\eta_2} \overline{Q}_{e,0})\,|\,\axl\.(A)\,\Big)[(\nabla_x \Theta)^\natural(0)]^{-1}\,.\notag
\end{align}

Again note that while $\widetilde{W}_{\rm curv}^{{\rm hom},\natural}( \mathcal{K}_{e}^\natural)$ (previously constructed) depends  on $\eta_3$ and $h$, $\widetilde{W}^{{\rm hom}}_{\rm curv}(\mathcal{K}_{e,s})$ does not depend on $\eta_3$ and $h$, since $\overline{Q}_{e,0}$ and $[(\nabla_x\Theta)(0)]$ do not depend on  $\eta_3$ and $h$.

\section{$\Gamma$-convergence of $\mathcal{J}_{h_j}$ }\label{result gamma}

We are now ready  to formulate the main result of this paper
\begin{theorem}\label{Theorem homo}
		Assume  that the initial configuration of the curved shell is defined by  a continuous injective mapping $\,y_0:\omega\subset\mathbb{R}^2\rightarrow\mathbb{R}^3$  which admits an extension to $\overline{\omega}$ into  $C^2(\overline{\omega};\mathbb{R}^3)$ such that for $$\Theta(x_1,x_2,x_3)=y_0(x_1,x_2)+x_3\, n_0(x_1,x_2)$$ we have $\det[\nabla_x\Theta(0)] \geq\, a_0 >0$ on $\overline{\omega}$,
	where $a_0$ is a constant, and  assume that the boundary data satisfy the conditions
	\begin{equation}\label{2n5}
	\varphi^\natural_d=\varphi_d\big|_{\Gamma_1} \text{(in the sense of traces) for } \ \varphi_d\in {\rm H}^1(\Omega_1;\mathbb{R}^3).
	\end{equation}
	Let the constitutive parameters satisfy 
	\begin{align}
	\mu\,>0, \qquad\quad \kappa>0, \qquad\quad \mu_{\rm c}> 0,\qquad\quad a_1>0,\quad\quad a_2>0,\qquad\quad a_3>0\,.
	\end{align} 
	Then, 
	for any sequence $(\varphi_{h_j}^\natural, \overline{Q}_{e,h_j}^{\natural})\in X$ such that $(\varphi_{h_j}^\natural, \overline{Q}_{e,h_j}^{\natural})\rightarrow(\varphi_0, \overline{Q}_{e,0})$ as $h_j\to 0$,	the sequence of functionals $\mathcal{J}_{h_j}\col X\rightarrow \overline{\mathbb{R}}$ 
	\textit{ $\Gamma$-converges} to the limit energy  functional $\mathcal{J}_0\col X\rightarrow \overline{\mathbb{R}}$ defined by
	\begin{equation}\label{couple}
	\mathcal{J}_0 (m,\overline{Q}_{e,0}) =
	\begin{cases}\dd
	\int_{\omega} [ {W}_{\rm mp}^{{\rm hom}}(\mathcal{E}_{m,\overline{Q}_{e,0}})+\widetilde{W}_{\rm curv}^{\rm hom}(\mathcal{K}_{e,s})]\det (\nabla y_0|n_0)\; d\omega\quad &\text{if}\quad (m,\overline{Q}_{e,0})\in \mathcal{S}_\omega'\.,
	\\
	+\infty\qquad &\text{else in}\; X,
	\end{cases}
	\end{equation}
	where 
	\begin{align}\label{approxi}
	m(x_1,x_2)&:=\varphi_0 (x_1,x_2)=\lim_{h_j\to 0} \varphi_{h_j}^\natural(x_1,x_2,\frac{1}{h_j}x_3), \qquad \overline{Q}_{e,0}(x_1,x_2)=\lim_{h_j\to 0} \overline{Q}_{e,h_j}^\natural(x_1,x_2,\frac{1}{h_j}x_3),\notag\\
	\mathcal{E}_{m,\overline{Q}_{e,0}}&=(\overline{Q}^{T}_{e,0}\nabla m-\nabla y_0|0)[\nabla_x\Theta(0)]^{-1},\\
	\mathcal{K}_{e,s} &=\Big(\mathrm{axl}(\overline{Q}^{T}_{e,0}\,\partial_{x_1} \overline{Q}_{e,0})\,|\, \mathrm{axl}(\overline{Q}^{T}_{e,0}\,\partial_{x_2} \overline{Q}_{e,0})\,|0\Big)[\nabla_x\Theta (0)\,]^{-1}\not\in {\rm Sym}(3)\,,\notag
	\end{align}
	and
	\begin{align}
	\nonumber {W}_{\rm mp}^{\rm hom}
	(\mathcal{E}_{m,\overline{Q}_{e,0}})&=
	\, \mu\,\lVert  \mathrm{sym}\,   \,\mathcal{E}_{m,\overline{Q}_{e,0} }^{\parallel}\rVert^2 +  \mu_{\rm c}\,\lVert \mathrm{skew}\,   \,\mathcal{E}_{m,\overline{Q}_{e,0} }^{\parallel}\rVert^2 +\,\dfrac{\lambda\,\mu\,}{\lambda+2\,\mu\,}\,\big[ \mathrm{tr}    (\mathcal{E}_{m,\overline{Q}_{e,0} }^{\parallel})\big]^2 +\frac{2\,\mu\, \,  \mu_{\rm c}}{\mu_c\,+\mu\,}\norm{\mathcal{E}_{m,\overline{Q}_{e,0} }^Tn_0}^2\\
	&=W_{\mathrm{shell}}\big(   \mathcal{E}_{m,\overline{Q}_{e,0} }^{\parallel} \big)+\frac{2\,\mu\, \,  \mu_{\rm c}}{\mu_c\,+\mu\,}\lVert \mathcal{E}_{m,\overline{Q}_{e,0} }^{\perp}\rVert^2,\\  \widetilde{W}^{{\rm hom}}_{\rm curv}(\mathcal{K}_{e,s})
	\nonumber&=\inf_{A\in \mathfrak{so}(3)}\widetilde{W}_{\rm curv}\Big(\mathrm{axl}(\overline{Q}^{T}_{e,0}\,\partial_{\eta_1} \overline{Q}_{e,0})\,|\, \mathrm{axl}(\overline{Q}^{T}_{e,0}\,\partial_{\eta_2} \overline{Q}_{e,0})\,|\,\axl(A)\,\Big)[(\nabla_x \Theta)^\natural(0)]^{-1}\\
	\nonumber&=\mu L_c^2\Big(b_1\norm{\sym\mathcal{K}_{e,s}^\parallel}^2+b_2\norm{\skew \mathcal{K}_{e,s}^\parallel}^2+\frac{b_1b_3}{(b_1+b_3)}\tr(\mathcal{K}_{e,s}^\parallel)^2+\frac{2\,b_1b_2}{b_1+b_2}\norm{\mathcal{K}_{e,s}^\perp}\Big)\,.
	\end{align}
\end{theorem}

\begin{proof}The first part of the proof is represented by the proof of equi-coercivity and compactness of the family of  energy functionals which are already done. The rest of the proof  will be divided into two  parts which make the subjects of the following two subsections.\end{proof}

\subsection{ {Step 1 of the proof: } The lim-inf condition}\label{liminf}
In this section we prove the following lemma
\begin{lemma} In the hypothesis of Theorem \ref{Theorem homo}, for any sequence $(\varphi_{h_j}^\natural, \overline{Q}_{e,h_j}^{\natural})\in X$ such that $(\varphi_{h_j}^\natural, \overline{Q}_{e,h_j}^{\natural})\rightarrow(\varphi_0 ^\natural, \overline{Q}_{e,0}^{\natural})$ for $h_j\to 0$, i.e.,
\begin{align}
\varphi_{h_j}^\natural\rightarrow \varphi ^\natural_0 \qquad \text{in}\quad {\rm L}^2(\Omega_1,\mathbb{R}^3), \qquad \overline{Q}_{e,h_j}^{\natural}\rightarrow \overline{Q}_{e,0}^{\natural}\qquad \text{in}\quad {\rm L}^{2}(\Omega_1, \textrm{\rm SO}(3)),
\end{align}
we  have
\begin{align}\label{infcon}
\mathcal{J}_0 (\varphi_0^\natural ,\overline{Q}_{e,0}^\natural)\leq \liminf_{h_j\rightarrow 0}\mathcal{J}^\natural_{h_j}(\varphi_{h_j}^{\natural},\overline{Q}_{e,h_j}^{\natural}).
\end{align}
\end{lemma}
\begin{proof}
It is clear that we may restrict our proof to sequences $(\varphi_{h_j}^\natural, \overline{Q}_{e,h_j}^{\natural})\in \mathcal{S}'\subset X'$, i.e., to sequences in which the functionals $\mathcal{J}^\natural_{h_j}(\varphi_{h_j}^{\natural}, \overline{Q}_{e,h_j}^{\natural})$ are finite, since otherwise the statement is satisfied. In addition, any $(\varphi_{h_j}^\natural, \overline{Q}_{e,h_j}^{\natural})$
such that $\mathcal{J}^\natural_{h_j}(\varphi_{h_j}^{\natural}, \overline{Q}_{e,h_j}^{\natural})<\infty$ is uniformly bounded in $X'$. Therefore, there exists a subsequence (not relabeled) which is weakly convergent in $X'$. Due to the strong convergence of the original sequence, the considered subsequence is weakly convergent to $(\varphi ^\natural, \overline{Q}_{e,0}^{\natural})$, i.e.,
\begin{align}
\varphi_{h_j}^\natural\rightharpoonup \varphi_0^\natural \qquad \text{in}\quad {\rm L}^2(\Omega_1,\mathbb{R}^3), \qquad \overline{Q}_{e,h_j}^{\natural}\rightharpoonup \overline{Q}_{e,0}^{\natural}\qquad \text{in}\quad {\rm L}^{2}(\Omega_1, \text{SO(3)}).
\end{align}

Therefore, we have the weak convergence $(\varphi_{h_j}^\natural,\overline{Q}_{e,h_j}^\natural )$ (without relabeling it) to $(\varphi_{0}^\natural,\overline{Q}_{e,0}^\natural)$ in ${\rm H}^{1}(\omega,\R^3)\times { {{\rm H}^{1}}}(\omega,\SO(3))$. For $\overline{U}_h^\natural=\overline{Q}_e^{\natural,T}\nabla_\eta^h\varphi^\natural[(\nabla_x\Theta)^\natural]^{-1}$ we have
\begin{align}
{{W}}_{\rm mp}(\UHN)=\mu\,\norm{\sym (\UHN-\id_3)}^2+\mu_c\,\norm{\skew (\UHN-\id_3)}^2 + \frac{\lambda}{2}[\tr(\sym(\UHN-\id_3))]^2\,,
\end{align}
while for $\mathcal{E}_{\varphi^\natural,\overline{Q}_{e}^{\natural}}=\mathcal{E}_{\varphi^\natural,\overline{Q}_{e}^{\natural} }^{\parallel}+\mathcal{E}_{\varphi^\natural,\overline{Q}_{e}^{\natural} }^{\perp}$ with $\mathcal{E}_{\varphi^\natural,\overline{Q}_{e}^{\natural}}=(\overline{Q}_{e}^{\natural,T}\nabla_{(\eta_1,\eta_2)} \varphi^\natural-[\nabla y_0]^\natural |0)[(\nabla_x\Theta)^\natural]^{-1}$ we have
$$ {W}_{\rm mp}^{{\rm hom},\natural}(  \mathcal{E}_{\varphi^\natural,\overline{Q}_{e}^{\natural} })=\, \mu\,\lVert  \mathrm{sym}\,   \,\mathcal{E}_{\varphi^\natural,\overline{Q}_{e}^{\natural} }^{\parallel}\rVert^2 +  \mu_{\rm c}\,\lVert \mathrm{skew}\,   \,\mathcal{E}_{\varphi^\natural,\overline{Q}_{e}^{\natural} }^{\parallel}\rVert^2 +\,\dfrac{\lambda\,\mu\,}{\lambda+2\,\mu\,}\,\big[ \mathrm{tr}    (\mathcal{E}_{\varphi^\natural,\overline{Q}_{e}^{\natural} }^{\parallel})\big]^2 +\frac{2\,\mu\, \,  \mu_{\rm c}}{\mu_c\,+\mu\,}\norm{\mathcal{E}_{\varphi^\natural,\overline{Q}_{e}^{\natural} }^\perp}^2\,.$$

  Hence, for the sequence $(\varphi_{h_j}^\natural, \overline{Q}_{e,h_j}^\natural )\in {\rm H}^{1}(\Omega_1,\mathbb{R}^3)\times { {{\rm H}^{1}}}(\Omega_1,\text{\rm SO}(3))$ where $(\varphi_{h_j}^\natural, \overline{Q}_{e,h_j}^{\natural})\rightarrow(\varphi ^\natural_0, \overline{Q}_{e,0}^{\natural})$ with $J^\natural_{h_j}(\varphi_{h_j}^\natural, \overline{Q}_{e,h_j}^\natural )<\infty$, we have
 \begin{align}\label{ineginf}
  {W}_{\rm mp}(\overline{Q}_{e,h_j}^{\natural,T}\nabla^{h_j}_\eta\varphi^\natural_{h_j}[(\nabla_x\Theta)^\natural ]^{-1})&= {W}_{\rm mp}\Big(\overline{Q}_{e,h_j}^{\natural,T}(\nabla_{(\eta_1,\eta_2)}\varphi_{h_j}^\natural|\frac{1}{h_j}\partial_{\eta_3}\varphi_{h_j}^\natural)[(\nabla_x\Theta)^\natural ]^{-1}\Big)\\&\geq  {W}_{\rm mp}^{{\rm hom},\natural}\Big(  \mathcal{E}_{\varphi_{h_j}^\natural,\overline{Q}_{e,h_j}^{\natural}}\Big),
 \end{align}
 where we recall that
  $
  \mathcal{E}_{\varphi_{h_j}^\natural,\overline{Q}_{e,h_j}^{\natural}}:=(\overline{Q}_{e,h_j}^{\natural,T}\nabla_{(\eta_1,\eta_2)} \varphi_{h_j}^\natural-(\nabla y_0)^\natural |0)[(\nabla_x\Theta)^\natural ]^{-1}\,.
 $

 Then by taking the integral over $\Omega_1$ on both sides and taking the $\liminf$ for $h_j$, we obtain
 \begin{align}
 \nonumber\liminf_{h_j\rightarrow 0}\int _{\Omega_1}  {W}_{\rm mp}(\overline{Q}_{e,h_j}^{\natural,T}\nabla^{h_j}_\eta\varphi^\natural_{h_j}[(\nabla_x\Theta)^\natural ]^{-1})\, \det [\nabla_x \Theta]^\natural(\eta)\;dV_ \eta&\geq \liminf_{h_j\rightarrow 0}\int_{\Omega_1} {W}_{\rm mp}^{{\rm hom},\natural}\Big( \mathcal{E}_{\varphi^\natural_{h_j},\overline{Q}_{e,h_j}^{\natural}}\Big)\, \det [\nabla_x \Theta]^\natural(\eta)\;dV_\eta\,.
 \end{align}
 In the expression of $\mathcal{E}_{\varphi^\natural_{h_j},\overline{Q}_{e,h_j}^{\natural}}$, the quantity $[\nabla _x\Theta]^{-1}$ is evaluated in $(x_1,x_2,x_3)=(\eta_1,\eta_2,h\, \eta_3)$. Therefore, we have to  study its behaviour for $h_j\to 0$. 
 	In addition,  we recall the convergence results \cite[Lemma 1]{le1996membrane}:
  {	\begin{align}\label{convdet}
 	\lim_{h_j\to 0}  \det [\nabla_x \Theta]^\natural(\eta_1,\eta_2,\eta_3)&=\lim_{h_j\to 0}  \det [\nabla_x \Theta]^\natural(x_1,x_2,\frac{1}{h_j}x_3)=\det[\nabla_x \Theta]^\natural(\eta_1,\eta_2,0)\notag\\&=\det (\nabla y_0|n_0) \, \quad  \qquad \qquad \qquad \qquad \qquad \quad \, \text{in} \quad C^0(\overline{\Omega}),\notag\\
 	\lim_{h_j\to 0}   [(\nabla_x \Theta)^{-1}]^\natural(\eta_1,\eta_2,\eta_3)&=\lim_{h_j\to 0}   [(\nabla_x \Theta)^{-1}]^\natural(x_1,x_2,\frac{1}{h_j}x_3)\\&=[(\nabla_x \Theta)^{-1}]^\natural(\eta_1,\eta_2,0)= (\nabla_x \Theta)^{-1}(0) \qquad \text{in} \quad C^0(\overline{\Omega}).\notag
 	\end{align}}
 
 Due to \eqref{convdet}, the weak convergence of the sequence $\varphi_{h_j}^{\natural}$ and  the strong convergence of the sequence $\overline{Q}_{e,h_j}^\natural$, we have the weak convergence
  {\begin{align}\label{wce}
 \mathcal{E}_{\varphi_{h_j}^\natural,\overline{Q}_{e,h_j}^{\natural}}&=(\overline{Q}_{e,h_j}^{\natural,T}\nabla_{(\eta_1,\eta_2)} \varphi_{h_j}^\natural-(\nabla y_0)^\natural |0)[(\nabla_x\Theta)^\natural ]^{-1}\,\notag\\&\qquad \rightharpoonup
 (\overline{Q}_{e,0}^{\natural,T}\nabla_{(\eta_1,\eta_2)} \varphi_{0}^\natural-(\nabla y_0)^\natural |0)[(\nabla_x\Theta)^\natural ]^{-1}(0)=:\mathcal{E}_{\varphi^\natural_{0},\overline{Q}_{e,0}^{\natural}}.
 \end{align}}

 Using again \eqref{convdet},  the convexity of the energy function $ {W}_{\rm mp}^{{\rm hom},\natural}$ with respect to $ \mathcal{E}_{\varphi^\natural_{h_j},\overline{Q}_{e,h_j}^{\natural}}$, the Fatou's Lemma, the characterization of $\liminf$ and the weak convergence \eqref{wce} we get 
 \begin{align}\label{Whommp0}
  \liminf_{h_j\rightarrow 0}\int_{\Omega_1} {W}_{\rm mp}^{{\rm hom},\natural}\Big( \mathcal{E}_{\varphi^\natural_{h_j},\overline{Q}_{e,h_j}^{\natural}}\Big)\,\det [\nabla_x \Theta]^\natural(\eta)\;dV_\eta\geq \int_{\Omega_1} {W}_{\rm mp}^{{\rm hom},\natural}\Big( \mathcal{E}_{\varphi^\natural_{0},\overline{Q}_{e,0}^{\natural}}\Big)\det (\nabla y_0|n_0)\;dV_\eta\,.
 \end{align}
Since both $\varphi^\natural_{0}$ and $\overline{Q}_{e,0}^{\natural}$ are independent of the transverse variable $\eta_3$, we also obtain
\begin{align}\label{Whommp1}
\liminf_{h_j\rightarrow 0}\int _{\Omega_1}  {W}_{\rm mp}(\overline{Q}_{e,h_j}^{\natural,T}\nabla^{h_j}_\eta\varphi^\natural_{h_j}[(\nabla_x\Theta)^\natural ]^{-1})\det [\nabla_x \Theta]^\natural(\eta)\;dV_ \eta&\geq\int_{-\frac{1}{2}}^{\frac{1}{2}}\int_{\omega} {W}_{\rm mp}^{{\rm hom},\natural}\Big( \mathcal{E}_{\varphi^\natural_{0},\overline{Q}_{e,0}^{\natural}}\Big)\,\det (\nabla y_0|n_0)\;dV_\eta
\notag\\&=\int_{\omega} {W}_{\rm mp}^{{\rm hom}}\Big( \mathcal{E}_{m,\overline{Q}_{e,0}}\Big)\,\det (\nabla y_0|n_0)\;d\omega\,.\end{align}

 We do the same process for the curvature energy, by using (\ref{last hom curv}), the convexity of $\widetilde{W}_{\rm curv}^{{\rm hom}}$ in its argument and the weak convergence
 \begin{align}
\nonumber\Big(\mathrm{axl}(\overline{Q}_{e,h_j}^{\natural,T}\,\partial_{\eta_1} \overline{Q}_{e,h_j}^\natural)\,|\, \mathrm{axl}(\overline{Q}_{e,h_j}^{\natural,T}\,&\partial_{\eta_2} \overline{Q}_{e,h_j}^\natural)\,|0\Big)[(\nabla_x\Theta)^\natural (\eta)\,]^{-1} \\
&\rightharpoonup \Big(\mathrm{axl}(\overline{Q}_{e,0}^{\natural,T}\,\partial_{\eta_1} \overline{Q}_{e,0}^\natural)\,|\, \mathrm{axl}(\overline{Q}_{e,0}^{\natural,T}\,\partial_{\eta_2} \overline{Q}_{e,0}^\natural)\,|0\Big)[\nabla_x\Theta (0)\,]^{-1}\,.
 \end{align}
Using also   \eqref{convdet},  we arrive at
 \begin{align}\label{after linfcurv}
 \nonumber \liminf_{h_j}\int_{\Omega_1}&\widetilde{W}_{\rm curv}(\Gamma_h^\natural)\,\det [\nabla_x \Theta]^\natural(\eta)\; dV_\eta
 \geq \liminf_{h_j}\int_{\Omega_1}\widetilde{W}_{\rm curv}^{{\rm hom},\natural}(\mathcal{K}_{e}^\natural)\,\det [\nabla_x \Theta]^\natural(\eta)\; dV_\eta\\
 \nonumber&\geq \liminf_{h_j}\int_{\Omega_1}\widetilde{W}_{\rm curv}^{{\rm hom}}(\mathcal{K}_{e,s})\,\det [\nabla_x \Theta]^\natural(\eta)\; dV_\eta\geq \int_{\Omega_1}\widetilde{W}_{\rm curv}^{{\rm hom}}(\mathcal{K}_{e,s})\,\det (\nabla y_0|n_0)\; dV_\eta\\
 &=\int_{-\frac{1}{2}}^{\frac{1}{2}}\int_{\omega}\widetilde{W}_{\rm curv}^{{\rm hom}}(\mathcal{K}_{e,s})\,\det (\nabla y_0|n_0)\; dV_\eta=\int_{\omega}\widetilde{W}_{\rm curv}^{{\rm hom}}(\mathcal{K}_{e,s})\,\det (\nabla y_0|n_0)\; d\omega\,.
 \end{align}
 Since, $ {W}_{\rm mp}(\overline{Q}_{e,h_j}^{\natural,T}\nabla^{h_j}_\eta \varphi^\natural_{h_j}[(\nabla_x\Theta)^\natural ]^{-1})>0$ and $\widetilde{W}_{\rm curv}(\Gamma_h^\natural)>0$, by combining (\ref{Whommp1}) and (\ref{after linfcurv}) we deduce
 \begin{align}\label{compa}
 \liminf_{h_j}\int_{\Omega_1}[ {W}_{\rm mp}&(\overline{Q}_{e,h_j}^{\natural,T}\nabla^{h_j}_\eta \varphi^\natural_{h_j}[(\nabla_x\Theta)^\natural ]^{-1})+\widetilde{W}_{\rm curv}(\Gamma_h^\natural)]\det [\nabla_x \Theta]^\natural(\eta)\; dV_\eta\\
&\geq\int_{\omega}\Big( {W}_{\rm mp}^{{\rm hom}}( \mathcal{E}_{m,\overline{Q}_{e,0}})+\widetilde{W}_{\rm curv}^{\rm hom}(\mathcal{K}_{e,s})\Big)\,\det (\nabla y_0|n_0)\; d\omega\notag=\mathcal{J}_0(m,\overline{Q}_{e,0})\,,
 \end{align}
 where we have used that $\overline{Q}_{e,0}^\natural\equiv\overline{Q}_{e,0}$ and $m=\varphi_0 $. Hence, the $\lim$-$\inf$ inequality \eqref{infcon} is proven.
 \end{proof}
 \subsection{ {Step 2 of the proof: } The lim-sup condition - recovery sequence}\label{limsup}
 Now we show the following lemma
 \begin{lemma}\label{thelimsup} In the hypothesis of Theorem \ref{Theorem homo}, for all $(\varphi_{0}^{\natural},\overline{Q}_{e,0}^\natural)\in {\rm L}^2(\Omega_1)\times {\rm L}^{2}(\Omega_1,\text{\rm SO}(3))$  there exists   $(\varphi_{h_j}^{\natural},\overline{Q}_{e,h_j}^\natural)\in {\rm L}^2(\Omega_1)\times {\rm L}^{2}(\Omega_1,\text{\rm SO}(3))$ with $(\varphi_{h_j}^{\natural},\overline{Q}_{e,h_j}^\natural)\rightarrow (\varphi_{0}^\natural,\overline{Q}_{e,0}^\natural)  $ such that
 \begin{align}
 \mathcal{J}_0  (\varphi_0^\natural ,\overline{Q}_{e,0}^\natural)\geq \limsup_{h_j\rightarrow 0}\mathcal{J}^\natural_{h_j}(\varphi_{h_j}^{\natural}, \overline{Q}_{e,h_j}^\natural).
 \end{align}
 \end{lemma}
\begin{proof}
 Similar to the case of the $\lim$-$\inf$ inequality, we can restrict our attention to sequences $(\varphi_{h_j}^\natural,\overline{Q}_{e,h_j}^\natural)\in X$ such that $\mathcal{J}^{\natural}_{h_j}(\varphi_{h_j}^\natural,\overline{Q}_{e,h_j}^\natural)<\infty$. Therefore, the sequence $(\varphi
 _{h_j}^\natural,\overline{Q}_{e,h_j}^\natural)\in X$ has a weakly convergent subsequence in $X'$, and we can focus on the space ${\rm H}^1(\Omega_1,\mathbb{R}^3)\times { {{\rm H}^{1}}}(\Omega_1,\text{SO(3)})$.
 
 One of the requirements for $\Gamma$-convergence, is the existence of a recovery sequence. Thus, the idea is to define an expansion for the deformation and the microrotation through the thickness. In reality, the minimizers of the energy model can be a good candidate for constructing the recovery sequence. To do so, we look at the first order Taylor expansion of the nonlinear deformation $\varphi^\natural_{h_j}$ in thickness direction $\eta_3$ 
 \begin{align}
 \varphi_{h_j}^\natural(\eta_1,\eta_2,\eta_3)=\;\varphi_{h_j}^\natural(\eta_1,\eta_2,0)+\eta_3\,\partial_{\eta_3}\varphi_{h_j}^\natural(\eta_1,\eta_2,0)\,.
 \end{align}
 
With the formula
 \begin{align}
  d^*&=\Big(1-\frac{\lambda}{2\,\mu\,+\lambda}\iprod{  \mathcal{E}_{m,s} ,\id_3}\Big) \overline{Q}_{e,0}^\natural n_0+\frac{\mu_c\,-\mu\,}{\mu_c\,+\mu\,}\;\overline{Q}_{e,0}^\natural  \mathcal{E}_{m,s}^T n_0\,,
 \end{align}
  and replacing $\frac{1}{h_j}\partial_{\eta_3}\varphi_{h_j}^\natural(\eta_1,\eta_2,0)$ with $d^*(\eta_1,\eta_2)$, which means replacing $\partial_{\eta_3}\varphi_{h_j}^\natural(\eta_1,\eta_2,0)$  by $h_jd^* (\eta_1,\eta_2)$, we make an ansatz for our recovery sequence as following
 \begin{align}\label{taylor phi}
 \varphi_{h_j}^\natural(\eta_1,\eta_2,\eta_3):=\;\varphi_{0}^\natural(\eta_1,\eta_2)+h_j\.\eta_3\.d^* (\eta_1,\eta_2).
 \end{align}
%
 Since $\nabla_{(\eta_1,\eta_2)} \varphi^\natural\in {\rm L}^2(\omega,\R^3)$ and $\overline{Q}_{e,0}\in \SO(3)$, we obtain that $d^* $
 belongs to ${\rm L}^2(\omega,\mathbb{R}^3)$ and by letting $h_j\rightarrow 0$, it can be seen that for this ansatz $\varphi_{h_j}^\natural\rightarrow \varphi_{0}^\natural$.
 
 The reconstruction for the rotation $\overline{Q}_{e,0}$ is not obvious, since on the one hand we have to maintain the rotation constraint along the sequence and on the other hand we must approach the lower bound, which excludes the simple reconstruction $\overline{Q}_{e,h_j}^\natural(\eta_1,\eta_2,\eta_3)=\overline{Q}_{e,0}(\eta_1,\eta_2)$. In order to meet both requirements we consider therefore
 \begin{align}
 \overline{Q}_{e,h_j}^\natural(\eta_1,\eta_2,\eta_3):=\overline{Q}_{e,0}(\eta_1,\eta_2)\cdot\exp(h_j\.\eta_3\.A^*(\eta_1,\eta_2)),
 \end{align}
 where $A^*\in \mathfrak{so}(3)$ is the term obtained in (\ref{inf curv}), depending on the given $\overline{Q}_{e,0}$, and we note that $A^*\in {\rm L}^{2}(\omega,\mathfrak{so}(3))$ by the coercivity of $\widetilde{W}_{\rm curv}$. Since $\exp:\mathfrak{so}(3)\rightarrow \text{\rm SO}(3)$, we obtain that $\overline{Q}_{e,h_j}^\natural\in \text{\rm SO}(3)$ and for $h_j\rightarrow 0$,  we have $\overline{Q}_{e,h_j}^\natural \rightarrow \overline{Q}_{e,0}\in {\rm L}^{2}(\Omega_1,\text{\rm SO}(3))$.
 
 Since $d^* $ need not to be differentiable, we should consider another modified recovery sequence. For fixed $\varepsilon>0$, we select $d_\varepsilon\in { {{\rm H}^{1}}}(\omega,\mathbb{R}^3)$ such that $\norm{d_\varepsilon-d^* }_{{\rm L}^2(\omega,\mathbb{R}^3)}<\varepsilon$. Therefore, accordingly we define the final recovery sequence for the deformation as following
 \begin{align}\label{recovery}
 \varphi_{h_j,\varepsilon}^\natural(\eta_1,\eta_2,\eta_3):=\;\varphi_{0}^\natural(\eta_1,\eta_2)+h_j\.\eta_3\.d_\varepsilon(\eta_1,\eta_2).
 \end{align}
 The same argument holds for $A^*$ {, i.e.,}  for fixed $\varepsilon>0$ we may choose $A_\varepsilon\in { {{\rm H}^{1}}}(\omega,\mathfrak{so}(3))$ such that $\norm{A_\varepsilon-A^*}_{{\rm L}^{2}(\omega,\mathfrak{so}(3))}<\varepsilon$. Hence, the final recovery sequence for the microrotation is 
 \begin{align}
 \overline{Q}_{e,h_j,\varepsilon}^\natural(\eta_1,\eta_2,\eta_3):=\overline{Q}_{e,0}(\eta_1,\eta_2)\cdot\exp(h_j\.\eta_3\.A_\varepsilon(\eta_1,\eta_2)).
 \end{align}
 The gradient of the new recovery sequence of deformation is 
 \begin{align}
 \nonumber\nabla_\eta\varphi_{h_j,\varepsilon}^\natural(\eta_1,\eta_2,\eta_3)&=(\nabla_{(\eta_1,\eta_2)}\varphi_{0}^\natural(\eta_1,\eta_2)|0)+h_j(0|d_\varepsilon(\eta_1,\eta_2))+h_j\eta_3(\nabla_{(\eta_1,\eta_2)}d_\varepsilon(\eta_1,\eta_2)|0)\\
 &=(\nabla\varphi_{0}^\natural(\eta_1,\eta_2)|h_jd_\varepsilon(\eta_1,\eta_2))+h_j\eta_3(\nabla d_\varepsilon(\eta_1,\eta_2)|0),
 \end{align}
 and the different terms in the curvature energy are
 \begin{align}
 \overline{Q}_{e,h_j,\varepsilon}^{\natural,T}\partial_{\eta_1}\overline{Q}_{e,h_j,\varepsilon}^{\natural}&=\exp(h_j\eta_3A_\varepsilon)^T\overline{Q}_{e,0}^T[\partial_{\eta_1}\overline{Q}_{e,0}\exp(h_j\eta_3A_{\varepsilon})+\overline{Q}_{e,0}D\exp(h_j\eta_3A_\varepsilon).[h_j\eta_3\partial_{\eta_1}A_{\varepsilon}]],\notag\\
 \overline{Q}_{e,h_j,\varepsilon}^{\natural,T}\partial_{\eta_2}\overline{Q}_{e,h_j,\varepsilon}^{\natural}&=\exp(h_j\eta_3A_\varepsilon)^T\overline{Q}_{e,0}^T[\partial_{\eta_2}\overline{Q}_{e,0}\exp(h_j\eta_3A_{\varepsilon})+\overline{Q}_{e,0}D\exp(h_j\eta_3A_\varepsilon).[h_j\eta_3\partial_{\eta_2}A_{\varepsilon}]],\\
 \nonumber\overline{Q}_{e,h_j,\varepsilon}^{\natural,T}\partial_{\eta_3}\overline{Q}_{e,h_j,\varepsilon}^{\natural}&=\exp(h_j\eta_3A_\varepsilon)^T\overline{Q}_{e,0}^T[\partial_{\eta_3}\overline{Q}_{e,0}\exp(h_j\eta_3A_{\varepsilon})+\overline{Q}_{e,0}D\exp(h_j\eta_3A_\varepsilon).[h_jA_{\varepsilon}]]\\
 &=h_j\exp(h_j\eta_3A_\varepsilon(\eta_1,\eta_2))^TD\exp(h_j\eta_3A_\varepsilon(\eta_1,\eta_2)).[A_\varepsilon],\notag
 \end{align}
with $\partial_{\eta_i}A_{\varepsilon}\in \mathfrak{so}(3)$. 
 Now we introduce the quantities
 \begin{align}
 \nonumber \widetilde{U}_0&=\overline{Q}_{e,0}^T(\nabla\varphi_0(\eta_1,\eta_2) |d^*(\eta_1,\eta_2) )[(\nabla_x\Theta)(0)]^{-1},\\
 \nonumber \widetilde{U}_{h_j}^\varepsilon &=\overline{Q}_{e,h_j,\varepsilon}^{\natural ,T}\Big((\nabla \varphi_{0}(\eta_1,\eta_2)|d_\varepsilon(\eta_1,\eta_2))+h_j\eta_3(\nabla d_\varepsilon(\eta_1,\eta_2)|0)\Big)[(\nabla_x\Theta)^\natural(\eta) ]^{-1},\\
  \widetilde{U}_{0}^\varepsilon &=\overline{Q}_{e,0}^{T}(\nabla \varphi_{0}(\eta_1,\eta_2)|d_\varepsilon(\eta_1,\eta_2))[(\nabla_x\Theta)(0) ]^{-1},\\
 \nonumber\Gamma_{h_j,\varepsilon}^{\natural}&:=\Big(\underbrace{\axl\Big( \overline{Q}_{e,h_j,\varepsilon}^{\natural , T}\partial_{\eta_1}\overline{Q}_{e,h_j,\varepsilon}^{\natural}\Big)}_{:=\,\Gamma_{h_j,\varepsilon}^{1,\natural}}\,|\, \underbrace{\axl\Big( \overline{Q}_{e,h_j,\varepsilon}^{\natural , T}\partial_{\eta_2}\overline{Q}_{e,h_j,\varepsilon}^{\natural}\Big)}_{:=\,\Gamma_{h_j,\varepsilon}^{2,\natural}}\,|\, \underbrace{\frac{1}{h_j}\axl\Big( \overline{Q}_{e,h_j,\varepsilon}^{\natural , T}\partial_{\eta_3}\overline{Q}_{e,h_j,\varepsilon}^{\natural}\Big)}_{:=\,\Gamma_{h_j,\varepsilon}^{3,\natural}}\Big)[(\nabla_x\Theta)^\natural(\eta) ]^{-1},\\
 \nonumber \Gamma_{0} &:=\Big(\underbrace{\axl\Big( \overline{Q}_{e,0}^{T}\partial_{\eta_1}\overline{Q}_{e,0}\Big)}_{:=\,\mathcal{K}_{e,s}^1}\,|\,\underbrace{\axl\Big( \overline{Q}_{e,0}^{T}\partial_{\eta_2}\overline{Q}_{e,0}\Big)}_{:=\,\mathcal{K}_{e,s}^2}\,|\,0\Big)\,[(\nabla_x\Theta)(0) ]^{-1}\,.\notag 
 \end{align}
 Note that
 \begin{align}
 \Gamma_{h_j,\varepsilon}^{3,\natural}
 &:=\axl\Big(\exp(h_j\eta_3A_\varepsilon(\eta_1,\eta_2))^T\D\exp(h_j\eta_3A_\varepsilon(\eta_1,\eta_2)).[A_\varepsilon]\Big).
 \end{align}
 It holds
 \begin{align*}
 &\norm{\widetilde{U}_{h_j}^\varepsilon-\widetilde{U}_{0}^\varepsilon}\rightarrow 0,\,\,\quad \text{as}\quad h_j\rightarrow 0,\hspace{3,8cm}\norm{\widetilde{U}_{h_j}^\varepsilon-\widetilde{U}_{0}}\rightarrow 0,\,\,\,\quad \text{as}\quad h_j\rightarrow 0, \varepsilon\rightarrow 0,\\
 &\norm{\Gamma_{h_j,\varepsilon}^{i,\natural}-\mathcal{K}_{e,s}^i}\rightarrow 0, \quad\text{as}\qquad h_j\rightarrow 0, \varepsilon\rightarrow 0,\quad i=1,2\,,\qquad\norm{ \Gamma_{h_j,\varepsilon}^{3,\natural}-\axl A_\varepsilon}\rightarrow 0, \quad\text{as}\quad h_j\rightarrow 0.
 \end{align*}
 We also have
\begin{align}
\nonumber\norm{\widetilde{U}_{0}^\varepsilon-\widetilde{U}_0}^2&= \norm{\overline{Q}_{e,0}^T(\nabla \varphi_{0}|d_\varepsilon)[\nabla_x\Theta(0)]^{-1}-\overline{Q}_{e,0}^T(\nabla \varphi_{0}|d^* )[\nabla_x\Theta(0)]^{-1})}^2\\
\nonumber&=\norm{\overline{Q}_{e,0}^T(0|0|d_\varepsilon-d^* )[\nabla_x\Theta(0)]^{-1}}^2\\
\nonumber&=\iprod{\overline{Q}_{e,0}^T(0|0|d_\varepsilon-d^* )[\nabla_x\Theta(0)]^{-1},\overline{Q}_{e,0}^T(0|0|d_\varepsilon-d^* )[\nabla_x\Theta(0)]^{-1}}\\
&=\iprod{\overline{Q}_{e,0}\overline{Q}_{e,0}^T(0|0|d_\varepsilon-d^* ),(0|0|d_\varepsilon-d^* )[\nabla_x\Theta(0)]^{-1}[\nabla_x\Theta(0)]^{-T}}\\
\nonumber&=\iprod{(0|0|d_\varepsilon-d^* ),(0|0|d_\varepsilon-d^* )(\widehat{\rm I}_{y_0})^{-1}}=\iprod{(0|0|d_\varepsilon-d^* )^T(0|0|d_\varepsilon-d^* ),(\widehat{\rm I}_{y_0})^{-1}}\\
\nonumber&=\iprod{(0|0|(d_\varepsilon-d^* )^T(d_\varepsilon-d^* )),( \widehat{\rm I}_{y_0})^{-1}}=\iprod{d_\varepsilon-d^* ,d_\varepsilon-d^* }=\norm{d_\varepsilon-d^* }^2\:\rightarrow 0\quad\text{as}\quad \varepsilon\rightarrow 0.
\end{align}
We may write
 \begin{align}
 \nonumber \mathcal{J}_{h_j}^{\natural,\textrm{mp}}(\varphi_{h_j,\varepsilon}^\natural,\overline{Q}_{e,h_j,\varepsilon}^\natural)&:=\int_{\Omega_1}  {W}_{\rm mp}(\widetilde{U}_{h_j}^\varepsilon)\,\det ((\nabla_x \Theta)^\natural)(\eta)\,dV_\eta
 \\&=\int_{\Omega_1}\Big[  {W}_{\rm mp}(\widetilde{U}_{h_j}^\varepsilon)-  {W}_{\rm mp}(\widetilde{U}_0)+  {W}_{\rm mp}(\widetilde{U}_0)\Big]\,\det ((\nabla_x \Theta)^\natural)(\eta)\,dV_\eta\notag\\
  &=\int_{\Omega_1}\Big[  {W}_{\rm mp}(\widetilde{U}_{h_j}^\varepsilon+\widetilde{U}_0-\widetilde{U}_0)-  {W}_{\rm mp}(\widetilde{U}_0)+  {W}_{\rm mp}(\widetilde{U}_0)\Big]\,\det ((\nabla_x \Theta)^\natural)(\eta)\,dV_\eta\\
 \nonumber &\leq\int_{\Omega_1}\Big[|  {W}_{\rm mp}(\widetilde{U}_{h_j}^\varepsilon+\widetilde{U}_0-\widetilde{U}_0)-  {W}_{\rm mp}(\widetilde{U}_0)|+  {W}_{\rm mp}(\widetilde{U}_0)\Big]\,\det ((\nabla_x \Theta)^\natural)(\eta)\,dV_\eta,
 \end{align}
 where we used that ${W}_{\rm mp}$ is positive. The exact quadratic expansion in the neighborhood of the point $\widetilde{U}_{h_j}^\varepsilon=\widetilde{U}_{0}+\widetilde{U}_{h_j}^\varepsilon-\widetilde{U}_{0} $ for  $  {W}_{\rm mp}$ is given by
 \begin{align*}
   {W}_{\rm mp}(\widetilde{U}_0+\widetilde{U}_{h_j}^\varepsilon-\widetilde{U}_0)=  {W}_{\rm mp}(\widetilde{U}_0)+\bigl\langle D  {W}_{\rm mp}(\widetilde{U}_0),\widetilde{U}_{h_j}^\varepsilon-\widetilde{U}_0\bigr\rangle+\frac{1}{2}D^2  {W}_{\rm mp}(\widetilde{U}_0).(\widetilde{U}_{h_j}^\varepsilon-\widetilde{U}_0,\widetilde{U}_{h_j}^\varepsilon-\widetilde{U}_0)\,.
 \end{align*}
 Therefore, with the assumption that $\norm{\widetilde{U}_{h_j}^\varepsilon-\widetilde{U}_0}\leq 1$, we have the following relations
 \begin{align}\label{supmp}
 \nonumber \mathcal{J}_{h_j}^{\natural,\textrm{mp}}&(\varphi_{h_j,\varepsilon}^\natural,  \overline{Q}_{e,h_j,\varepsilon}^\natural)\leq\int_{\Omega_1} \Big[ {W}_{\rm mp}(\widetilde{U}_0)+\norm{D  {W}_{\rm mp}(\widetilde{U}_0)}\norm{\widetilde{U}_{h_j}^\varepsilon-\widetilde{U}_0}+\frac{1}{4}\norm{D^2  {W}_{\rm mp}(\widetilde{U}_0)}\norm{\widetilde{U}_{h_j}^\varepsilon-\widetilde{U}_0}^2\;\Big]\,\det ((\nabla_x \Theta)^\natural)(\eta)\, dV_\eta\\
 \nonumber&\leq \int_{\Omega_1} \Big[ {W}_{\rm mp}(\widetilde{U}_0)+C\norm{\widetilde{U}_0}\norm{\widetilde{U}_{h_j}^\varepsilon-\widetilde{U}_0}+C_1\norm{\widetilde{U}_{h_j}^\varepsilon-\widetilde{U}_0}\;\Big]\,\det ((\nabla_x \Theta)^\natural)(\eta)\, dV_\eta\\
 &\leq \int_{\Omega_1} \Big[ {W}_{\rm mp}(\widetilde{U}_0)+(C\norm{\widetilde{U}_0}+C_1)\norm{\widetilde{U}_{h_j}^\varepsilon-\widetilde{U}_0}\; \Big]\,\det ((\nabla_x \Theta)^\natural)(\eta)\,dV_\eta,
 \end{align}
 where $C$ and $C_1$ are upper bounds for $\norm{D  {W}_{\rm mp}(\widetilde{U}_0)}$ and $\norm{D^2  {W}_{\rm mp}(\widetilde{U}_0)}$, respectively.
 Now we consider the terms of $ \widetilde{W}_{\rm curv}$
 \begin{align}\label{supcurv}
 \nonumber \mathcal{J}_{h_j}^{\natural,\textrm{curv}}(\Gamma_{h_j,\varepsilon}^\natural)&:=\int_{\Omega_1} \widetilde{W}_{\rm curv}\Big((\Gamma_{h_j,\varepsilon}^{1,\natural},\Gamma_{h_j,\varepsilon}^{2,\natural}, \Gamma_{h_j,\varepsilon}^{3,\natural})[(\nabla_x\Theta)^\natural(\eta) ]^{-1}\Big)\,\det ((\nabla_x \Theta)^\natural)(\eta)\;dV_\eta\notag\\&\leq  \int_{\Omega_1}\Big[ \widetilde{W}_{\rm curv}\Big((\Gamma_{h_j,\varepsilon}^{1,\natural},\Gamma_{h_j,\varepsilon}^{2,\natural}, \Gamma_{h_j,\varepsilon}^{3,\natural})[(\nabla_x\Theta)^\natural(\eta) ]^{-1}\Big) - \widetilde{W}_{\rm curv}\Big((\Gamma_{0}^{1},\Gamma_{0}^{2}, A_\varepsilon)[(\nabla_x\Theta)^\natural(\eta) ]^{-1}\Big)\notag\\&\qquad\quad+ \widetilde{W}_{\rm curv}\Big((\Gamma_{0}^{1},\Gamma_{0}^{2}, A_\varepsilon)[(\nabla_x\Theta)^\natural(\eta) ]^{-1}\Big) - \widetilde{W}_{\rm curv}\Big((\Gamma_{0}^{1},\Gamma_{0}^{2}, A^*)[(\nabla_x\Theta)^\natural(\eta) ]^{-1}\Big)\\&\qquad\quad+\widetilde{W}_{\rm curv}\Big((\Gamma_{0}^{1},\Gamma_{0}^{2}, A^*)[(\nabla_x\Theta)^\natural(\eta) ]^{-1}\Big)\Big]\,\det ((\nabla_x \Theta)^\natural)(\eta)\;dV_\eta\notag\\&\leq  \int_{\Omega_1}\Big[ \Big|\widetilde{W}_{\rm curv}\Big((\Gamma_{h_j,\varepsilon}^{1,\natural},\Gamma_{h_j,\varepsilon}^{2,\natural}, \Gamma_{h_j,\varepsilon}^{3,\natural})[(\nabla_x\Theta)^\natural(\eta) ]^{-1}\Big) - \widetilde{W}_{\rm curv}\Big((\Gamma_{0}^{1},\Gamma_{0}^{2}, A_\varepsilon)[(\nabla_x\Theta)^\natural(\eta) ]^{-1}\Big)\Big|\notag\\&\qquad\quad+ \Big|\widetilde{W}_{\rm curv}\Big((\Gamma_{0}^{1},\Gamma_{0}^{2}, A_\varepsilon)[(\nabla_x\Theta)^\natural(\eta) ]^{-1}\Big) - \widetilde{W}_{\rm curv}\Big((\Gamma_{0}^{1},\Gamma_{0}^{2}, A^*)[(\nabla_x\Theta)^\natural(\eta) ]^{-1}\Big)\Big|\notag\\&\qquad\quad+\widetilde{W}_{\rm curv}\Big((\Gamma_{0}^{1},\Gamma_{0}^{2}, A^*)[(\nabla_x\Theta)^\natural(\eta) ]^{-1}\Big)\Big]\,\det ((\nabla_x \Theta)^\natural)(\eta)\;dV_\eta,\notag
 \end{align}
 where we have used the triangle inequality. 
 
 Note that beside the boundedness of  $\det [\nabla_x \Theta]^\natural(0)$, due to the hypothesis that  $\det[ \nabla_x \Theta(0)]\geq a_0>0$, it follows that there exits a constant $C>0$ such that
 \begin{align}
 \forall \, x\in \overline{\omega}\col\qquad \lVert  [\nabla_x \Theta(0)]^{-1}\rVert\leq C.
 \end{align}
 We notice that both energy parts are positive and $\det [\nabla_x \Theta](0)$ is bounded. Also $ \widetilde{W}_{\rm curv}$ is continuous  and $\norm{A_\varepsilon-A^*}_{{\rm L}^{2}(\omega,\mathfrak{so}(3))}<\varepsilon$. By using (\ref{inf curv}) and \eqref{convdet}, and applying $\limsup_{h_j\rightarrow 0}$ on both sides of (\ref{supmp}) and (\ref{supcurv}) with $h_j\to0$ and $\varepsilon \rightarrow 0$ we get
 \begin{align}
 \nonumber\limsup_{h_j\rightarrow 0} \mathcal{J}_{h_j}^{\natural}(\varphi_{h_j,\varepsilon}^\natural,\overline{Q}_{e,h_j,\varepsilon}^\natural)&\leq  \int_{\Omega_1}( {W}_{\rm mp}(\widetilde{U}_0)+\widetilde{W}_{\rm curv}\Big((\Gamma_{0}^{1},\Gamma_{0}^{2}, A^*)[(\nabla_x\Theta)^\natural(0) ]^{-1}\Big)\,\det [\nabla_x \Theta](0)\;dV_\eta\\
 &= \int_{\Omega_1}(  {W}_{\rm mp}(\widetilde{U}_0)+\widetilde{W}_{\rm curv}\Big((\Gamma_{0}^{1},\Gamma_{0}^{2},0)[(\nabla_x\Theta)(0) ]^{-1}\Big)\,\det [\nabla_x \Theta](0)\;dV_\eta.
 \end{align}
 However, $  {W}_{\rm mp}(\widetilde{U}_0)$ and $\widetilde{W}_{\rm curv}\Big((\Gamma_{0}^{1},\Gamma_{0}^{2},0)[(\nabla_x\Theta)(0) ]^{-1}\Big)$ are already independent of the third variable $\eta_3$, hence  we deduce
 \begin{align}
 \nonumber\limsup_{h_j\rightarrow 0} \mathcal{J}_{h_j}^{\natural}(\varphi_{h_j,\varepsilon}^\natural,\overline{Q}_{e,h_j,\varepsilon}^\natural)&\leq \mathcal{J}_{0}(m ,\overline{Q}_{e,0}),\quad\quad \varphi ^\natural\equiv\varphi ,\quad\quad \overline{Q}_{e,0}^\natural\equiv\overline{Q}_{e,0}\quad\quad \text{and}\quad\quad m=\varphi_0\,.\qedhere
 \end{align}
 \end{proof}
 
 \section{The Gamma-limit including external loads}\label{sec load}
 The main result of this paper is the following theorem
 \begin{theorem}\label{gammalimit I}
 	Assume  that the initial configuration  is defined by  a continuous injective mapping $\,y_0:\omega\subset\mathbb{R}^2\rightarrow\mathbb{R}^3$  which admits an extension to $\overline{\omega}$ into  $C^2(\overline{\omega};\mathbb{R}^3)$ such that $\det[\nabla_x\Theta(0)] \geq\, a_0 >0$ on $\overline{\omega}$,	where $a_0$ is a constant, and  assume that the boundary data satisfy the conditions
 	\begin{equation}
 	\varphi^\natural_d=\varphi_d\big|_{\Gamma_1} \text{(in the sense of traces) for } \ \varphi_d\in {\rm H}^1(\Omega_1;\mathbb{R}^3).
 	\end{equation}
 	Let the constitutive parameters satisfy 
 	\begin{align}
 	\mu\,>0, \quad\quad \kappa>0, \quad\quad \mu_{\rm c}> 0,\quad\quad a_1>0,\quad\quad a_2>0,\quad\quad a_3>0\.,
 	\end{align} 
 	Then, 
 	for any sequence $(\varphi_{h_j}^\natural, \overline{Q}_{e,h_j}^{\natural})\in X$ such that $(\varphi_{h_j}^\natural, \overline{Q}_{e,h_j}^{\natural})\rightarrow(\varphi_0, \overline{Q}_{e,0})$ as $h_j\to 0$,	the sequence of functionals $\mathcal{I}_{h_j}\col X\rightarrow \overline{\mathbb{R}}$ 
 	\begin{align}
 	\mathcal{I}_{h_j}^\natural(\varphi^\natural,  \nabla^{h_j}_\eta \varphi^\natural,\overline{Q}_e^\natural,\Gamma_{h_j}^\natural) 
 	&=
 	\begin{cases}
 	\displaystyle\frac{1}{h_j}\,J_{h_j}^\natural(\varphi^\natural,  \nabla^{h_j}_\eta \varphi^\natural, \overline{Q}_e^\natural,\Gamma_{h_j}^\natural)-\frac{1}{h} {\Pi}^\natural_{h_j}(\varphi^\natural,\overline{Q}_e^\natural)\quad &\text{if}\;\; (\varphi^\natural,\overline{Q}_e^\natural)\in \mathcal{S}',
 	\\
 	+\infty\qquad &\text{else in}\; X\,,
 	\end{cases}
 	\end{align}
 	\textit{ $\Gamma$-converges} to the limit energy  functional $\mathcal{I}_0\col X\rightarrow \overline{\mathbb{R}}$ defined by
 	\begin{equation}
 	\mathcal{I}_0 (m,\overline{Q}_{e,0}) =
 	\begin{cases}
 	\mathcal{J}_0 (m,\overline{Q}_{e,0})-\Pi(m,\overline{Q}_{e,0})\quad &\text{if}\quad (m,\overline{Q}_{e,0})\in \mathcal{S}_\omega'\.,
 	\\
 	+\infty\qquad &\text{else in}\; X,
 	\end{cases}
 	\end{equation}
 	where
 	\begin{equation}
 	\mathcal{J}_0 (m,\overline{Q}_{e,0}) =
 	\begin{cases}
 \dd	\int_{\omega} [ {W}_{\rm mp}^{{\rm hom}}(\mathcal{E}_{m,\overline{Q}_{e,0}})+\widetilde{W}_{\rm curv}^{\rm hom}(\mathcal{K}_{e,s})]\det (\nabla y_0|n_0)\; d\omega\quad &\text{if}\quad (m,\overline{Q}_{e,0})\in \mathcal{S}_\omega'\.,
 	\\
 	+\infty\qquad &\text{else in}\; X,
 	\end{cases}
 	\end{equation}
 	and $\Pi(m,\overline{Q}_{e,0})=\Pi_{\widetilde{f},\omega}(\widetilde{u}_0)+\Pi_{\widetilde{c},\gamma_1}(\overline{Q}_{e,0})$  {defined by the external loads}.
 \end{theorem}
\begin{rem}\label{rem}

 Before proving the above theorem, we will give the expression of the external loads potential in $\Omega_1$. We have
  	\begin{align}
  	{\Pi}^\natural_h(\varphi^\natural,\overline{Q}_e^\natural)={\Pi}_f^\natural(\varphi^\natural)+{\Pi}_{c}^\natural(\overline{Q}_e^\natural)\,,\quad\quad{\Pi}_f^\natural(\varphi^\natural)=h\,\int_{\Omega_1} \iprod{\widetilde{f}^\natural,\widetilde{u}^\natural}\,dV_\eta\,, \quad{\Pi}_c^\natural(\overline{Q}_e^\natural)=h\,\int_{\Gamma_1} \iprod{\widetilde{c}^\natural,\overline{Q}_e^\natural}\,dS_\eta\,,
  	\end{align}
  	with $\widetilde{f}^\natural(\eta)=\widetilde{f}(\zeta(\eta))$, $\widetilde{u}^\natural(\eta)=\widetilde{u}(\zeta(\eta))$,
  	$\widetilde{c}^\natural(\zeta)=\widetilde{c}(\zeta(\eta))$, $\overline{Q}_e^\natural(\eta)=\overline{Q}_e(\zeta(\eta))$ and $\widetilde{u}^\natural(\eta_i)=\varphi^\natural(\eta_i)-\Theta^\natural(\eta_i)$. We use the following expressions
  	\begin{align}
  	\nonumber\Theta^\natural(\eta)&=y_0^\natural(\eta_1,\eta_2)+h_j\.\eta_3\.n_0(\eta_1,\eta_2)\,,\quad\quad \varphi^\natural_{h_j}(\eta)=\varphi_{0}^\natural(\eta_1,\eta_2)+h_j\.\eta_3\.d^*(\eta_1,\eta_2)\,,\\
  	\widetilde{u}^\natural(\eta_i)&=\varphi^\natural(\eta_i)-\Theta^\natural(\eta_i)=\underbrace{\Big(\varphi_{0}^\natural(\eta_1,\eta_2)-y_0^\natural(\eta_1,\eta_2)\Big)}_{\widetilde{u}_0(\eta_1,\eta_2)}+h_j\eta_3\Big(d^*(\eta_1,\eta_2)-n_0(\eta_1,\eta_2)\Big)\,.
  	\end{align}
  	We calculate the  {work due to the} loads separately. We have
  	\begin{align}
  	\nonumber{\Pi}_f^{\natural}(\varphi^\natural_{h_j})&=h_j\int_{\Omega_1}\iprod{\widetilde{f}^\natural,\widetilde{u}^\natural}dV_\eta=h_j\int_{\Omega_1}\iprod{\widetilde{f}^\natural,\widetilde{u}_0(\eta_1,\eta_2)}dV_\eta+h_j^2\eta_3\int_{\Omega_1}\iprod{\widetilde{f}^\natural,(d^*(\eta_1,\eta_2)-n_0(\eta_1,\eta_2))}dV_\eta\\
  	\nonumber&=h_j\int_{\omega}\int_{-\frac{1}{2}}^{\frac{1}{2}}\iprod{\widetilde{f}^\natural,\widetilde{u}_0(\eta_1,\eta_2)}d\eta_3\,d\omega+h_j^2\int_{\omega}\int_{-\frac{1}{2}}^{\frac{1}{2}}\eta_3\iprod{\widetilde{f}^\natural,(d^*(\eta_1,\eta_2)-n_0(\eta_1,\eta_2))}d\eta_3\,d\omega\\
  	&=h_j\int_\omega\iprod{\int_{-\frac{1}{2}}^{\frac{1}{2}}\widetilde{f}^\natural\,d\eta_3,\widetilde{u}_0(\eta_1,\eta_2)}\,d\omega+h_j^2\int_\omega\iprod{\int_{-\frac{1}{2}}^{\frac{1}{2}}\eta_3\widetilde{f}^\natural\,d\eta_3,(d^*-n_0)(\eta_1,\eta_2)}\,d\omega:=\Pi_{\widetilde{f},\omega}(\widetilde{u}_0)\,.
  	\end{align}
  	For applying the same method for the potential of external applied boundary surface couple, we need to have an approximation for the exponential function which is already used in the expression of the recovery sequence for the microrotation $\overline{Q}^\natural_{e,h_j}$, i.e., $\exp(X)= \id+X+\frac{1}{2!}X^2+\cdots$,
  	  	which implies
  	\begin{align}
  	\overline{Q}_{e,h_j}^\natural=\overline{Q}_{e,0}\cdot\exp(h_j\.\eta_3\.A^*(\eta_1,\eta_2))=\overline{Q}_{e,0}+\overline{Q}_{e,0}h_j\.\eta_3\.A^*(\eta_1,\eta_2)+\frac{1}{2}\overline{Q}_{e,0}h_j^2\.\eta_3^2\.A^*(\eta_1,\eta_2)^2+\cdots\,.
  	\end{align}
  	Hence,
  	\begin{align}
  	\nonumber{\Pi}_{c}^\natural(\overline{Q}_{e,h_j}^\natural)&=h_j\int_{\Gamma_1}\iprod{\widetilde{c}^\natural,\overline{Q}_{e,0}+\overline{Q}_{e,0}h_j\eta_3A^*(\eta_1,\eta_2)+\frac{1}{2}\overline{Q}_{e,0}h_j^2\.\eta_3^2\.A^*(\eta_1,\eta_2)^2+\cdots}\,dS_\eta\\
  	\nonumber&=h_j\int_{\Gamma_1}\iprod{\widetilde{c}^\natural,\overline{Q}_{e,0}}\,dS_\eta+h_j^2\eta_3\int_{\Gamma_1}\iprod{\widetilde{c}^\natural,\overline{Q}_{e,0}A^*(\eta_1,\eta_2)}\,dS_\eta+\frac{1}{2}h_j^3\eta_3^2\int_{\Gamma_1}\iprod{\widetilde{c}^\natural,\overline{Q}_{e,0}A^*(\eta_1,\eta_2)^2}\,dS_\eta+\cdots\\
  	\nonumber&=h_j\int_{(\gamma_1\times [-\frac{1}{2},\frac{1}{2}])}\iprod{\widetilde{c}^\natural,\overline{Q}_{e,0}}\,dS_\eta+h_j^2\eta_3\int_{(\gamma_1\times [-\frac{1}{2},\frac{1}{2}])}\iprod{\widetilde{c}^\natural,\overline{Q}_{e,0}A^*(\eta_1,\eta_2)}\,dS_\eta+O(h_j^3)\\
  	&=h_j\int_{\gamma_1 }\int_{-\frac{1}{2}}^{\frac{1}{2}}\iprod{\widetilde{c}^\natural,\overline{Q}_{e,0}}\,d\eta_3\,ds+h_j^2\eta_3\int_{\gamma_1 }\int_{-\frac{1}{2}}^{\frac{1}{2}}\iprod{\widetilde{c}^\natural,\overline{Q}_{e,0}A^*(\eta_1,\eta_2)}\,d\eta_3\,ds+O(h_j^3)\\
  	\nonumber&=\underbrace{h_j\int_{\gamma_1}\iprod{\int_{-\frac{1}{2}}^{\frac{1}{2}}\widetilde{c}^\natural\,d\eta_3,\overline{Q}_{e,0}}\,ds+h_j^2\int_{\gamma_1}\iprod{\int_{-\frac{1}{2}}^{\frac{1}{2}}\eta_3\widetilde{c}^\natural\,d\eta_3,\overline{Q}_{e,0}A^*(\eta_1,\eta_2)}\,ds}_{:=\,\Pi_{\widetilde{c},\gamma_1}(\overline{Q}_{e,0})}+\,O(h_j^3)\,.
  	\end{align}
  	 Therefore,
  	\begin{align}
\Pi_{h_j}^\natural(\varphi^\natural,\overline{Q}_{e}^\natural)&=\Pi_{\widetilde{f},\omega}(\widetilde{u}_0)+\Pi_{\widetilde{c},\gamma_1}(\overline{Q}_{e,0})+O(h_j^3)=\Pi(m,\overline{ Q}_{e,0})+O(h^3_{h_j})\,,\qquad \widetilde{u}_0=m-y_0\,,
  	\end{align}
  	which regularity condition confirm the boundedness and continuity of external loads.
  \end{rem}
Now we come back to the proof of Theorem \ref{gammalimit I}.

\begin{proof}[Proof of Theorem \ref{gammalimit I}]
	As a first step we have  considered the functionals
	\begin{align}
	\mathcal{J}_h^\natural(\varphi^\natural,  \nabla^h_\eta \varphi^\natural,\overline{Q}_e^\natural,\Gamma_h^\natural) 
	&=
	\begin{cases}
	\displaystyle\frac{1}{h}\,J_h^\natural(\varphi^\natural,  \nabla^h_\eta \varphi^\natural, \overline{Q}_e^\natural,\Gamma_h^\natural)\quad &\text{if}\;\; (\varphi^\natural,\overline{Q}_e^\natural)\in \mathcal{S}',
	\\
	+\infty\qquad &\text{else in}\; X.
	\end{cases}
	\end{align}
	In subsections \ref{liminf} and \ref{limsup}, we have shown that the following inequality holds
		\begin{align}
		\limsup_{h_j\rightarrow 0} \mathcal{J}_{h_j}^{\natural}(\varphi_{h_j,\varepsilon}^\natural,\overline{Q}_{e,h_j,\varepsilon}^\natural)\leq \mathcal{J}_{0}(\varphi_0 ,\overline{Q}_{e,0})\leq\liminf_{h_j\rightarrow 0} \mathcal{J}_{h_j}^{\natural}(\varphi_{h_j,\varepsilon}^\natural,\overline{Q}_{e,h_j,\varepsilon}^\natural)\,,
		\end{align}
		which implies that $\mathcal{J}_{0}(\varphi_0 ,\overline{Q}_{e,0})$ is the $\Gamma$-$\lim$ of the sequence $\mathcal{J}_{h_j}^{\natural}(\varphi_{h_j,\varepsilon}^\natural,\overline{Q}_{e,h_j,\varepsilon}^\natural)$, i.e.,
		\begin{align}
		\mathcal{J}_{0}(\varphi_0 ,\overline{Q}_{e,0})=\Gamma\text{-}\lim(\mathcal{J}_{h_j}^{\natural}(\varphi_{h_j,\varepsilon}^\natural,\overline{Q}_{e,h_j,\varepsilon}^\natural))\,,\qquad m\equiv \varphi_{0}\,.
		\end{align}
		Remark \ref{rem}, shows that the family $(\mathcal{J}_{h_j}^\natural(\varphi^\natural,\overline{Q}_e^\natural)-{\Pi}_{h_j}^\natural(\varphi^\natural,\overline{Q}_e^\natural))_j$ is $\Gamma$-convergent (because the external load potential is continuous). This guarantees the existence of $\Gamma$-convergence for the family $(\mathcal{I}_{h_j}^\natural)_j$. Therefore, we may write
		\begin{align}
		\mathcal{I}_0(m,\overline{Q}_{e,0})=\Gamma\text{-}\lim \mathcal{I}_{h_j}^\natural(\varphi_{h_j,\varepsilon}^\natural,\overline{Q}_{e,h_j,\varepsilon}^\natural)=\mathcal{J}_0(m,\overline{Q}_{e,0})-\Pi(\varphi_0,\overline{Q}_{e,0})\,,\qquad m\equiv \varphi_{0}\,,
		\end{align}
		which is the desired formula.		
\end{proof}

 \section{Consistency with related shell and plate models}\label{comparison}
 \subsection{A comparison to the Cosserat flat shell $\Gamma$-limit}
 In this part we check whether our model is consistent with the Cosserat flat shell model obtained in \cite{neff2007geometrically1}. In the case of the plate model (flat initial configuration) we can assume that $\Theta(x_1,x_2,x_3)=(x_1,x_2,x_3)$ which gives $\nabla_x\Theta=\id_3$ and $y_0(x_1,x_2)=(x_1,x_2):={\text{id}}(x_1,x_2)$. Also $Q_0=\id_3$, $n_0=e_3$ and $\overline{Q}_{e,0}(x_1,x_2)=\overline{R}(x_1,x_2)$.
 
 The family of functionals \cite{Pietraszkiewicz-book04, Pietraszkiewicz10} coincide with that considered in the analysis of $\Gamma$-\,convergence for a flat referential configuration, while its descaled $\Gamma$-limit is
 	\begin{equation}
 \mathcal{J}_0  (m,\overline{R}) =
 \begin{cases}
 \dd\int_{\omega} h\,[ {W}_{\rm mp}^{{\rm hom}}(\mathcal{E}_{m,s}^{\rm plate})+\overline{W}_{\rm curv}^{\rm hom}(\mathcal{K}_{e,s}^{\rm plate})]\; d\omega\quad &\text{if}\quad (m,\overline{R})\in \mathcal{S}_\omega'\.,
 \\
 +\infty\qquad &\text{else in}\; X,
 \end{cases}
 \end{equation}
 where 
 \begin{align}
 \nonumber\mathcal{E}_{m,s}^{\rm plate}&=\overline{R}^{T}(\nabla m|0)-\id_2^\flat=\overline{R}^{T} (\nabla m|0)-\id_3+e_3\otimes e_3,\\
 \mathcal{K}_{e,s}^{\rm plate} &=\Big(\mathrm{axl}(\overline{Q}^{T}_{e,0}\,\partial_{x_1} \overline{Q}_{e,0})\,|\, \mathrm{axl}(\overline{Q}^{T}_{e,0}\,\partial_{x_2} \overline{Q}_{e,0})\,|0\Big)\not\in {\rm Sym}(3)\,,
 \end{align}
 and
 \begin{align}
\notag  {W}_{\rm mp}^{\rm hom}
 (\mathcal{E}_{m,s}^{\rm plate})&=
 \, \mu\,\lVert  \mathrm{sym}\,   \,[\mathcal{E}_{m,s}^{\rm plate}]^{\parallel}\rVert^2 +  \mu_{\rm c}\,\lVert \mathrm{skew}\,   \,[\mathcal{E}_{m,s}^{\rm plate}]^{\parallel}\rVert^2 +\,\dfrac{\lambda\,\mu\,}{\lambda+2\,\mu\,}\,\big[ \mathrm{tr}    ([\mathcal{E}_{m,s}^{\rm plate}]^{\parallel})\big]^2 +\frac{2\,\mu\, \,  \mu_{\rm c}}{\mu_c\,+\mu\,}\norm{[\mathcal{E}_{m,s}^{\rm plate}] ^T\,e_3}^2\\
 &=W_{\mathrm{shell}}\big( [\mathcal{E}_{m,s}^{\rm plate}]^{\parallel} \big)+\frac{2\,\mu\, \,  \mu_{\rm c}}{\mu_c\,+\mu\,}\lVert [\mathcal{E}_{m,s}^{\rm plate}]^{\perp}\rVert^2,\\ \overline{W}^{\rm hom}_{\rm curv}(\mathcal{K}_{e,s}^{\rm plate})
 &=\inf_{A\in \mathfrak{so}(3)}\overline{W}^*_{\rm curv}\Big(\mathrm{axl}(\overline{R}^{T}\,\partial_{\eta_1} \overline{R})\,|\, \mathrm{axl}(\overline{R}^{T}\,\partial_{\eta_2} \overline{R})\,|\,\axl(A)\,\Big),\notag
 \end{align}
together with
\begin{align}
[\mathcal{E}_{m,s}^{\rm plate}]^\parallel\coloneqq(\id_3-e_3\otimes e_3) \,[\mathcal{E}_{m,s}^{\rm plate}], \qquad \qquad [\mathcal{E}_{m,s}^{\rm plate}]^\perp\coloneqq(e_3\otimes e_3)\,[\mathcal{E}_{m,s}^{\rm plate}]\,,
\end{align}
where $W_{\mathrm{shell}}(X)=\mu\norm{\sym X}^2+\mu_c\norm{\skew X}^2+\frac{\lambda\mu}{\lambda+2\mu}[\tr(X)]^2.$
Let us denote by $\overline{R}_i$ the columns of the matrix $\overline{R}$, i.e., $\overline{R}=\big(\overline{R}_1\,|\,\overline{R}_2\,|\,\overline{R}_3\big)$, $\overline{R}_i=\overline{R}\, e_i$.
Since $(\id_3-e_3\otimes e_3)\overline{R}^{T}=\big(\overline{R}_1\,|\,\overline{R}_2\,|\,0\big)^T$, it follows that $ [\mathcal{E}_{m,s}^{\rm plate}]^\parallel=\big(\overline{R}_1\,|\,\overline{R}_2\,|\,0\big)^T(\nabla m|0)-\id_2^\flat=\big(\big(\overline{R}_1\,|\,\overline{R}_2\big)^T\,\nabla m\big)^\flat-\id_2^\flat$, while
\begin{align}
[\mathcal{E}_{m,s}^{\rm plate}]^\perp=(0\,|\,0\,|\,\overline{R}_3)^T(\nabla m|0)=\begin{footnotesize}
\begin{pmatrix}
0&0&0\\
0&0&0\\
\langle \overline{R}_3,\partial_{x_1}m\rangle&\langle \overline{R}_3,\partial_{x_2}m\rangle&0
\end{pmatrix}\,.
\end{footnotesize}
\end{align}

 Hence, in the Cosserat flat shell  model we have
 \begin{align}
  {W}_{\rm mp}^{\rm hom}
 (\mathcal{E}_{m,s}^{\rm plate})&=\mu\,\lVert \sym \big(\big(\overline{R}_1\,|\,\overline{R}_2\big)^T\,\nabla m-\id_2\big)\rVert^2+
 \mu_{\rm c}\lVert \skew \big(\big(\overline{R}_1\,|\,\overline{R}_2\big)^T\,\nabla m-\id_2\big)\rVert^2
\notag\\
&\quad  +
\frac{\lambda\,\mu\,}{\lambda+2\,\mu\,}[{\rm tr} \big(\big(\overline{R}_1\,|\,\overline{R}_2\big)^T\,\nabla m-\id_2\big)]^2+
\frac{2\,\mu\, \,  \mu_{\rm c}}{\mu_c\,+\mu\,}( \langle\overline{R}_3,\partial_{x_1}m\rangle^2+\langle\overline{R}_3,\partial_{x_2}m\rangle^2 )\,,
 \end{align}
 which agrees with the $\Gamma$-limit found in \cite{neff2007geometrically1}.

 \subsection{A comparison with the nonlinear derivation Cosserat shell model}
 
 In   \cite{GhibaNeffPartI}, under  assumptions \eqref{ch5in} upon the thickness by using the derivation approach, 
 the authors have obtained the following two-dimensional minimization problem   for the
 deformation of the midsurface $m:\omega
 \,{\to}\,
 \mathbb{R}^3$ and the microrotation of the shell
 $\overline{Q}_{e,s}:\omega
 \,{\to}\,
 \textrm{SO}(3)$ solving on $\omega
 \,\subset\mathbb{R}^2
 $: minimize with respect to $ (m,\overline{Q}_{e,s}) $ the  functional
 \begin{equation}\label{e89}
 I(m,\overline{Q}_{e,s})\!=\!\! \int_{\omega}   \!\!\Big[  \,
 W_{\mathrm{memb}}\big(  \mathcal{E}_{m,s}  \big) +W_{\mathrm{memb,bend}}\big(  \mathcal{E}_{m,s} ,\,  \mathcal{K}_{e,s} \big)   +
 W_{\mathrm{bend,curv}}\big(  \mathcal{K}_{e,s}    \big)
 \Big] \,\underbrace{{\rm det}(\nabla y_0|n_0)}_{{\rm det}\nabla\Theta \,}       \, d \omega \,,
 \end{equation}
 where the  membrane part $\,W_{\mathrm{memb}}\big(  \mathcal{E}_{m,s} \big) \,$, the membrane--bending part $\,W_{\mathrm{memb,bend}}\big(  \mathcal{E}_{m,s} ,\,  \mathcal{K}_{e,s} \big) \,$ and the bending--curvature part $\,W_{\mathrm{bend,curv}}\big(  \mathcal{K}_{e,s}    \big)\,$ of the shell energy density are given by
 \begin{align}\label{e90}
 W_{\mathrm{memb}}\big(  \mathcal{E}_{m,s} \big)=& \, \Big(h+{\rm K}\,\dfrac{h^3}{12}\Big)\,
 W_{\mathrm{shell}}\big(    \mathcal{E}_{m,s} \big),\vspace{2.5mm}\notag\\    
 W_{\mathrm{memb,bend}}\big(  \mathcal{E}_{m,s} ,\,  \mathcal{K}_{e,s} \big)=& \,   \Big(\dfrac{h^3}{12}\,-{\rm K}\,\dfrac{h^5}{80}\Big)\,
 W_{\mathrm{shell}}  \big(   \mathcal{E}_{m,s} \, {\rm B}_{y_0} +   {\rm C}_{y_0} \mathcal{K}_{e,s} \big)  \\&
 -\dfrac{h^3}{3} \mathrm{ H}\,\mathcal{W}_{\mathrm{shell}}  \big(  \mathcal{E}_{m,s} ,
 \mathcal{E}_{m,s}{\rm B}_{y_0}+{\rm C}_{y_0}\, \mathcal{K}_{e,s} \big)+
 \dfrac{h^3}{6}\, \mathcal{W}_{\mathrm{shell}}  \big(  \mathcal{E}_{m,s} ,
 ( \mathcal{E}_{m,s}{\rm B}_{y_0}+{\rm C}_{y_0}\, \mathcal{K}_{e,s}){\rm B}_{y_0} \big)\vspace{2.5mm}\notag\\&+ \,\dfrac{h^5}{80}\,\,
 W_{\mathrm{mp}} \big((  \mathcal{E}_{m,s} \, {\rm B}_{y_0} +  {\rm C}_{y_0} \mathcal{K}_{e,s} )   {\rm B}_{y_0} \,\big),  \vspace{2.5mm}\notag\\
 W_{\mathrm{bend,curv}}\big(  \mathcal{K}_{e,s}    \big) = &\,  \,\Big(h-{\rm K}\,\dfrac{h^3}{12}\Big)\,
 W_{\mathrm{curv}}\big(  \mathcal{K}_{e,s} \big)    +  \Big(\dfrac{h^3}{12}\,-{\rm K}\,\dfrac{h^5}{80}\Big)\,
 W_{\mathrm{curv}}\big(  \mathcal{K}_{e,s}   {\rm B}_{y_0} \,  \big)  + \,\dfrac{h^5}{80}\,\,
 W_{\mathrm{curv}}\big(  \mathcal{K}_{e,s}   {\rm B}_{y_0}^2  \big),\notag
 \end{align}
 where
 \begin{align}\label{quadraticforms}
 \nonumber W_{\mathrm{shell}}( X) & =   \mu\,\lVert  \mathrm{sym}\,X\rVert ^2 +  \mu_{\rm c}\lVert \mathrm{skew}\,X\rVert ^2 +\dfrac{\lambda\,\mu\,}{\lambda+2\,\mu\,}\,\big[ \mathrm{tr}   (X)\big]^2\,,\\
 &=  \mu\, \lVert  \mathrm{  dev \,sym} \,X\rVert ^2  +  \mu_{\rm c} \lVert  \mathrm{skew}   \,X\rVert ^2 +\,\dfrac{2\,\mu\,(2\,\lambda+\mu\,)}{3(\lambda+2\,\mu\,)}\,[\mathrm{tr}  (X)]^2\,,\\
 \mathcal{W}_{\mathrm{shell}}(  X,  Y)& =   \mu\,\bigl\langle  \mathrm{sym}\,X,\,\mathrm{sym}\,   \,Y \bigr\rangle   +  \mu_{\rm c}\bigl\langle \mathrm{skew}\,X,\,\mathrm{skew}\,   \,Y \bigr\rangle   +\,\dfrac{\lambda\,\mu\,}{\lambda+2\,\mu\,}\,\mathrm{tr}   (X)\,\mathrm{tr}   (Y),  \notag\vspace{2.5mm}\\
 W_{\mathrm{mp}}(  X)&= \mu\,\lVert \mathrm{sym}\,X\rVert ^2+  \mu_{\rm c}\lVert \mathrm{skew}\,X\rVert ^2 +\,\dfrac{\lambda}{2}\,\big[  \tr(X)\,\big]^2=
 \mathcal{W}_{\mathrm{shell}}(  X)+ \,\dfrac{\lambda^2}{2\,(\lambda+2\,\mu\,)}\,[\mathrm{tr} (X)]^2,\notag\vspace{2.5mm}\\
 W_{\mathrm{curv}}(  X )&=\mu\, L_{\rm c}^2 \left( b_1\,\lVert  \dev\,\text{sym} \,X\rVert ^2+b_2\,\lVert \text{skew}\,X\rVert ^2+4b_3\,
 [\tr (X)]^2\right), \quad \forall\, X,Y\in \mathbb{R}^{3\times 3}\,.\notag
 \end{align}
 
 In the formulation of the minimization problem, the {\it Weingarten map (or shape operator)} is defined by 
 $
 {\rm L}_{y_0}\,=\, {\rm I}_{y_0}^{-1} {\rm II}_{y_0}\in \mathbb{R}^{2\times 2},
 $
 where ${\rm I}_{y_0}:=[{\nabla  y_0}]^T\,{\nabla  y_0}\in \mathbb{R}^{2\times 2}$ and  ${\rm II}_{y_0}:\,=\,-[{\nabla  y_0}]^T\,{\nabla  n_0}\in \mathbb{R}^{2\times 2}$ are  the matrix representations of the {\it first fundamental form (metric)} and the  {\it  second fundamental form} of the surface, respectively.  
In that paper, the authors have also introduced the tensors defined by
 \begin{align}\label{AB}
 {\rm A}_{y_0}&:=(\nabla y_0|0)\,\,[\nabla\Theta_x(0) \,]^{-1}\in\mathbb{R}^{3\times 3}, \qquad \qquad 
 {\rm B}_{y_0}:=-(\nabla n_0|0)\,\,[\nabla\Theta_x(0) \,]^{-1}\in\mathbb{R}^{3\times 3},
 \end{align}
 and the so-called \textit{{alternator tensor}} ${\rm C}_{y_0}$ of the surface \cite{Zhilin06}
 \begin{align}
 {\rm C}_{y_0}:=\det(\nabla\Theta_x(0) \,)\, [	\nabla\Theta_x(0) \,]^{-T}\,\begin{footnotesize}\begin{pmatrix}
 0&1&0 \\
 -1&0&0 \\
 0&0&0
 \end{pmatrix}\end{footnotesize}\,  [	\nabla\Theta_x(0) \,]^{-1}.
 \end{align}
 
 Comparing with the $\Gamma$-limit obtained in the present paper,  the internal energy density obtained via the derivation approach depends also on 
 \begin{align}\label{EG}
 \mathcal{E}_{m,s} {\rm B}_{y_0}  + \mathrm{C}_{y_0} \mathcal{K}_{e,s} 
 	= &\, -[\nabla\Theta_x(0) \,]^{-T}
 	\begin{footnotesize}\left( \begin{array}{c|c}
 			\mathcal{R}-\mathcal{G} \,{\rm L}_{y_0} & 0 \vspace{4pt}\\
 			\mathcal{T} \,{\rm L}_{y_0} & 0
 		\end{array} \right)\end{footnotesize} [\nabla\Theta_x(0) \,]^{-1},
 	\end{align}
 where
 the nonsymmetric quantity $ \mathcal{R}-\mathcal{G} \,{\rm L}_{y_0}$ represents {\it the change of curvature tensor}. The choice of this name is justified subsequently in the framework of the linearized theory, see \cite{GhibaNeffPartIV,GhibaNeffPartV}. Let us notice that the elastic shell bending--curvature tensor $
\mathcal{K}_{e,s}$ appearing in the Cosserat $\Gamma$-limit is not capable to measure the change of curvature, see \cite{GhibaNeffPartIII,GhibaNeffPartIV,GhibaNeffPartV,GhibaNeffPartVI}, and that sometimes a confusion is made between bending and change of curvature measures, see also \cite{acharya2000nonlinear,vsilhavycurvature,anicic1999formulation,anicic2002mesure,anicic2003shell}
 
If we ignore the effect of the change of curvature tensor \eqref{EG} {in the model obtained  via the derivation approach,  there exists no coupling terms in $\mathcal{E}_{m,s}$ and $\mathcal{K}_{e,s}$ and we obtain a particular form of the energy,} i.e.,
\begin{align}\label{epart}
W_{\rm our}\big(  \mathcal{E}_{m,s},\mathcal{K}_{e,s}  \big)=\,&  \Big(h+{\rm K}\,\dfrac{h^3}{12}\Big)\,
W_{\mathrm{shell}}\big(    \mathcal{E}_{m,s} \big) + \Big(h-{\rm K}\,\dfrac{h^3}{12}\Big)\,
W_{\mathrm{curv}}\big(  \mathcal{K}_{e,s} \big),
\end{align}
where
\begin{align}\label{expW}
W_{\mathrm{shell}}\big(    \mathcal{E}_{m,s} \big)=&\, \mu\,\lVert  \mathrm{sym}\,   \,\mathcal{E}_{m,s}^{\parallel}\rVert^2 +  \mu_{\rm c}\,\lVert\mathrm{skew}\,   \,\mathcal{E}_{m,s}^{\parallel}\rVert^2 +\,\dfrac{\lambda\,\mu\,}{\lambda+2\,\mu\,}\,\big[ \mathrm{tr}     (\mathcal{E}_{m,s}^{\parallel})\big]^2+\frac{\mu\, +  \mu_{\rm c}}{2}\,\lVert \mathcal{E}_{m,s}^{\perp}\rVert^2\\
=&\, \mu\,\lVert  \mathrm{sym}\,   \,\mathcal{E}_{m,s}^{\parallel}\rVert^2 +  \mu_{\rm c}\,\lVert \mathrm{skew}\,   \,\mathcal{E}_{m,s}^{\parallel}\rVert^2 +\,\dfrac{\lambda\,\mu\,}{\lambda+2\,\mu\,}\,\big[ \mathrm{tr}    (\mathcal{E}_{m,s}^{\parallel})\big]^2+\frac{\mu\, +  \mu_{\rm c}}{2}\,\lVert \mathcal{E}_{m,s}^T\,n_0\rVert^2,\notag
\end{align}
and
\begin{align}\label{curvderi}
W_{\mathrm{curv}}( \mathcal{K}_{e,s} )=&\,\mu\, {L}_{\rm c}^2 \left( b_1\,\lVert  \text{sym} \,\mathcal{K}_{e,s}^{\parallel}\rVert^2+b_2\,\lVert \text{skew}\,\mathcal{K}_{e,s}^{\parallel}\rVert^2\
+\frac{12\,b_3-b_1}{3}\,
[\tr(\mathcal{K}_{e,s}^\parallel)]^2+ \frac{ b_1+b_2}{2}\,\lVert \mathcal{K}_{e,s}^\perp\rVert^2\right).
\end{align}

Skipping now all bending related $h^3$-terms we note that there is only one difference between the membrane  energy obtained via the derivation approach and the membrane energy obtained via $\Gamma$-convergence, i.e., the weight of the energy term $\lVert \mathcal{E}_{m,s}^T\,n_0\rVert^2$:
\begin{itemize}
	\item derivation approach: the algebraic mean of $\mu\,$ and $\mu_{\rm c}$, i.e., $\displaystyle\frac{\mu\, +  \mu_{\rm c}}{2}\,$;
		\item $\Gamma$-convergence: \quad \quad \ \ the harmonic mean of $\mu\,$ and $\mu_{\rm c}$, i.e., $\displaystyle\frac{2\,\mu\,   \mu_{\rm c}}{\mu\,+\mu_{\rm c}}\,$.
\end{itemize}
This difference has already been observed for the Cosserat flat shell $\Gamma$-limit \cite{neff2010reissner}.

We recall again the obtained curvature energy in \cite{Ghiba2022} as 
\begin{align}\label{curgam}
W_{\text{curv}}^{\text{hom}}(\mathcal{K}_{e,s})=\mu L_c^2\Big(b_1\norm{\sym\mathcal{K}_{e,s}^\parallel}^2+b_2\norm{\skew \mathcal{K}_{e,s}^\parallel}^2+\frac{b_1b_3}{(b_1+b_3)}\tr(\mathcal{K}_{e,s}^\parallel)^2+\frac{2\,b_1b_2}{b_1+b_2}\norm{\mathcal{K}_{e,s}^\perp}^2\Big)\,.
\end{align}
A comparison between (\ref{curvderi}) and (\ref{curgam}) shows that, like in the case for the membrane part, the weight of the energy term $\norm{\mathcal{K}_{e,s}^\perp}^2=\norm{\mathcal{K}_{e,s}^T n_0}^2
$ are different as following
\begin{itemize}
	\item derivation approach: the algebraic mean of $b_1$ and $b_2$, i.e., $\displaystyle\frac{b_1 +  b_2}{2}\,$;
	\item $\Gamma$-convergence: \quad \quad \ \ the harmonic mean of $b_1$ and $b_2$, i.e., $\displaystyle\frac{2\,b_1   b_2}{b_1+b_2}\,$.
\end{itemize}

  In the model obtained via the derivation approach \cite{GhibaNeffPartI}, the constitutive coefficients in the shell model depend on both the  Gau\ss \ curvature ${\rm K}$ and the mean curvature ${\rm H}$. In the approach presented in the current paper this does not occur. However, we will consider this aspect in forthcoming works, by considering the $\Gamma$-limit method in order to obtain higher order terms in terms of the thickness in the membrane energy, see \cite{friesecke2003derivation, friesecke2002foppl,friesecke2002theorem, friesecke2006hierarchy}.

\subsection{A comparison with the general 6-parameter shell model}
In the resultant 6-parameter theory of shells, the strain energy density for isotropic shells has been presented in various forms. The simplest expression $W_{\rm P}(\mathcal{E}_{m,s} ,\mathcal{K}_{e,s})$ has been proposed in the papers \cite{Pietraszkiewicz-book04,Pietraszkiewicz10} in the form
\begin{align}\label{56}
2\,W_{\rm P}(\mathcal{E}_{m,s} ,\mathcal{K}_{e,s})=& \;\quad C\big[\,\nu \,(\mathrm{tr}\, \mathcal{E}_{m,s}^{\parallel})^2 +(1-\nu)\, \mathrm{tr}((\mathcal{E}_{m,s}^{\parallel} )^T \mathcal{E}_{m,s}^{\parallel} )\big]  + \alpha_{s\,}C(1-\nu) \, \lVert  \mathcal{E}_{m,s}^T  {n}_0\rVert^2\notag \\
&+\,D \big[\,\nu\,(\mathrm{tr}\, \mathcal{K}_{e,s}^{\parallel})^2 + (1-\nu)\, \mathrm{tr}((\mathcal{K}_{e,s}^{\parallel} )^T \mathcal{K}_{e,s}^{\parallel} )\big]  + \alpha_{t\,}D(1-\nu) \,     \lVert \mathcal{K}_{e,s}^T  {n}_0\rVert^2\,,
  \end{align}
with the Poisson ratio $\nu=\frac{\lambda}{2(\mu+\lambda)}$.

In \cite{Eremeyev06}, Eremeyev and Pietraszkiewicz have proposed a more general form of the strain energy density, namely
\begin{align}\label{58}
2\, W_{\rm EP}(\mathcal{E}_{m,s} ,\mathcal{K}_{e,s})=& \quad \alpha_1\,\big(\mathrm{tr} \, \mathcal{E}_{m,s}^{\parallel}\big)^2 +\alpha_2 \, \mathrm{tr} \big(\mathcal{E}_{m,s}^{\parallel}\big)^2    + \alpha_3 \,\mathrm{tr}\big((\mathcal{E}_{m,s}^{\parallel})^T \mathcal{E}_{m,s}^{\parallel} \big)  + \alpha_4  \,   \lVert  \mathcal{E}_{m,s}^T  {n}_0\rVert^2 \notag\\
& + \beta_1\,\big(\mathrm{tr}\,  \mathcal{K}_{e,s}^{\parallel}\big)^2 +\beta_2\,  \mathrm{tr} \big(\mathcal{K}_{e,s}^{\parallel}\big)^2    + \beta_3\, \mathrm{tr}\big((\mathcal{K}_{e,s}^{\parallel})^T \mathcal{K}_{e,s}^{\parallel} \big)  + \beta_4 \,    \lVert  \mathcal{K}_{e,s}^T  {n}_0\rVert^2.
\end{align}
Already, note the absence of coupling terms involving $\mathcal{K}_{e,s}^{\parallel}$ and $\mathcal{E}_{m,s}^{\parallel}$. The eight coefficients $\alpha_k\,$, $\beta_k$ ($k=1,2,3,4$) can depend in general on the structure of the curvature tensor  $\mathcal{K}^0\,=\, Q_0 ( \, \text{axl}(Q_{0}^T\partial_{x_1} Q_{0})  \, | \,  \text{axl}(Q_{0}^T\partial_{x_2} Q_{0}) \, |\, \,0 \, )[\nabla \Theta(0)]^{-1}$ of the curved reference configuration. We can decompose the strain energy density  \eqref{58}  in the  in-plane part  \linebreak $W_{\rm \text{plane}-EP}(\mathcal{E}_{m,s})$ and the curvature part $W_{\rm \text{curv}-EP}(\mathcal{K}_{e,s})$ and write their expressions  in  the form
\begin{align}\label{59}
W_{\rm EP}(\mathcal{E}_{m,s} ,\mathcal{K}_{e,s})=& \,W_{\rm \text{plane}-EP}(\mathcal{E}_{m,s})+ W_{\rm \text{curv}-EP}(\mathcal{K}_{e,s})\,,
\\
2\,W_{\rm \text{plane}-EP}(\mathcal{E}_{m,s})=&\, (\alpha_2\!+\!\alpha_3) \,\lVert  \text{sym}\, \mathcal{E}_{m,s}^{\parallel}\rVert^2\!+ (\alpha_3\!-\!\alpha_2)\,\lVert \text{skew}\, \mathcal{E}_{m,s}^{\parallel}\rVert^2\!+ \alpha_1\,\big(\mathrm{tr} ( \mathcal{E}_{m,s}^{\parallel})\big)^2 +  \alpha_4\,\lVert \mathcal{E}_{m,s}^Tn_0  \rVert^2, \vspace{4pt}\notag\\
2\, W_{\rm \text{curv}-EP}(\mathcal{K}_{e,s})=&\, (\beta_2\!+\!\beta_3) \,\lVert  \text{sym}\, \mathcal{K}_{e,s}^{\parallel}\rVert^2\!+ (\beta_3\!-\!\beta_2)\,\lVert \text{skew}\, \mathcal{K}_{e,s}^{\parallel}\rVert^2\!+ \beta_1\,\big(\mathrm{tr} ( \mathcal{K}_{e,s}^{\parallel})\big)^2
+ \beta_4\,\lVert \mathcal{K}_{e,s}^Tn_0  \rVert^2.\notag
\end{align}

By comparing our membrane energy
\begin{align}\label{hominf_2Drw}
 {W}_{\rm mp}^{\rm hom}
(\mathcal{E}_{m,s})&=
\, \mu\,\lVert  \mathrm{sym}\,   \,\mathcal{E}_{m,s}^{\parallel}\rVert^2 +  \mu_{\rm c}\,\lVert \mathrm{skew}\,   \,\mathcal{E}_{m,s}^{\parallel}\rVert^2 +\,\dfrac{\lambda\,\mu\,}{\lambda+2\,\mu\,}\,\big[ \mathrm{tr}    (\mathcal{E}_{m,s}^{\parallel})\big]^2 +\frac{2\,\mu\, \,  \mu_{\rm c}}{\mu_c\,+\mu\,}\norm{\mathcal{E}_{m,s}^Tn_0}^2\\
&=W_{\mathrm{shell}}\big(   \mathcal{E}_{m,s}^{\parallel} \big)+\frac{4\,\mu\, \,  \mu_{\rm c}}{\mu_c\,+\mu\,}\lVert \mathcal{E}_{m,s}^{\perp}\rVert^2,\notag
\end{align} 
with $W_{\rm EP}\big(  \mathcal{E}_{m,s},\mathcal{K}_{e,s}  \big)$ we deduce the following identification of the constitutive coefficients $\alpha_1\,,...,\alpha_4$
\begin{align}
\alpha_1&=h\,\dfrac{2\,\mu\,\lambda}{2\,\mu\,+\lambda}\,,\quad \qquad\alpha_2=h\,(\mu\,-\mu_{\rm c}),\qquad \alpha_3=h\,(\mu\,+\mu_{\rm c}),\quad \qquad\alpha_4=h\frac{2\,\mu\, \,  \mu_{\rm c}}{\mu\,+\mu_c\,}.\notag
\end{align} 
We observe that 
$
\mu_{\rm c}^{\mathrm{drill}}\,:=\,\alpha_3-\alpha_2\,=\,2\,h\,\mu_{\rm c}\,,
$
which means that the in-plane rotational couple modulus $\,\mu_{\rm c}^{\mathrm{drill}}\,$ of the Cosserat shell model is determined by the Cosserat couple modulus $\,\mu_{\rm c}\,$ of the 3D Cosserat material. {An analogous conclusion is given in \cite{AltenbachEremeyev} where linear deformations
	are considered.}

Now a comparison between our curvature energy 
\begin{align}\label{curgam1}
W_{\text{curv}}^{\text{hom}}(\mathcal{K}_{e,s})=\mu L_c^2\Big(b_1\norm{\sym\mathcal{K}_{e,s}^\parallel}^2+b_2\norm{\skew \mathcal{K}_{e,s}^\parallel}^2+\frac{b_1b_3}{(b_1+b_3)}\tr(\mathcal{K}_{e,s}^\parallel)^2+\frac{2\,b_1b_2}{b_1+b_2}\norm{\mathcal{K}_{e,s}^\perp}^2\Big)\,.
\end{align}
and $W_{\rm \text{curv}-EP}(\mathcal{K}_{e,s})$, leads us to the identification of the constitutive coefficients $\beta_1,\cdots,\beta_4$
\begin{align*}
\beta_1=2\.\mu\. L_c^2\frac{b_1b_3}{b_1+b_3}\,,\quad\qquad \beta_2=\mu\. L_c^2b_1\,,\quad\qquad \beta_3=\mu \.L_c^2(b_1+b_2)\,,\quad\qquad \beta_4=4\.\mu\. L_c^2\frac{b_1b_2}{b_1+b_2}\,.
\end{align*}

\subsection{A comparison to another $O(h^5)$-Cosserat shell model}

	In \cite{birsan2020derivation}, by using a method which extends the reduction procedure from classical elasticity to  the case of Cosserat shells, B\^{\i}rsan has obtained a minimization problem, which for the particular case of a quadratic ansatz for the deformation map and skipping higher order terms is based on the following energy
	\begin{align}
	I(m,\overline{Q}_{e,s})\!=\!\! \int_{\omega}   \!\!\Big[  \,
W_{\mathrm{memb,bend}}^{\rm (quad)}\big(  \mathcal{E}_{m,s} ,\,  \mathcal{K}_{e,s} \big)   +
	W_{\mathrm{bend,curv}}\big(  \mathcal{K}_{e,s}    \big)
	\Big] \,{{\rm det}(\nabla y_0|n_0)}       \, d \omega \,,
	\end{align}
	with $W_{\mathrm{memb,bend}}^{\rm (quad)}\big(  \mathcal{E}_{m,s} ,\,  \mathcal{K}_{e,s} \big)= \, h\,
	W_{\mathrm{Coss}}\big(\mathcal{E}_{m,s} \big)$ and	$W_{\mathrm{bend,curv}}\big(  \mathcal{K}_{e,s}\big) = \,  \,h\,
	W_{\mathrm{curv}}\big(  \mathcal{K}_{e,s} \big)$, where
	\begin{align}
	\nonumber W_{\textrm{Coss}}(X)&= \mathcal{W}_{\textrm{Coss}}(X,X)= \mu\,\lVert  \mathrm{sym}   \,X^\parallel\rVert^2 +  \mu_{\rm c}\,\lVert \mathrm{skew}  \,X^\parallel\rVert^2+\,   \frac{2\mu\,\mu_{\rm c}}{\mu\,+\mu_{\rm c}}\,\lVert X^\perp\rVert^2 +\,\dfrac{\lambda\,\mu\,}{\lambda+2\,\mu\,}\,\big[ \mathrm{tr}   (X)\big]^2\,,\\
	\nonumber \mathcal{W}_{\textrm{Coss}}(X,Y)&=  \mu\,\bigl\langle \mathrm{sym}   \,X^\parallel,\mathrm{sym}  Y^\parallel\bigr\rangle  +  \mu_{\rm c}\,\bigl\langle  \mathrm{skew}  \,X^\parallel,\mathrm{skew} Y^\parallel\bigr\rangle +\,   \frac{2\mu\,\mu_{\rm c}}{\mu\,+\mu_{\rm c}}\,\bigl\langle X^\perp,Y^\perp\bigr\rangle  +\,\dfrac{\lambda\,\mu\,}{\lambda+2\,\mu\,}\, \mathrm{tr}   (X)\,\mathrm{tr}   (Y)\,,\\
	W_{\mathrm{mp}}(  X)&= \mu\,\lVert \mathrm{sym}\,X\rVert ^2+  \mu_{\rm c}\lVert \mathrm{skew}\,X\rVert ^2 +\,\dfrac{\lambda}{2}\,\big[  \tr(X)\,\big]^2=
	\mathcal{W}_{\mathrm{shell}}( X, X)+ \,\dfrac{\lambda^2}{2\,(\lambda+2\,\mu\,)}\,[\mathrm{tr} (X)]^2,\notag\vspace{2.5mm}\\
	W_{\mathrm{curv}}(  X )&=\mu\, L_{\rm c}^2 \left( b_1\,\lVert  \dev\,\text{sym} \,X\rVert ^2+b_2\,\lVert \text{skew}\,X\rVert ^2+4\.b_3\,
	[\tr (X)]^2\right), \quad \forall\, X,Y\in \mathbb{R}^{3\times 3}\,.\notag
	\end{align}

	As it can be seen, in the obtained model by B\^{\i}rsan, there are some coupled terms of stress tensor and bending-curvature tensor, too. 
	This is not surprising, since B\^{\i}rsan has obtained the starting example from the model in \cite{GhibaNeffPartI}. The main difference, in comparison to the model obtained in  \cite{GhibaNeffPartI} is that
	\begin{align*}
	W_{\textrm{Coss}}(X)= \mathcal{W}_{\textrm{Coss}}(X,X)= \mu\,\lVert  \mathrm{sym}   \,X^\parallel\rVert^2 +  \mu_{\rm c}\,\lVert \mathrm{skew}  \,X^\parallel\rVert^2+\,   \frac{2\mu\,\mu_{\rm c}}{\mu\,+\mu_{\rm c}}\,\lVert X^\perp\rVert^2 +\,\dfrac{\lambda\,\mu\,}{\lambda+2\,\mu\,}\,\big[ \mathrm{tr}   (X)\big]^2\,,
	\end{align*}
	from \cite{GhibaNeffPartI} is replaced by 
	\begin{align} 
	\mathcal{W}_{\text{Coss}}(X,Y):=W_{\text{shell}}(X^\parallel,Y^\parallel)+\frac{2\,\mu\,\mu_c}{\mu+\mu_c}\bigl\langle X^\perp, Y^\perp\bigr\rangle\,,
	\end{align}
	 for all tensors $\,   X,\,    Y\in \mathbb{R}^{3\times 3}$ of the form $(*|*|0)\cdot[	\nabla_x \Theta(0)]^{-1}$.
	Note that
		\begin{align} 
	\mathcal{W}_{\text{shell}}(X,Y):=W_{\text{shell}}(X^\parallel,Y^\parallel)+\frac{\mu+\mu_c}{2}\bigl\langle X^\perp, Y^\perp\bigr\rangle\,,
	\end{align}
	holds true for all tensors $\,   X,\,    Y\in \mathbb{R}^{3\times 3}$ of the form $(*|*|0)\cdot[	\nabla_x \Theta(0)]^{-1}$.
	Hence, for 
	this type of tensors we have
	\begin{align} 
	\mathcal{W}_{\text{Coss}}(X,Y):=\mathcal{W}_{\text{shell}}(X,Y)-\frac{\mu+\mu_c}{2}\bigl\langle X^\perp, Y^\perp\bigr\rangle+\frac{2\,\mu\,\mu_c}{\mu+\mu_c}\bigl\langle X^\perp, Y^\perp\bigr\rangle\,.
	\end{align}
	
	The main point of the comparison presented in this subsection is that the membrane term of order $O(h)$ coincide with the homogenized membrane energy determined by us in the present paper, i.e.,
	\begin{align}
	{W}_{\rm mp}^{\rm hom}
	(\mathcal{E}_{m,s})\equiv {W}_{\text{Coss}}(\mathcal{E}_{m,s}).
	\end{align}
	
	With a small comparison between the obtained membrane energy via $\Gamma$-convergence and the one obtained via the derivation approach model by B\^{\i}rsan, obviously we see that for a $O(h)$-Cosserat shell theory, there is no difference between the coefficients, i.e., \\
	\begin{itemize*}
		\item special derivation approach: the harmonic mean of $\mu$ and $\mu_c$; $\displaystyle \frac{2\mu\.\mu_c}{\mu+\mu_c}$\,,\\
		\item $\Gamma$-limit approach:  \;\;\qquad\qquad the harmonic mean of $\mu$ and $\mu_c$; $\displaystyle\frac{2\mu\.\mu_c}{\mu+\mu_c}$\,.\\
	\end{itemize*}

 \section{Linearisation of the $\Gamma$-limit Cosserat membrane shell model}

\subsection{The linearised model}

In this section we develop the linearization  of the $\Gamma$-limit functional for the elastic Cosserat shell model, i.e., for situations of small midsurface deformations and small Cosserat-curvature change. 
Let us consider
\begin{align}
m(x_1,x_2)=y_0(x_1,x_2)+v(x_1,x_2),
\end{align}
where $v:\omega\to \mathbb{R}^3$ is the infinitesimal shell-midsurface displacement. For the rotation tensor $ \overline{Q}_{e,0}\in\rm{SO}(3) $ there exists a skew-symmetric matrix  \begin{align}
\overline{A}_\vartheta:={\rm Anti}(\vartheta_1,\vartheta_2,\vartheta_3):=\begin{footnotesize}
\begin{pmatrix}
0&-\vartheta_3&\vartheta_2\\
\vartheta_3&0&-\vartheta_1\\
-\vartheta_2&\vartheta_1&0
\end{pmatrix}\end{footnotesize}\in \mathfrak{so}(3), \quad \qquad {\rm Anti}:\mathbb{R}^3\to \mathfrak{so}(3),
\end{align}
where $ \vartheta={\rm axl}( \overline{A}_\vartheta) $ denotes the axial vector of $ \overline{A}_\vartheta $, such that $\overline{Q}_{e,0}:=\exp(\overline{A}_\vartheta)\;= \;\sum_{k=0}^{\infty} \frac{1}{k!} \,\overline{A}_\vartheta^k\; = \;\id_3 + \overline{A}_\vartheta+\textrm{h.o.t.}$
The tensor field $\overline{A}_\vartheta$ is   the infinitesimal microrotation. Here, ``h.o.t'' stands for terms of  higher order than linear with respect to $u$ and $\overline{A}_\vartheta$.  

Using these linearisations  of the kinematic variables, we find the linearisations of the strain tensors. Indeed, since
\begin{equation}\label{equ1}
\overline{Q}_{e,0}^T\nabla m-\nabla y_0 = (\id_3 +\overline{A}_\vartheta^T+\textrm{h.o.t.} )(\nabla v + \nabla y_0) -\nabla y_0 
=\nabla v - \overline{A}_\vartheta\nabla y_0+\textrm{h.o.t.},
\end{equation}
we get for the non-symmetric  \textit{shell strain tensor} (which characterises both the in-plane deformation and the transverse shear deformation)
\begin{equation*} 
\mathcal{E}_{m,s} = ( \overline{Q}_{e,0}^T\nabla m-\nabla y_0\; |\; 0)\; [\nabla\Theta \,]^{-1}\,,
\end{equation*}
the linearization
\begin{align}\label{equ2}
\nonumber
\mathcal{E}_{m,s}^{\rm{lin}} = ( \nabla v - \overline{A}_\vartheta\nabla y_0\; |\; 0)\; [\nabla\Theta \,]^{-1}= (\partial_{x_1} u - \vartheta\times a_1\;|\;  \partial_{x_2} u - \vartheta\times a_2 \;|\; 0)\; [\nabla\Theta \,]^{-1}\;\not\in {\rm Sym}(3).
\end{align}
And for the \textit{shell bending-curvature tensor}
\begin{align} \mathcal{K}_{e,s} := \Big(\mbox{axl}(\overline{Q}_{e,0}^T\partial_{x_1}\overline{Q}_{e,0})\,|\, \mbox{axl}(\overline{Q}_{e,0}^T\partial_{x_2}\overline{Q}_{e,0})\,|\, 0 \Big) \; [\nabla\Theta \,]^{-1}\,,\end{align} 
we calculate
\begin{equation} 
\overline{Q}_{e,0}^T\partial_{x_\alpha}\overline{Q}_{e,0}= (\id_3 - \overline{A}_\vartheta )\,\,\partial_{x_\alpha}\overline{A}_\vartheta+\textrm{h.o.t.}
= \partial_{x_\alpha}\overline{A}_\vartheta +\textrm{h.o.t.} = \underbrace{\overline{A}_{\partial_{x_\alpha}\vartheta}}_{\equiv \;{\rm Anti} \,\partial_{x_\alpha}\vartheta\.=\,\partial_{x_\alpha} {\rm Anti} \, \vartheta}+\,\textrm{h.o.t.}\;,
\end{equation}
i.e.,
\begin{equation}\label{equ6}
\mbox{axl}\big(\overline{Q}_{e,0}^T\partial_{x_\alpha}\overline{Q}_{e,0}\big)= \partial_{x_\alpha} \vartheta+\textrm{h.o.t.}\,,
\end{equation}
and we deduce
\begin{equation}\label{equ71}
\mathcal{K}_{e,s}^{\rm{lin}} \,\,= \,\, (\mbox{axl}\big(\partial_{x_1}\overline{A}_\vartheta\big) \,|\, \mbox{axl}\big(\partial_{x_2}\overline{A}_\vartheta\big)\, |\, 0) \; [\nabla\Theta \,]^{-1}\,,
\end{equation}
together with 
\begin{equation}\label{equ7}
\mathcal{K}_{e,s}^{\rm{lin}} \,\,= \,\, (\partial_{x_1}\vartheta \,|\, \partial_{x_2}\vartheta\, |\, 0) \; [\nabla\Theta \,]^{-1} \,\,= \,\, (\nabla\vartheta\, |\, 0) \; [\nabla\Theta \,]^{-1}\;.
\end{equation}
The form of the energy density remains unchanged upon linearization, since the model is physically linear. Thus, 
the linearization of the  $\Gamma$-limits reads: for a midsurface displacement vector field 
$v:\omega\subset\mathbb{R}^2\to\mathbb{R}^3$ and the micro-rotation vector field $\vartheta:\omega\subset\mathbb{R}^2\to\mathbb{R}^3$:
\begin{align}\label{e89l}
\mathcal{J}_0(m,\overline{Q}_{e,0})\!=\!\! \int_{\omega}   h \,\Big[  &
\overline {W}_{\mathrm{\text{mp}}}^{\text{hom}}\big(    \mathcal{E}_{m,s}^{\rm{lin}}  \big)+ \,
\overline{W}_{\mathrm{curv}}^{{\rm hom}}\big(  \mathcal{K}_{e,s}^{\rm{lin}}  \big) \Big]   {\rm det}(\nabla y_0|n_0)  \,   d \omega\notag- {\overline{\Pi}}^{\rm lin}(u,\vartheta)\,,\notag
\end{align}
where
\begin{align}
\nonumber \overline{W}_{\rm mp}^{\rm hom}
(   \mathcal{E}_{m,s}^{\rm{lin}})&=
\, \mu\,\lVert  \mathrm{sym}\,   \,   \mathcal{E}_{m,s}^{\rm{lin},\parallel}\rVert^2 +  \mu_{\rm c}\,\lVert \mathrm{skew}\,   \,\mathcal{E}_{m,s}^{\rm{lin},\parallel}\rVert^2 +\,\dfrac{\lambda\,\mu\,}{\lambda+2\,\mu\,}\,\big[ \mathrm{tr}    (\mathcal{E}_{m,s}^{\rm{lin},\parallel})\big]^2 +\frac{2\,\mu\, \,  \mu_{\rm c}}{\mu_c\,+\mu\,}\norm{\mathcal{E}_{m,s}^{\rm{lin},T}n_0}^2\\
&=W_{\mathrm{shell}}\big( \mathcal{E}_{m,s}^{\rm{lin},\parallel} \big)+\frac{2\,\mu\, \,  \mu_{\rm c}}{\mu_c\,+\mu\,}\lVert\mathcal{E}_{m,s}^{\rm{lin},\perp}\rVert^2,\\ \overline{W}^{\rm hom}_{\rm curv}(\mathcal{K}_{e,s}^{\rm{lin}})
\nonumber&=\mu L_c^2\Big(b_1\norm{\sym\mathcal{K}_{e,s}^{\rm{lin},\parallel}}^2+b_2\norm{\skew \mathcal{K}_{e,s}^{\rm{lin},\parallel}}^2+\frac{b_1b_3}{(b_1+b_3)}\tr(\mathcal{K}_{e,s}^{\rm{lin},\parallel})^2+\frac{2\,b_1b_2}{b_1+b_2}\norm{\mathcal{K}_{e,s}^{\rm{lin},\perp}}\Big)\,,
\end{align}
and $\overline{\Pi}^{\rm lin}(u,\vartheta)$ is the linearization of the continuous external loading potential $\overline{\Pi}$. 

\subsection{Comparison with the linear Reissner-Mindlin membrane-bending model}
The following model
\begin{align}\label{Reissner}
\nonumber\int_\omega h\Big(\mu\,&\norm{\sym\nabla(v_1,v_2)}^2+\frac{\kappa\.\mu\,}{2}\norm{\nabla v_3-\matr{\theta_1\\\theta_2}}^2+\frac{\mu\,\lambda}{2\,\mu\,+\lambda}\tr(\sym\nabla(v_1,v_2))^2\Big)\\&\quad+\frac{h^3}{12}\Big(\mu\,\norm{\sym\nabla(\theta_1,\theta_2)}^2+\frac{\mu\,\lambda}{2\,\mu\,+\lambda}\tr(\nabla(\theta_1,\theta_2))^2\Big)d\omega \rightarrow \min\;\text{w.r.t}. (v,\theta)\,,\\
\nonumber& \quad v|_{\gamma_0}=u^d(x,y,0)\,,\quad-\theta|_{\gamma_0}=(u_{1,z}^d,u_{2,z}^d,0)^T\,,
\end{align}
is the linear Reissner-Mindlin membrane-bending model which has five degree of freedom, three from the midsurface displacement $v\col \omega\subset \R^2\to \R^3$ and the other two are from the out-of-plane rotation parameter $\theta\col \omega \to \R^2$ that describes the infinitesimal increment of the director and $0<\kappa\leq 1$ is the so called \textit{shear correction factor}. In this model the drill rotations (rotations about the normal) are absent. 

As derived in \cite{neff2007geometrically}, the Reissner-Mindlin membrane-bending model can be obtained as $\Gamma$-limit of the linear Cosserat elasticity model. Neff et al. in \cite{neff2010reissner} applied the nonlinear scaling for the displacement  and linear scaling for the infinitesimal microrotation for the minimization problem with respect to $(u,\overline{A})$:
\begin{align}
I(u,\overline{A})=\int_{\Omega_h}W_\text{mp}(\overline{\varepsilon})+W_{\text{curv}}(\nabla\axl\overline{A})\;dV
\mapsto \min \quad\text{w.r.t}\quad (u,\overline{A})\,,
\end{align}
where $\overline{\varepsilon}=\nabla u-\overline{A}$, 
and
\begin{align}
\nonumber W_{\text{mp}}(\overline{\varepsilon})&=\mu\,\norm{\sym \overline{\varepsilon}}^2+\mu_c\,\norm{\skew \overline{\varepsilon}}^2+\frac{\lambda}{2}[\tr(\overline{\varepsilon})]^2\,,\\
W_{\text{curv}}(\mathcal{A})&=\mu\,\frac{\widehat{L}_c^2(h)}{2}\Big(\alpha_1\norm{\sym\nabla\axl\overline{A}}^2+\alpha_2\norm{\skew \nabla\axl\overline{A}}^2+\frac{\alpha_3}{2}[\tr(\nabla\axl\overline{A})]^2\Big)\,,
\end{align}
for $\alpha_1,\alpha_2,\alpha_3\geq 0$. Then, they obtained the following minimization problem: 
\begin{align}
I^{\text{hom}}(v,\overline{A})=\int_\omega W_{\text{mp}}^{\text{hom}}(\nabla v,\axl \overline{A})+W_{\text{curv}}^{\text{hom}}(\nabla\axl\overline{A})
\;d\omega\,,
\end{align}
with respect to $(v,\theta)$, where $v\col\omega\subset \R^2\to \R^3$ is the deformation of the midsurface and $\overline{A}\col \omega\subset \R^2\to \so(3)$ as the infinitesimal microrotation of the plate on $\omega$ with the boundary condition $v|_{\gamma_0}=u_d(x,y,0), \gamma_0\subset \partial\omega$ and
\begin{align}\label{Neff hom}
\nonumber W_{\text{mp}}^{\text{hom}}(\nabla v,\theta)&:=\mu\,\norm{\sym \nabla_{(\eta_1,\eta_2)}(v_1,v_2)}^2+2\frac{\mu\,\mu_c\,}{\mu\,+\mu_c\,}\norm{\nabla_{(\eta_1,\eta_2)}v_3-\matr{-\theta_2\\\theta_1}}^2+\frac{\mu\,\lambda}{2\,\mu\,+\lambda}\tr[\nabla_{(\eta_1,\eta_2)}(v_1,v_2)]^2\,,\\
W_{\text{curv}}^{\text{hom}}(\nabla\theta)&:=\mu\,\frac{\widehat{L}_c^2(h)}{2}\Big(\alpha_1\norm{\sym\nabla_{(\eta_1,\eta_2)}(\theta_1,\theta_2)}^2+\frac{\alpha_1\alpha_3}{2\alpha_1+\alpha_3}\tr[\nabla_{(\eta_1,\eta_2)}(\theta_1,\theta_2)]^2\Big)\,.
\end{align}

Comparing the Reissner-Mindlin membrane-bending model with the linearisation of the $\Gamma$-model obtained in the present paper,    it can be seen that the Reissner-Mindlin model is obtained by $\Gamma$-convergence, upon selecting $\alpha_1=\mu\,, \alpha_3=\lambda$ in our model and by neglecting the drilling (the third component of the  director).

In this formula one can recognize the harmonic mean $\mathcal{H}$
\begin{align}
\frac{1}{2}\mathcal{H}(\mu\,,\frac{\lambda}{2})=\frac{\mu\,\lambda}{2\,\mu\,+\lambda}\,,\qquad\mathcal{H}(\mu\,,\mu_c\,)=\frac{2\,\mu\,\mu_c\,}{\mu\,+\mu_c\,}\,,\qquad\frac{1}{2}\mathcal{H}(\alpha_1,\frac{\alpha_3}{2})=\frac{\alpha_1\alpha_3}{2\alpha_1+\alpha_3}\,.
\end{align}
In our paper we used the nonlinear scaling for both deformation and microrotation, while in \cite{neff2010reissner}, they applied linear scaling for microrotation and nonlinear scaling for deformation.
The other comparison is regarding the th elastic shell strain tensor and elastic shell bending curvature tensor which in our model are not de-coupled, and in (\ref{Neff hom}) the in-plane deflections $v_1,v_2$ are not decoupled from $\theta_3$ as well. 

\subsection{Aganovic and Neff's flat shell model}\label{AgaNeff}
Aganovi\'{c} et al.\cite{aganovic2007} proposed a linear Cosserat flat shell model based on asymptotic analysis of the linear isotropic micropolar Cosserat model. They used the nonlinear scaling for both the displacement and infinitesimal microrotations. Therefore, their minimization problem reads:
\begin{align}
\nonumber\int_\omega h&\Big(\mu\,\norm{\sym\big(\nabla(v_1,v_2)-\matr{0&-\theta_3\\\theta_3&0}\big)}^2+\mu_c\,\norm{\skew\big(\nabla(v_1,v_2)-\matr{0&-\theta_3\\\theta_3&0}\big)}^2+\footnotesize\frac{2\,\mu\,\mu_c\,}{\mu\,+\mu_c\,}\norm{\nabla v_3-\matr{-\theta_2\\\theta_1}}^2\\
&\quad+\frac{\mu\,\lambda}{2\,\mu\,+\lambda}\tr(\sym\big(\nabla(v_1,v_2)-\matr{0&-\theta_3\\\theta_3&0}\big)^2\Big)\\
\nonumber&\quad+\mu\,\frac{h \,L_c^2}{2}\Big(\alpha_1\norm{\sym\nabla(\theta_1,\theta_2)}^2+\alpha_2\norm{\skew \nabla(\theta_1,\theta_2)}^2+\frac{2\alpha_1\alpha_2}{\alpha_1+\alpha_2}\norm{\nabla\theta_3}^2+\frac{\alpha_1\alpha_3}{2\alpha_1+\alpha_3}\tr(\nabla(\theta_1,\theta_2))^2\Big)\,d\omega\\
\nonumber&\quad
\rightarrow \min\;\text{w.r.t}. (v,\theta)\,, 
\end{align}
where it is assumed that $\alpha_2,\kappa>0$, otherwise this model with the assumption $\alpha_2=0$ will give the Reissner-Mindlin model. This means that we can not ignore the in-plane drill component $\theta_3$ here and in the case of $\alpha_2>0$ one does not obtain the Reissner-Mindlin model.
The asymptotic model coincides with the assumptions of Neff et al. in \cite{neff2007geometrically1}, where their assumption was about scaling the nonlinear Cosserat plate model with nonlinear scaling for both deformation and microrotation. The membrane part of this energy coincides with the homogenized membrane energy of our model with the same coefficients.

 \section{Conclusion}
 In this paper we have considered the $\Gamma$-limit procedure in order to derive a Cosserat thin shell model having a curved reference configuration. The paper is based on the development in \cite{neff2007geometrically1}, where the $\Gamma$-limit was obtained for a flat reference configuration of the shell. Here, the major complication arises from the curvy shell reference configuration. By introducing suitable mappings, we can encode the "curvy" information on a fictitious flat reference configuration. There, we use the nonlinear scaling for both the nonlinear deformation and the microrotation. This leads to a Cosserat membrane model, in which the effect of Cosserat-curvature survive the $\Gamma$-limit procedure. The homogenized membrane and curvature energy expressions are made explicit after some lengthy technical calculations. This is only possible because we use a physically linear, isotropic Cosserat model. Since the limit equations are obtained by $\Gamma$- convergence, they are automatically well-posed. We finally compare the Cosserat membrane shell model with some other dimensionally reduced proposals and linearizations. The full regularity of weak solutions for this Cosserat shell model (for some choice of constitutive parameters) will be established in \cite{Gastel2022}.
 \bigskip
 
 \footnotesize
 \begin{footnotesize}
	\noindent{\bf Acknowledgements:}   This research has been funded by the Deutsche Forschungsgemeinschaft (DFG, German Research Foundation) -- Project no. 415894848: NE 902/8-1 (P. Neff and M. Mohammadi Saem).
 \end{footnotesize}
 \section*{References}
\printbibliography[heading=none]
 	\addcontentsline{toc}{section}{References}
\begin{appendix}
	\section{Appendix}
	\subsection{An auxiliary optimization problem}\label{Calculhom} 
	In this section we solve the auxiliary optimization problem (\ref{hom inf}).   We calculate the variation of the energy (\ref{hom inf}) at equilibrium to be minimized over $c\in \mathbb{R}^3$ in order to   determine the minimizer $d^*$. 
	For arbitrary increment $\delta d^*\in \mathbb{R}^3$, we have 
	\begin{align}
	\forall\;\; \delta d^*\in \R^3: \quad\bigl\langle \D {W}_{\rm mp}(\overline{Q}_{e}^{\natural,T}(\nabla_{(\eta_1,\eta_2)} \varphi^\natural|d^*)[(\nabla_x\Theta)^\natural ]^{-1}), \overline{Q}_{e}^{\natural,T}(0|0|\delta d^*)[(\nabla_x\Theta)^\natural ]^{-1}\bigr\rangle&=0.
	\end{align}

	By applying $\D {W}_{\rm mp}$ we obtain
	\begin{align}
	\nonumber&\bigl\langle 2\,\mu\,\Big(\sym (\overline{Q}_{e}^{\natural,T}(\nabla_{(\eta_1,\eta_2)} \varphi^\natural|d^*)[(\nabla_x\Theta)^\natural ]^{-1}-\id_3)\Big), \overline{Q}_{e}^{\natural,T}(0|0|\delta d^*)[(\nabla_x\Theta)^\natural ]^{-1}\bigr\rangle_{\R^{3\times 3}}\\
	&\quad+\bigl\langle 2\,\mu_c\,\Big(\skew(\overline{Q}_{e}^{\natural,T}(\nabla_{(\eta_1,\eta_2)} \varphi^\natural|d^*)[(\nabla_x\Theta)^\natural ]^{-1})\Big), \overline{Q}_{e}^{\natural,T}(0|0|\delta d^*)[(\nabla_x\Theta)^\natural ]^{-1}\bigr\rangle_{\R^{3\times 3}}\\
	&\quad+\lambda\tr\Big(\sym(\overline{Q}_{e}^{\natural,T}(\nabla_{(\eta_1,\eta_2)} \varphi^\natural|d^*)[(\nabla_x\Theta)^\natural ]^{-1}-\id_3)\Big)\, \iprod{\id_3, \overline{Q}_{e}^{\natural,T}(0|0|\delta d^*)[(\nabla_x\Theta)^\natural ]^{-1}}_{\R^{3\times 3}}=0.\nonumber 
	\end{align}
	This is equivalent to
	\begin{align}
	&\bigl\langle 2\,\mu\,\overline{Q}_{e}^\natural\Big(\sym (\overline{Q}_{e}^{\natural,T}(\nabla_{(\eta_1,\eta_2)} \varphi^\natural|d^*)[(\nabla_x\Theta)^\natural ]^{-1}-\id_3)\Big)[(\nabla_x\Theta)^\natural ]^{-T}e_3, \delta d^*\bigr\rangle_{\R^3}\notag\\
	&\quad\quad+\bigl\langle 2\,\mu_c\,\overline{Q}_{e}^\natural\Big(\skew(\overline{Q}_{e}^{\natural,T}(\nabla_{(\eta_1,\eta_2)} \varphi^\natural|d^*)[(\nabla_x\Theta)^\natural ]^{-1})\Big)[(\nabla _x\Theta)^\natural]^{-T}e_3,\delta d^*\bigr\rangle_{\R^3}\\
	&\quad\quad+\lambda\tr\Big(\sym(\overline{Q}_{e}^{\natural,T}(\nabla_{(\eta_1,\eta_2)} \varphi^\natural|d^*)[(\nabla_x\Theta)^\natural ]^{-1}-\id_3)\Big)\,\iprod{\overline{Q}_{e}^\natural[(\nabla_x\Theta)^\natural ]^{-T}e_3, \delta d^*}_{\R^3}=0\,,\notag
	\end{align}
	and it gives
	\begin{align}
	\nonumber & \bigl\langle 2\,\mu\,\overline{Q}_{e}^\natural\Big(\sym (\overline{Q}_{e}^{\natural,T}(\nabla_{(\eta_1,\eta_2)} \varphi^\natural|d^*)[(\nabla_x\Theta)^\natural ]^{-1}-\id_3)\Big)n_0, \delta d^*\bigr\rangle_{\R^3}\\
	&\qquad+\dyniprod{2\,\mu_c\,\overline{Q}_{e}^\natural\Big(\skew(\overline{Q}_{e}^{\natural,T}(\nabla_{(\eta_1,\eta_2)} \varphi^\natural|d^*)[(\nabla_x\Theta)^\natural ]^{-1})\Big)n_0,\delta d^*}_{\R^3}\\
	&\qquad+\lambda\tr\Big(\sym(\overline{Q}_{e}^{\natural,T}(\nabla_{(\eta_1,\eta_2)} \varphi^\natural|d^*)[(\nabla_x\Theta)^\natural ]^{-1}-\id_3)\Big)\iprod{\overline{Q}_{e}^\natural n_0, \delta d^*}_{\R^3}=0.\notag
	\end{align}
	Recall that the \textit{first Piola--Kirchhoff stress tensor} in the reference configuration $\Omega_\xi$ is given by $S_1(F_\xi,\overline{R}_\xi):=\D_{F_\xi} {W}_{\rm mp}(F_\xi,\overline{R}_\xi)$, while
	the $\textit{Biot-type stress tensor}$ is   $T_{\text{Biot}}(\overline{U}_\xi):=\D_{\overline{U}_\xi}W_{\rm mp}(\overline{U}_\xi)$. Since $\D_{F_\xi}\overline{U}_\xi \..\.X=\overline{R}^T_\xi X$ and
	$$\iprod{\D_{F_\xi} {W}_{\rm mp}(F_\xi,\overline{R}_\xi),X}=\iprod{\D_{\overline{U}_\xi}W_{\rm mp}(\overline{U}_\xi),\D_{F_{\xi}} \overline{U}_{\xi} X}\.,\ \forall X\in \mathbb{R}^{3\times 3}\,,$$
	we obtain 
	\begin{align}
	\D_{F_{\xi}} {W}_{\rm mp}(F_{\xi},\overline{R}_{\xi})=\overline{R}_{\xi}\,\D_{\overline{U}_{\xi}}W_{\rm mp}(\overline{U}_{\xi})\,.
	\end{align} Therefore, $S_1(F_\xi,\overline{R}_\xi)=\overline{R}_\xi\, T_{\text{Biot}}(\overline{U}_\xi)$ and $ T_{\text{Biot}}(\overline{U}_\xi)=\overline{R}^T_\xi\, S_1(F_\xi,\overline{R}_\xi)$. 
	Here, we have
	\begin{align}
	T_{\text{Biot}}(\overline{U}_\xi)=2\,\mu\,\sym(\overline{U}_\xi-\id_3)+2\,\mu_c\,\skew
	(\overline{U}_\xi-\id_3)+\lambda\tr(\sym(\overline{U}_\xi-\id_3))\id_3\,,
	\end{align}
	where $\overline{U}_{\xi}(\Theta(x_1,x_2,x_3))=\UE(x_1,x_2,x_3)$.
	Thus, we can express the first Piola Kirchhoff stress tensor 
	\begin{align}
	S_1(F_\xi,\overline{R}_\xi)=\overline{R}_\xi\Big[2\,\mu\,\sym(\overline{R}_\xi^T F_\xi-\id_3)+2\,\mu_c\,\skew
	(\overline{R}_\xi^T F_\xi-\id_3)+\lambda\tr(\sym(\overline{R}_\xi^T F_\xi-\id_3))\id_3\Big]\.,
	\end{align}
	with $\overline{R}_\xi(\Theta(x_1,x_2,x_3))=\overline{Q}_e(x_1,x_2,x_3)$ for the elastic microrotation $\overline{Q}_e\col\Omega_h\rightarrow \text{SO(3)}$.
	Hence, we must have
	\begin{align}
	\forall \delta d^*\in \R^3\col \qquad\iprod{S_1((\nabla_{(\eta_1,\eta_2)} \varphi^\natural|d^*)&[(\nabla_x\Theta)^\natural ]^{-1},\overline{Q}_{e}^\natural)n_0,\delta d^*}_{\R^3}=0,
	\end{align}
	implying
	\begin{align}\label{bcot}
	S_1((\nabla_{(\eta_1,\eta_2)} \varphi^\natural|d^*)&[(\nabla_x\Theta)^\natural ]^{-1},\overline{Q}_{e}^\natural)\,n_0=0 \qquad \forall\, \eta_3\in \left[-\frac{1}{2},\frac{1}{2}\right].
	\end{align}
	\\
	In shell theories, the usual assumption is   that the normal stress on the transverse boundaries are vanishing, that is
	\begin{align}\label{zeronor}
	S_1(F_\xi,\overline{R}_\xi)\big|_{\omega_\xi^\pm}\, (\pm n_0)=0\,, \qquad \text{(normal stress on lower and upper faces is zero)}\,.
	\end{align}
	We notice that the condition \eqref{bcot} is for all $\eta_3\in \left[-\frac{1}{2},\frac{1}{2}\right]$, while the condition \eqref{zeronor} is only for $\eta_3=\pm\frac{1}{2}$. Therefore, it is possible that the Cosserat-membrane type $\Gamma$-limit underestimates the real stresses (e.g., the transverse shear stresses).
	From the relation between the first Piola-Kirchhoff tensor and the Biot-stress tensor we obtain 
	\begin{align}\label{multi n0} T_{\text{Biot}}\Big(\overline{Q}_{e}^{\natural,T}(\nabla_{(\eta_1,\eta_2)} \varphi^\natural|d^*)[(\nabla_x\Theta)^\natural ]^{-1}\Big) n_0=0\,,\qquad \forall \,\eta_3\in [-\frac{1}{2},\frac{1}{2}]\,,
	\end{align}
	or, equivalently,  
	\begin{align}\label{multi n00}
	T_{\text{Biot}}(\overline{U}_{\varphi^\natural,\overline{Q}_{e}^{\natural},d^*})\,n_0=0,
	\end{align}
	where
	\begin{align}\label{Tbiot}
	T_{\text{Biot}}(\overline{U}_{\varphi^\natural,\overline{Q}_{e}^{\natural},d^*})=2\,\mu\,\sym(\overline{U}_{\varphi^\natural,\overline{Q}_{e}^{\natural},d^*}-\id_3)+2\,\mu_c\,\skew(\overline{U}_{\varphi^\natural,\overline{Q}_{e}^{\natural},d^*}-\id_3)+\lambda\tr(\sym(\overline{U}_{\varphi^\natural,\overline{Q}_{e}^{\natural},d^*}-\id_3))\id_3\,,
	\end{align}
	and we have introduced the notation $\overline{U}_{\varphi^\natural,\overline{Q}_{e}^{\natural},d^*}:=\overline{Q}_{e}^{\natural,T}(\nabla_{(\eta_1,\eta_2)} \varphi^\natural|d^*)[(\nabla_x\Theta)^\natural ]^{-1}$.
	With the help of the following decomposition
	\begin{align}
	\overline{U}_{\varphi^\natural,\overline{Q}_{e}^{\natural},d^*}-\id_3&=(\overline{Q}_{e}^{\natural,T}\nabla_{(\eta_1,\eta_2)} \varphi^\natural-(\nabla y_0)^\natural |0)[(\nabla_x\Theta)^\natural]^{-1}+(0|0|\overline{Q}_{e}^{\natural,T}d^*-n_0)[(\nabla_x\Theta)^\natural]^{-1}\notag\\&=  \mathcal{E}_{\varphi^\natural,\overline{Q}_{e}^{\natural} } +(0|0|\overline{Q}_{e}^{\natural,T}d^*-n_0)[(\nabla_x\Theta)^\natural]^{-1}\,,
	\end{align}
	with $\mathcal{E}_{\varphi^\natural,\overline{Q}_{e}^{\natural} }=(\overline{Q}_{e}^{\natural,T}\nabla_{(\eta_1,\eta_2)} \varphi^\natural-(\nabla y_0)^\natural |0)[(\nabla_x\Theta)^\natural]^{-1}$, and relations (\ref{Tbiotsym})-(\ref{Tbiottra}), the relation (\ref{Tbiot}) can be expressed as
	\begin{align}
	\nonumber T_{\text{Biot}}(\overline{U}_{\varphi^\natural,\overline{Q}_{e}^{\natural},d^*})n_0&=\mu\,\Big(  \mathcal{E}^T_{\varphi^\natural,\overline{Q}_{e}^{\natural} } n_0+(\overline{Q}_{e}^{\natural,T}d^*-n_0)+[(\nabla_x\Theta)^\natural]^{-T}(0|0|\overline{Q}_{e}^{\natural,T}d^*-n_0)^Tn_0\Big)\\
	\nonumber&\quad +\mu_c\,\Big(-  \mathcal{E}^T_{\varphi^\natural,\overline{Q}_{e}^{\natural} } n_0+(\overline{Q}_{e}^{\natural,T}d^*-n_0)-[(\nabla_x\Theta)^\natural]^{-T}(0|0|\overline{Q}_{e}^{\natural,T}d^*-n_0)^Tn_0\Big)\\
	\nonumber&\quad+\lambda\Big(\iprod{  \mathcal{E}_{\varphi^\natural,\overline{Q}_{e}^{\natural} } ,\id_3}n_0+(\overline{Q}_{e}^{\natural,T}d^*-n_0)n_0\otimes n_0\Big)\\
	\nonumber&= (\mu\,+\mu_c\,)(\overline{Q}_{e}^{\natural,T}d^*-n_0)+(\mu\,-\mu_c\,)  \mathcal{E}^T_{\varphi^\natural,\overline{Q}_{e}^{\natural} } n_0+(\mu\,-\mu_c\,)((0|0|\overline{Q}_{e}^{\natural,T}d^*-n_0)[(\nabla_x\Theta)^\natural]^{-1})^Tn_0\\
	&\qquad+\lambda\tr(  \mathcal{E}_{\varphi^\natural,\overline{Q}_{e}^{\natural} } )n_0+\lambda(\overline{Q}_{e}^{\natural,T}d^*-n_0)n_0\otimes n_0,
	\end{align}
	and the condition \eqref{multi n00} on $T_{\text{Biot}}$ reads
	\begin{align}\label{result}
	\nonumber(\mu\,+\mu_c\,)(\overline{Q}_{e}^{\natural,T}d^*-n_0)+(\mu\,-\mu_c\,)(\overline{Q}_{e}^{\natural,T}d^*-n_0)n_0\otimes n_0&+\lambda(\overline{Q}_{e}^{\natural,T}d^*-n_0)n_0\otimes n_0\\
	&=-\Big[(\mu\,-\mu_c\,)\mathcal{E}_{\varphi^\natural,\overline{Q}_{e}^{\natural} }^Tn_0+\lambda\tr(  \mathcal{E}_{\varphi^\natural,\overline{Q}_{e}^{\natural} } )n_0\Big],
	\end{align}
	where $((0|0|\overline{Q}_{e}^{\natural,T}d^*-n_0)[(\nabla_x\Theta)^\natural]^{-1})^Tn_0=(\overline{Q}_{e}^{\natural,T}d^*-n_0)n_0\otimes n_0$. Before  continuing the calculations, we introduce the tensor
	\begin{align}\label{AB}
	{\rm A}_{y_0}&:=(\nabla y_0|0)\,\,[(\nabla_x\Theta)(0) \,]^{-1}=\id_3-n_0\otimes n_0\in{\rm Sym}(3),
	\end{align}
	and we notice that, identically as in the proof of Lemma 4.3 in \cite{GhibaNeffPartI}, we can show that
	\begin{align}\label{identitatiPI}\mathcal{E}_{\varphi^\natural,\overline{Q}_{e}^{\natural} } {\rm A}_{y_0}\,=\,\mathcal{E}_{\varphi^\natural,\overline{Q}_{e}^{\natural} } \qquad \Longleftrightarrow \qquad \mathcal{E}_{\varphi^\natural,\overline{Q}_{e}^{\natural} } n_0\otimes n_0=0.
	\end{align}
	Actually, for an arbitrary matrix $X\,=\,(*|*|0)\,[	\nabla_x \Theta(0)]^{-1}$, since 
	${\rm A}_{y_0}^2={\rm A}_{y_0}\in\mathrm{Sym}(3)$ and $X{\rm A}_{y_0}=X$, we have 
	\begin{align*}
	\bigl\langle  (\id_3-{\rm A}_{y_0}) \,X ,  {\rm A}_{y_0} \,X\bigr\rangle = \bigl\langle  ({\rm A}_{y_0}-{\rm A}_{y_0}^2) \,X ,  \,X\bigr\rangle =0,
	\end{align*}
	but also
	\begin{align}
	(\id_3-{\rm A}_{y_0}) \,X^T=\big(X(\id_3-{\rm A}_{y_0})\big)^T=\big(X-X{\rm A}_{y_0}\big)^T=0,
	\end{align}
	and consequently
	\begin{align*}
	\bigl\langle  X^T (\id_3-{\rm A}_{y_0}) , {\rm A}_{y_0} \,X\bigr\rangle  = 0 \qquad \text{as well as} \qquad
	\bigl\langle  X^T (\id_3-{\rm A}_{y_0}) ,(\id_3-{\rm A}_{y_0}) \,X\bigr\rangle  = 0.
	\end{align*}
	In addition,  since 
	${\rm A}_{y_0}\,=\,\id_3-(0|0|n_0)\,(0|0|n_0)^T=\,\id_3-n_0\otimes n_0$,  the following equalities holds
	\begin{align}\label{EKperp}
	\lVert (\id_3-{\rm A}_{y_0})\,X\rVert^2&= \bigl\langle  \,X,(\id_3-{\rm A}_{y_0})^2\,X\bigr\rangle = \bigl\langle  \,X,(\id_3-{\rm A}_{y_0})\,X\bigr\rangle = \bigl\langle  \,X,(0|0|n_0)\,(0|0|n_0)^T\,X\bigr\rangle \notag\\&= \bigl\langle  \,(0|0|n_0)^T X,(0|0|n_0)^T\,X\bigr\rangle =\lVert X\,(0|0|n_0)^T\rVert^2
	=\lVert X^T\,(0|0|n_0)\rVert^2=\lVert X^T\,n_0\rVert^2.
	\end{align}

	We have the following decomposition 
	\begin{align}
	(\overline{Q}_{e}^{\natural,T}d^*-n_0)&=\id_3(\overline{Q}_{e}^{\natural,T}d^*-n_0)=(A_{y_0}+n_0\otimes n_0)(\overline{Q}_{e}^{\natural,T}d^*-n_0)\notag\\
	&= A_{y_0}(\overline{Q}_{e}^{\natural,T}d^*-n_0)+n_0\otimes n_0(\overline{Q}_{e}^{\natural,T}d^*-n_0).
	\end{align}
	By using that 
	\begin{align}
	n_0\otimes n_0 (\overline{Q}_{e}^{\natural,T}d^*-n_0)=n_0\iprod{n_0,(\overline{Q}_{e}^{\natural,T}d^*-n_0)}=\iprod{(\overline{Q}_{e}^{\natural,T}d^*-n_0),n_0}n_0=(\overline{Q}_{e}^{\natural,T}d^*-n_0) n_0\otimes n_0,
	\end{align}
	and with (\ref{result}), we get
	\begin{align}
	\nonumber(\mu\,+\mu_c\,)A_{y_0}(\overline{Q}_{e}^{\natural,T}d^*-n_0)+(\mu\,&+\mu_c\,)n_0\otimes n_0(\overline{Q}_{e}^{\natural,T}d^*-n_0)+(\mu\,-\mu_c\,)n_0\otimes n_0(\overline{Q}_{e}^{\natural,T}d^*-n_0)\\
	\quad&+\lambda\, n_0\otimes n_0(\overline{Q}_{e}^{\natural,T}d^*-n_0)=-\Big[(\mu\,-\mu_c\,) \mathcal{E}^T_{\varphi^\natural,\overline{Q}_{e}^{\natural} } n_0+\lambda\tr(  \mathcal{E}_{\varphi^\natural,\overline{Q}_{e}^{\natural} } )n_0\Big].
	\end{align}
	Therefore,
	\begin{align}\label{invers}
	\Big((\mu\,&+\mu_c\,)A_{y_0}+(2\,\mu\,+\lambda)n_0\otimes n_0\Big)(\overline{Q}_{e}^{\natural,T}d^*-n_0)=-\Big[(\mu\,-\mu_c\,)  \mathcal{E}^T_{\varphi^\natural,\overline{Q}_{e}^{\natural} } n_0+\lambda\tr(  \mathcal{E}_{\varphi^\natural,\overline{Q}_{e}^{\natural} } )n_0\Big].
	\end{align}
	Direct calculation shows
	\begin{align}
	\Big((\mu\,+\mu_c\,)A_{y_0}+(2\,\mu\,+\lambda)n_0\otimes n_0\Big)^{-1}:=\Big(\frac{1}{\mu\,+\mu_c\,}A_{y_0}+\frac{1}{2\,\mu\,+\lambda}n_0\otimes n_0\Big)\,.
	\end{align}
	Next, by using
	\begin{align}
	A_{y_0}n_0&=(\id_3-n_0\otimes n_0)n_0=n_0-n_0\iprod{n_0,n_0}=n_0-n_0=0,\notag\\
	n_0\otimes n_0  \,\mathcal{E}^T_{\varphi^\natural,\overline{Q}_{e}^{\natural} }  n_0 &= (0|0|n_0)(0|0|n_0)^T  \mathcal{E}^T_{\varphi^\natural,\overline{Q}_{e}^{\natural} }  n_0=(0|0|n_0)\Big((\overline{Q}_{e}^{\natural,T}\nabla_{(\eta_1,\eta_2)} \varphi^\natural-(\nabla y_0)^\natural |0)[(\nabla_x\Theta)^\natural]^{-1}(0|0|n_0)\Big)^T n_0\notag\\
	&=(0|0|n_0)\Big((\overline{Q}_{e}^{\natural,T}\nabla_{(\eta_1,\eta_2)} \varphi^\natural-(\nabla y_0)^\natural |0)(0|0|e_3)\Big)^T n_0=0\,,
	\end{align}
	eq. (\ref{invers}) can be written as
	\begin{align}\label{unknown}
	\overline{Q}_{e}^{\natural,T}d^*-n_0&=-\Big[\frac{1}{\mu\,+\mu_c\,}A_{y_0}+\frac{1}{2\,\mu\,+\lambda}n_0\otimes n_0\Big]\times\Big[(\mu\,-\mu_c\,)  \mathcal{E}^T_{\varphi^\natural,\overline{Q}_{e}^{\natural} } n_0+\lambda\tr(  \mathcal{E}_{\varphi^\natural,\overline{Q}_{e}^{\natural} } )n_0\Big]\\
	\nonumber&=-\Big[\frac{\mu\,-\mu_c\,}{\mu\,+\mu_c\,}A_{y_0}  \mathcal{E}^T_{\varphi^\natural,\overline{Q}_{e}^{\natural} } n_0+\frac{\mu\,-\mu_c\,}{2\,\mu\,+\lambda}\,n_0\otimes n_0\, \mathcal{E}^T_{\varphi^\natural,\overline{Q}_{e}^{\natural} } n_0+\frac{\lambda}{\mu\,+\mu_c\,}\tr(  \mathcal{E}_{\varphi^\natural,\overline{Q}_{e}^{\natural} } )A_{y_0}n_0\\
	\nonumber&\quad\qquad+\frac{\lambda}{2\,\mu\,+\lambda}\tr(  \mathcal{E}_{\varphi^\natural,\overline{Q}_{e}^{\natural} } )(n_0\otimes n_0) n_0\Big]=-\Big[\frac{\mu\,-\mu_c\,}{\mu\,+\mu_c\,}A_{y_0}  \mathcal{E}^T_{\varphi^\natural,\overline{Q}_{e}^{\natural} } n_0+\frac{\lambda}{2\,\mu\,+\lambda}\tr(  \mathcal{E}_{\varphi^\natural,\overline{Q}_{e}^{\natural} } )n_0\Big]\,.
	\end{align}

	Simplifying (\ref{unknown}) we obtain
	\begin{align}
	\nonumber d^*&=\Big(1-\frac{\lambda}{2\,\mu\,+\lambda}\iprod{  \mathcal{E}_{\varphi^\natural,\overline{Q}_{e}^{\natural} } ,\id_3}\Big) \overline{Q}_{e}^\natural n_0+\frac{\mu_c\,-\mu\,}{\mu_c\,+\mu\,}\;\overline{Q}_{e}^\natural  \mathcal{E}_{\varphi^\natural,\overline{Q}_{e}^{\natural} }^T n_0.
	\end{align}
	In terms of $\overline{Q}_{e}^\natural=\overline{R}^\natural Q_0^{\natural,T}$ we obtain the following expression for $d^*$
	\begin{align}\label{formula ba}
	\nonumber d^*&=\Big(1-\frac{\lambda}{2\,\mu\,+\lambda}\iprod{(Q_0^{\natural}\overline{R}^{\natural,T}\nabla_{(\eta_1,\eta_2)} \varphi^\natural-(\nabla y_0)^\natural |0)[(\nabla_x\Theta)^\natural]^{-1},\id_3}\Big) \overline{R}^{\natural}Q_0^{\natural,T}n_0\\
	&\qquad+\frac{\mu_c\,-\mu\,}{\mu_c\,+\mu\,}\;\overline{R}^{\natural}Q_0^{\natural,T}\Big((Q_0^{\natural}\overline{R}^{\natural,T}\nabla_{(\eta_1,\eta_2)} \varphi^\natural-(\nabla y_0)^\natural |0)[(\nabla_x\Theta)^\natural]^{-1}\Big)^Tn_0.
	\end{align}

	\subsection{Calculations for the $T_{\text{Biot}}$ stress }
	Here we present the lengthy calculation related to the $T_{\text{Biot}}$ stress tensor in expression (\ref{Tbiot}). We have
	\begin{align}\label{Tbiotsym}
	\nonumber 2\sym (\overline{U}_{\varphi^\natural,\overline{Q}_{e}^{\natural},d^*}&-\id_3)n_0= \Big(2\sym(  \mathcal{E}_{\varphi^\natural,\overline{Q}_{e}^{\natural} } )+2\sym((0|0|\overline{Q}_{e}^{\natural,T}d^*-n_0)[(\nabla_x\Theta)^\natural]^{-1})\Big)n_0\\
	\nonumber&= \Big(  \mathcal{E}_{\varphi^\natural,\overline{Q}_{e}^{\natural} } +  \mathcal{E}^T_{\varphi^\natural,\overline{Q}_{e}^{\natural} } \Big)n_0+\Big((0|0|\overline{Q}_{e}^{\natural,T}d^*-n_0)[(\nabla_x\Theta)^\natural]^{-1}+[(\nabla_x\Theta)^\natural]^{-T}(0|0|\overline{Q}_{e}^{\natural,T}d^*-n_0)^T\Big)n_0\\
	\nonumber&=\underbrace{  \mathcal{E}_{\varphi^\natural,\overline{Q}_{e}^{\natural} }  n_0}_{=0}+  \mathcal{E}^T_{\varphi^\natural,\overline{Q}_{e}^{\natural} } n_0+(0|0|\overline{Q}_{e}^{\natural,T}d^*-n_0)[(\nabla_x\Theta)^\natural]^{-1}n_0+[(\nabla_x\Theta)^\natural]^{-T}(0|0|\overline{Q}_{e}^{\natural,T}d^*-n_0)^Tn_0\\
	&=  \mathcal{E}^T_{\varphi^\natural,\overline{Q}_{e}^{\natural} } n_0+(0|0|\overline{Q}_{e}^{\natural,T}d^*-n_0)e_3+[(\nabla_x\Theta)^\natural]^{-T}(0|0|\overline{Q}_{e}^{\natural,T}d^*-n_0)^Tn_0\\
	\nonumber&=  \mathcal{E}^T_{\varphi^\natural,\overline{Q}_{e}^{\natural} } n_0+(\overline{Q}_{e}^{\natural,T}d^*-n_0)+[(\nabla_x\Theta)^\natural]^{-T}(0|0|\overline{Q}_{e}^{\natural,T}d^*-n_0)^Tn_0,
	\end{align}
	and
	\begin{align}\label{Tbiotskew}
	\nonumber 2\skew (\overline{U}_{\varphi^\natural,\overline{Q}_{e}^{\natural},d^*}&-\id_3)n_0= \Big(2\skew(  \mathcal{E}_{\varphi^\natural,\overline{Q}_{e}^{\natural} } )+2\skew((0|0|\overline{Q}_{e}^{\natural,T}d^*-n_0)[(\nabla_x\Theta)^\natural]^{-1})\Big)n_0\\
	\nonumber&= \Big(  \mathcal{E}_{\varphi^\natural,\overline{Q}_{e}^{\natural} } -  \mathcal{E}^T_{\varphi^\natural,\overline{Q}_{e}^{\natural} } \Big)n_0+\Big((0|0|\overline{Q}_{e}^{\natural,T}d^*-n_0)[(\nabla_x\Theta)^\natural]^{-1}-[(\nabla_x\Theta)^\natural]^{-T}(0|0|\overline{Q}_{e}^{\natural,T}d^*-n_0)^T\Big)n_0\\
	&=-  \mathcal{E}^T_{\varphi^\natural,\overline{Q}_{e}^{\natural} } n_0+(\overline{Q}_{e}^{\natural,T}d^*-n_0)-[(\nabla_x\Theta)^\natural]^{-T}(0|0|\overline{Q}_{e}^{\natural,T}d^*-n_0)^Tn_0.
	\end{align}
	Calculating the trace of $T_{\text{Biot}}$ gives 
	\begin{align}\label{Tbiottra}
	\nonumber\tr(\sym (\overline{U}_{\varphi^\natural,\overline{Q}_{e}^{\natural},d^*}-\id_3))n_0&=\iprod{\sym(\overline{U}_{\varphi^\natural,\overline{Q}_{e}^{\natural},d^*}-\id_3),\id_3}n_0
	=\Big(\iprod{  \mathcal{E}_{\varphi^\natural,\overline{Q}_{e}^{\natural} } ,\id_3}+\iprod{(0|0|\overline{Q}_{e}^{\natural,T}d^*-n_0)[(\nabla_x\Theta)^\natural]^{-1},\id_3}\Big)n_0\\
	&=\iprod{  \mathcal{E}_{\varphi^\natural,\overline{Q}_{e}^{\natural} } ,\id_3}n_0+(\overline{Q}_{e}^{\natural,T}d^*-n_0)n_0\otimes n_0,
	\end{align}
	where we have used that $\iprod{(0|0|\overline{Q}_{e}^{\natural,T}d^*-n_0)[(\nabla_x\Theta)^\natural]^{-1},\id_3}_{\R^{3\times 3}}\,n_0
	=\iprod{(\overline{Q}_{e}^{\natural,T}d^*-n_0),n_0}_{\R^3}\,n_0=(\overline{Q}_{e}^{\natural,T}d^*-n_0)\.n_0\otimes n_0$.
		
	\subsection{Calculations for the homogenized membrane energy}
	In this part we do the calculations for obtaining the minimizer separately. By inserting $d^*$ in the membrane part of the relation (\ref{boun co}), we have
	\begin{align}\label{reduced}
	\nonumber \norm{\sym (\UHN-\id_3)}^2&=\norm{\sym(\overline{Q}_{e}^{\natural,T}(\nabla_{(\eta_1,\eta_2)} \varphi^\natural|d^*)[(\nabla_x\Theta)^\natural ]^{-1}-\id_3)}^2\\\notag &=\lVert\sym\Big(\underbrace{\overline{Q}_{e}^{\natural,T}(\nabla_{(\eta_1,\eta_2)} \varphi^\natural-[\nabla y_0]^\natural |0)[(\nabla_x\Theta)^\natural ]^{-1}}_{=  \mathcal{E}_{\varphi^\natural,\overline{Q}_{e}^{\natural} } } +(0|0|\overline{Q}_{e}^{\natural,T}d^*-n_0)[(\nabla_x\Theta)^\natural ]^{-1}\Big))\rVert^2\\
	\nonumber&=\norm{\sym  \mathcal{E}_{\varphi^\natural,\overline{Q}_{e}^{\natural} } }^2+ \norm{\sym((0|0|\overline{Q}_{e}^{\natural,T}d^*-n_0)[(\nabla_x\Theta)^\natural ]^{-1})}^2\\
	\nonumber &\quad+2\dyniprod{\sym  \mathcal{E}_{\varphi^\natural,\overline{Q}_{e}^{\natural} } ,\sym((0|0|\overline{Q}_{e}^{\natural,T}d^*-n_0)[(\nabla_x\Theta)^\natural ]^{-1})}\\
	\nonumber &= \norm{\sym  \mathcal{E}_{\varphi^\natural,\overline{Q}_{e}^{\natural} } }^2+\norm{\sym\Big(\frac{\mu_c\,-\mu\,}{\mu_c\,+\mu\,} \mathcal{E}_{\varphi^\natural,\overline{Q}_{e}^{\natural} } ^Tn_0\otimes n_0-\frac{\lambda}{2\,\mu\,+\lambda}\tr(  \mathcal{E}_{\varphi^\natural,\overline{Q}_{e}^{\natural} } )n_0\otimes n_0\Big)}^2\\
	&\quad+2\dyniprod{\sym  \mathcal{E}_{\varphi^\natural,\overline{Q}_{e}^{\natural} }  ,\sym\Big(\frac{\mu_c\,-\mu\,}{\mu_c\,+\mu\,} \mathcal{E}_{\varphi^\natural,\overline{Q}_{e}^{\natural} } ^Tn_0\otimes n_0-\frac{\lambda}{2\,\mu\,+\lambda}\tr(  \mathcal{E}_{\varphi^\natural,\overline{Q}_{e}^{\natural} } )n_0\otimes n_0\Big)}.
	\end{align}
\allowdisplaybreaks	We have
	\begin{align}
	\nonumber \norm{\sym\Big(\frac{\mu_c\,-\mu\,}{\mu_c\,+\mu\,}  \mathcal{E}_{\varphi^\natural,\overline{Q}_{e}^{\natural} } ^Tn_0\otimes n_0&-\frac{\lambda}{2\,\mu\,+\lambda}\tr(  \mathcal{E}_{\varphi^\natural,\overline{Q}_{e}^{\natural} } )n_0\otimes n_0\Big)}^2\\
	\nonumber &=\frac{(\mu_c\,-\mu\,)^2}{(\mu_c\,+\mu\,)^2}\norm{\sym(  \mathcal{E}_{\varphi^\natural,\overline{Q}_{e}^{\natural} }^T n_0\otimes n_0)}^2+\frac{\lambda^2}{(2\,\mu\,+\lambda)^2}\tr(  \mathcal{E}_{\varphi^\natural,\overline{Q}_{e}^{\natural} } )^2\norm{n_0\otimes n_0}^2\\
	\nonumber &\quad-2\,\frac{\mu_c\,-\mu\,}{\mu_c\,+\mu\,}\;\frac{\lambda}{2\,\mu\,+\lambda}\tr(  \mathcal{E}_{\varphi^\natural,\overline{Q}_{e}^{\natural} } )\dyniprod{\sym(  \mathcal{E}_{\varphi^\natural,\overline{Q}_{e}^{\natural} } ^Tn_0\otimes n_0),n_0\otimes n_0}\\
	\nonumber &= \frac{(\mu_c\,-\mu\,)^2}{(\mu_c\,+\mu\,)^2}\dyniprod{\sym(  \mathcal{E}_{\varphi^\natural,\overline{Q}_{e}^{\natural} }^T n_0\otimes n_0),\sym(  \mathcal{E}_{\varphi^\natural,\overline{Q}_{e}^{\natural} }^T n_0\otimes n_0)}+\frac{\lambda^2}{(2\,\mu\,+\lambda)^2}\tr(  \mathcal{E}_{\varphi^\natural,\overline{Q}_{e}^{\natural} } )^2\\\notag
	&\quad-\frac{\mu_c\,-\mu\,}{\mu_c\,+\mu\,}\;\frac{\lambda}{2\,\mu\,+\lambda}\tr(  \mathcal{E}_{\varphi^\natural,\overline{Q}_{e}^{\natural} } )\dyniprod{  \mathcal{E}_{\varphi^\natural,\overline{Q}_{e}^{\natural} }^T n_0\otimes n_0,n_0\otimes n_0}\\&\quad -\frac{\mu_c\,-\mu\,}{\mu_c\,+\mu\,}\;\frac{\lambda}{2\,\mu\,+\lambda}\tr(  \mathcal{E}_{\varphi^\natural,\overline{Q}_{e}^{\natural} } )\dyniprod{n_0\otimes n_0\,  \mathcal{E}_{\varphi^\natural,\overline{Q}_{e}^{\natural} } ,n_0\otimes n_0}\\
	\nonumber &= \frac{(\mu_c\,-\mu\,)^2}{4(\mu_c\,+\mu\,)^2}\iprod{  \mathcal{E}_{\varphi^\natural,\overline{Q}_{e}^{\natural} }^T n_0\otimes n_0,  \mathcal{E}_{\varphi^\natural,\overline{Q}_{e}^{\natural} }^T n_0\otimes n_0}+\frac{(\mu_c\,-\mu\,)^2}{4(\mu_c\,+\mu\,)^2}\iprod{  \mathcal{E}_{\varphi^\natural,\overline{Q}_{e}^{\natural} }^T n_0\otimes n_0,n_0\otimes n_0 \, \mathcal{E}_{\varphi^\natural,\overline{Q}_{e}^{\natural} } }\\
	\nonumber &\quad+\frac{(\mu_c\,-\mu\,)^2}{4(\mu_c\,+\mu\,)^2}\iprod{n_0\otimes n_0\,  \mathcal{E}_{\varphi^\natural,\overline{Q}_{e}^{\natural} } ,  \mathcal{E}_{\varphi^\natural,\overline{Q}_{e}^{\natural} }^T n_0\otimes n_0}+\frac{(\mu_c\,-\mu\,)^2}{4(\mu_c\,+\mu\,)^2}\iprod{n_0\otimes n_0\,  \mathcal{E}_{\varphi^\natural,\overline{Q}_{e}^{\natural} } ,n_0\otimes n_0\,  \mathcal{E}_{\varphi^\natural,\overline{Q}_{e}^{\natural} } }\\
	\nonumber &\quad+\frac{\lambda^2}{(2\,\mu\,+\lambda)^2}\tr(  \mathcal{E}_{\varphi^\natural,\overline{Q}_{e}^{\natural} } )^2=\frac{(\mu_c\,-\mu\,)^2}{2(\mu_c\,+\mu\,)^2}\norm{  \mathcal{E}_{\varphi^\natural,\overline{Q}_{e}^{\natural} }^T n_0}^2+\frac{\lambda^2}{(2\,\mu\,+\lambda)^2}\tr(  \mathcal{E}_{\varphi^\natural,\overline{Q}_{e}^{\natural} } )^2.
	\end{align}
	Since, using \eqref{identitatiPI} we have $\iprod{  \mathcal{E}_{\varphi^\natural,\overline{Q}_{e}^{\natural} } ^T n_0\otimes n_0,n_0\otimes n_0}= \iprod{   n_0\otimes n_0,\mathcal{E}_{\varphi^\natural,\overline{Q}_{e}^{\natural} } n_0\otimes n_0}=0$, 
	and since we have used the matrix expression $  \mathcal{E}_{\varphi^\natural,\overline{Q}_{e}^{\natural} } =(*|*|0)[(\nabla_x\Theta)^\natural ]^{-1}$ and $n_0\otimes n_0=(0|0|n_0)[(\nabla_x\Theta)^\natural(0) ]^{-1}$, we deduce
	\begin{align}\label{transpose Ems}
	\nonumber\iprod{  \mathcal{E}_{\varphi^\natural,\overline{Q}_{e}^{\natural} }^T n_0\otimes n_0,  \mathcal{E}_{\varphi^\natural,\overline{Q}_{e}^{\natural} }^T n_0\otimes n_0}\notag&=\bigl\langle  \mathcal{E}_{\varphi^\natural,\overline{Q}_{e}^{\natural} }^T (0|0|n_0)[(\nabla_x\Theta)^\natural(0) ]^{-1},   \mathcal{E}_{\varphi^\natural,\overline{Q}_{e}^{\natural} }^T (0|0|n_0)[(\nabla_x\Theta)^\natural(0) ]^{-1}\bigr\rangle\\
	\nonumber&=\bigl\langle(0|0|  \mathcal{E}_{\varphi^\natural,\overline{Q}_{e}^{\natural} }^T n_0)^T(0|0|  \mathcal{E}_{\varphi^\natural,\overline{Q}_{e}^{\natural} }^T n_0), [(\nabla_x\Theta)^\natural(0) ]^{-1}[(\nabla_x\Theta)^\natural(0) ]^{-T}\bigr\rangle\\
	&=\iprod{(0|0|  \mathcal{E}_{\varphi^\natural,\overline{Q}_{e}^{\natural} }^T n_0)^T(0|0|  \mathcal{E}_{\varphi^\natural,\overline{Q}_{e}^{\natural} }^T n_0),(\widehat{\rm I}_{y_0})^{-1}}\\
	&=\dyniprod{\begin{pmatrix}
		0&0&0\\0&0&0\\&  \mathcal{E}_{\varphi^\natural,\overline{Q}_{e}^{\natural} }^T n_0
		\end{pmatrix}(0|0|  \mathcal{E}_{\varphi^\natural,\overline{Q}_{e}^{\natural} }^T n_0),\begin{pmatrix}
		*&*&0\\ *&*&0\\0&0&1
		\end{pmatrix}}=\iprod{  \mathcal{E}_{\varphi^\natural,\overline{Q}_{e}^{\natural} }^T n_0,  \mathcal{E}_{\varphi^\natural,\overline{Q}_{e}^{\natural} }^T n_0}=\norm{  \mathcal{E}_{\varphi^\natural,\overline{Q}_{e}^{\natural} }^T n_0}^2.\notag
	\end{align}
	On the other hand,
	\begin{align}
	2&\dyniprod{\sym  \mathcal{E}_{\varphi^\natural,\overline{Q}_{e}^{\natural} } ,\sym(\frac{\mu_c\,-\mu\,}{\mu_c\,+\mu\,}\mathcal{E}_{\varphi^\natural,\overline{Q}_{e}^{\natural} }^Tn_0\otimes n_0-\frac{\lambda}{2\,\mu\,+\lambda}\tr(  \mathcal{E}_{\varphi^\natural,\overline{Q}_{e}^{\natural} } )n_0\otimes n_0)}\notag\\
	&=\frac{1}{2}\dyniprod{  \mathcal{E}_{\varphi^\natural,\overline{Q}_{e}^{\natural} } +  \mathcal{E}_{\varphi^\natural,\overline{Q}_{e}^{\natural} }^T,\frac{\mu_c\,-\mu\,}{\mu_c\,+\mu\,}\mathcal{E}_{\varphi^\natural,\overline{Q}_{e}^{\natural} }^Tn_0\otimes n_0+\frac{\mu_c\,-\mu\,}{\mu_c\,+\mu\,}n_0\otimes n_0\;  \mathcal{E}_{\varphi^\natural,\overline{Q}_{e}^{\natural} } -\frac{2\lambda}{2\,\mu\,+\lambda}\tr(  \mathcal{E}_{\varphi^\natural,\overline{Q}_{e}^{\natural} } )n_0\otimes n_0}\notag\\
	&=\frac{\mu_c\,-\mu\,}{2(\mu_c\,+\mu\,)}\dyniprod{  \mathcal{E}_{\varphi^\natural,\overline{Q}_{e}^{\natural} } ,  \mathcal{E}_{\varphi^\natural,\overline{Q}_{e}^{\natural} }^T n_0\otimes n_0}+\frac{\mu_c\,-\mu\,}{2(\mu_c\,+\mu\,)}\dyniprod{  \mathcal{E}_{\varphi^\natural,\overline{Q}_{e}^{\natural} } ,n_0\otimes n_0\;  \mathcal{E}_{\varphi^\natural,\overline{Q}_{e}^{\natural} } }\\
	&\quad-\frac{\lambda}{(2\,\mu\,+\lambda)}\tr(  \mathcal{E}_{\varphi^\natural,\overline{Q}_{e}^{\natural} } )\iprod{  \mathcal{E}_{\varphi^\natural,\overline{Q}_{e}^{\natural} } ,n_0\otimes n_0}+\frac{\mu_c\,-\mu\,}{2(\mu_c\,+\mu\,)}\dyniprod{  \mathcal{E}_{\varphi^\natural,\overline{Q}_{e}^{\natural} }^T,  \mathcal{E}_{\varphi^\natural,\overline{Q}_{e}^{\natural} }^T n_0\otimes n_0}\notag\\
	&\quad+\frac{\mu_c\,-\mu\,}{2(\mu_c\,+\mu\,)}\dyniprod{  \mathcal{E}_{\varphi^\natural,\overline{Q}_{e}^{\natural} }^T,n_0\otimes n_0\,  \mathcal{E}_{\varphi^\natural,\overline{Q}_{e}^{\natural} } }-\frac{\lambda}{(2\,\mu\,+\lambda)}\tr(  \mathcal{E}_{\varphi^\natural,\overline{Q}_{e}^{\natural} } )\iprod{  \mathcal{E}_{\varphi^\natural,\overline{Q}_{e}^{\natural} }^T,n_0\otimes n_0}=\frac{\mu_c\,-\mu\,}{\mu_c\,+\mu\,}\norm{  \mathcal{E}_{\varphi^\natural,\overline{Q}_{e}^{\natural} }^T n_0}^2,\notag
	\end{align}
	due to \eqref{EKperp}.
	Therefore, \eqref{reduced} can be reduced to
	\begin{align}\label{symter}
	\norm{\sym(\overline{Q}_{e}^{\natural,T}&(\nabla_{(\eta_1,\eta_2)} \varphi^\natural|c)[(\nabla_x\Theta)^\natural ]^{-1}-\id_3)}^2\\&=\norm{\sym  \mathcal{E}_{\varphi^\natural,\overline{Q}_{e}^{\natural} } }^2+\frac{(\mu_c\,-\mu\,)^2}{2(\mu_c\,+\mu\,)^2}\norm{  \mathcal{E}_{\varphi^\natural,\overline{Q}_{e}^{\natural} }^T n_0}^2+\frac{\lambda^2}{(2\,\mu\,+\lambda)^2}\tr(  \mathcal{E}_{\varphi^\natural,\overline{Q}_{e}^{\natural} } )^2+\frac{\mu_c\,-\mu\,}{(\mu_c\,+\mu\,)}\norm{  \mathcal{E}_{\varphi^\natural,\overline{Q}_{e}^{\natural} }^T n_0}^2.\notag
	\end{align}

	Now we continue the calculations for the skew symmetric part,
	\begin{align}\label{skewap}
	\nonumber\norm{\skew(\overline{Q}_{e}^{\natural,T}(\nabla_{(\eta_1,\eta_2)} \varphi^\natural|d^*)[(\nabla_x\Theta)^\natural ]^{-1})}^2\notag &=\norm{\skew(\overline{Q}_{e}^{\natural,T}(\nabla_{(\eta_1,\eta_2)} \varphi^\natural|0)[(\nabla_x\Theta)^\natural ]^{-1})}^2\notag + \norm{\skew((0|0|\overline{Q}_{e}^{\natural,T}d^*)[(\nabla_x\Theta)^\natural ]^{-1})}^2\\
	&\quad+2\bigl\langle\skew(\overline{Q}_{e}^{\natural,T}(\nabla_{(\eta_1,\eta_2)} \varphi^\natural|0)[(\nabla_x\Theta)^\natural ]^{-1}), \skew((0|0|\overline{Q}_{e}^{\natural,T}d^*)[(\nabla_x\Theta)^\natural ]^{-1})\bigr\rangle.
	\end{align}
	In a similar manner, we calculate the terms separately. Since $n_0\otimes n_0$ is symmetric, we obtain
	\begin{align}
	\norm{\skew((0|0|\overline{Q}_{e}^{\natural,T}d^*)[(\nabla_x\Theta)^\natural ]^{-1})}^2\notag&=\norm{\skew(n_0\otimes n_0+\frac{\mu_c\,-\mu\,}{\mu_c\,+\mu\,}  \mathcal{E}_{\varphi^\natural,\overline{Q}_{e}^{\natural} }^T \, n_0\otimes n_0-\frac{\lambda}{2\,\mu\,+\lambda}\tr(  \mathcal{E}_{\varphi^\natural,\overline{Q}_{e}^{\natural} }^T )\, n_0\otimes n_0)}^2\\
	&=\frac{(\mu_c\,-\mu\,)^2}{(\mu_c\,+\mu\,)^2}\norm{\skew(  \mathcal{E}_{\varphi^\natural,\overline{Q}_{e}^{\natural} }^T\, n_0\otimes n_0)}^2.\notag
	\end{align}
	But, we have
	\begin{align}
	\norm{\skew(  \mathcal{E}_{\varphi^\natural,\overline{Q}_{e}^{\natural} } ^Tn_0\otimes n_0)}^2&=\frac{1}{4}\dyniprod{  \mathcal{E}_{\varphi^\natural,\overline{Q}_{e}^{\natural} }^T n_0\otimes n_0,  \mathcal{E}_{\varphi^\natural,\overline{Q}_{e}^{\natural} }^T\, n_0\otimes n_0}-\frac{1}{4}\dyniprod{  \mathcal{E}_{\varphi^\natural,\overline{Q}_{e}^{\natural} }^T n_0\otimes n_0,n_0\otimes n_0 \, \mathcal{E}_{\varphi^\natural,\overline{Q}_{e}^{\natural} } }\\
	&\quad-\frac{1}{4}\dyniprod{n_0\otimes n_0 \, \mathcal{E}_{\varphi^\natural,\overline{Q}_{e}^{\natural} } ,  \mathcal{E}_{\varphi^\natural,\overline{Q}_{e}^{\natural} }^T\, n_0\otimes n_0}+\frac{1}{4}\dyniprod{n_0\otimes n_0\,  \mathcal{E}_{\varphi^\natural,\overline{Q}_{e}^{\natural} } ,n_0\otimes n_0 \, \mathcal{E}_{\varphi^\natural,\overline{Q}_{e}^{\natural} } }=\frac{1}{2}\norm{  \mathcal{E}_{\varphi^\natural,\overline{Q}_{e}^{\natural} }^T n_0}^2\,,\notag
	\end{align}
	where we used the fact that $(n_0\otimes n_0)^2=(n_0\otimes n_0)$. The difficulty in the skew symmetric part of (\ref{skewap}) is solved in the following calculation
	\begin{align}\label{appskew}
	\nonumber 2\bigl\langle\skew(\overline{Q}_{e}^{\natural,T}(\nabla_{(\eta_1,\eta_2)} \varphi^\natural|0)&[(\nabla_x\Theta)^\natural ]^{-1}), \skew((0|0|\overline{Q}_{e}^{\natural,T}d^*)[(\nabla_x\Theta)^\natural ]^{-1})\bigr\rangle\\
	\nonumber&= 2\,\frac{(\mu_c\,-\mu\,)}{(\mu_c\,+\mu\,)}\dyniprod{\skew(\overline{Q}_{e}^{\natural,T}(\nabla_{(\eta_1,\eta_2)} \varphi^\natural|0)[(\nabla_x\Theta)^\natural ]^{-1}),\skew(  \mathcal{E}_{\varphi^\natural,\overline{Q}_{e}^{\natural} }^T n_0\otimes n_0)}\\
	&=\frac{(\mu_c\,-\mu\,)}{2(\mu_c\,+\mu\,)}\iprod{\overline{Q}_{e}^{\natural,T}(\nabla_{(\eta_1,\eta_2)} \varphi^\natural|0)[(\nabla_x\Theta)^\natural ]^{-1},  \mathcal{E}_{\varphi^\natural,\overline{Q}_{e}^{\natural} }^T n_0\otimes n_0}\\&\quad \notag-\frac{(\mu_c\,-\mu\,)}{2(\mu_c\,+\mu\,)}\iprod{\overline{Q}_{e}^{\natural,T}(\nabla_{(\eta_1,\eta_2)} \varphi^\natural|0)[(\nabla_x\Theta)^\natural ]^{-1},n_0\otimes n_0\,  \mathcal{E}_{\varphi^\natural,\overline{Q}_{e}^{\natural} }}\\
	\nonumber&\quad-\frac{(\mu_c\,-\mu\,)}{2(\mu_c\,+\mu\,)}\iprod{(\overline{Q}_{e}^{\natural,T}(\nabla_{(\eta_1,\eta_2)} \varphi^\natural|0)[(\nabla_x\Theta)^\natural ]^{-1})^T,  \mathcal{E}_{\varphi^\natural,\overline{Q}_{e}^{\natural} }^T n_0\otimes n_0}\\&\quad +\frac{(\mu_c\,-\mu\,)}{2(\mu_c\,+\mu\,)}\iprod{(\overline{Q}_{e}^{\natural,T}(\nabla_{(\eta_1,\eta_2)} \varphi^\natural|0)[(\nabla_x\Theta)^\natural ]^{-1})^T,n_0\otimes n_0 \, \mathcal{E}_{\varphi^\natural,\overline{Q}_{e}^{\natural} } }=-\frac{(\mu_c\,-\mu\,)}{(\mu_c\,+\mu\,)}\norm{  \mathcal{E}_{\varphi^\natural,\overline{Q}_{e}^{\natural} }^T n_0}^2.\notag
	\end{align}
	Therefore,
	\begin{align}
	2\bigl\langle &\skew(\overline{Q}_{e}^{\natural,T}(\nabla_{(\eta_1,\eta_2)} \varphi^\natural|0)[(\nabla_x\Theta)^\natural ]^{-1}), \skew((0|0|\overline{Q}_{e}^{\natural,T}d^*)[(\nabla_x\Theta)^\natural ]^{-1})\bigr\rangle
	=-\frac{(\mu_c\,-\mu\,)}{(\mu_c\,+\mu\,)}\norm{  \mathcal{E}^T_{\varphi^\natural,\overline{Q}_{e}^{\natural} } n_0}^2\,,
	\end{align}
	and we obtain
	\begin{align}\label{skew}
	\nonumber\norm{\skew(\overline{Q}_{e}^{\natural,T}(\nabla_{(\eta_1,\eta_2)} \varphi^\natural|d^*)[(\nabla_x\Theta)^\natural ]^{-1})}^2&=\norm{\skew(\overline{Q}_{e}^{\natural,T}(\nabla_{(\eta_1,\eta_2)} \varphi^\natural|0)[(\nabla_x\Theta)^\natural ]^{-1})}^2 +\frac{(\mu_c\,-\mu\,)^2}{2(\mu_c\,+\mu\,)^2}\norm{  \mathcal{E}^T_{\varphi^\natural,\overline{Q}_{e}^{\natural} } n_0}^2\\
	&\quad-\frac{(\mu_c\,-\mu\,)}{(\mu_c\,+\mu\,)}\norm{  \mathcal{E}^T_{\varphi^\natural,\overline{Q}_{e}^{\natural} } n_0}^2.
	\end{align}
	The last requirement for our calculations, is
	\begin{align}\label{trace}
	\nonumber&\Big[\tr\Big(\sym(\overline{Q}_{e}^{\natural,T}(\nabla_{(\eta_1,\eta_2)} \varphi^\natural|d^*)[(\nabla_x\Theta)^\natural ]^{-1}-\id_3)\Big)\Big]^2\\
	&=\Big(\tr\big(\sym((\overline{Q}_{e}^{\natural,T}\nabla_{(\eta_1,\eta_2)} \varphi^\natural-[\nabla y_0]^\natural |0)[(\nabla_x\Theta)^\natural ]^{-1})\big) +\tr\big(\sym((0|0|\overline{Q}_{e}^{\natural,T}d^*-n_0)[(\nabla_x\Theta)^\natural ]^{-1})\big)\Big)^2\\
	\nonumber&=\Big(\tr(  \mathcal{E}_{\varphi^\natural,\overline{Q}_{e}^{\natural} } )+\frac{(\mu_c\,-\mu\,)}{2(\mu_c\,+\mu\,)}(\iprod{  \mathcal{E}_{\varphi^\natural,\overline{Q}_{e}^{\natural} }^T n_0\otimes n_0,\id_3}+\iprod{n_0\otimes n_0\,  \mathcal{E}_{\varphi^\natural,\overline{Q}_{e}^{\natural} } ,\id_3})-\frac{\lambda}{2\,\mu\,+\lambda}\tr(  \mathcal{E}_{\varphi^\natural,\overline{Q}_{e}^{\natural} } )\underbrace{\iprod{n_0\otimes n_0,\id_3}}_{\iprod{n_0,n_0}=1}\Big)^2\\
	&=\Big(\frac{2\,\mu\,}{2\,\mu\,+\lambda}\tr(  \mathcal{E}_{\varphi^\natural,\overline{Q}_{e}^{\natural} } )+\frac{(\mu_c\,-\mu\,)}{2(\mu_c\,+\mu\,)}(\iprod{  \mathcal{E}_{\varphi^\natural,\overline{Q}_{e}^{\natural} }^T ,n_0\otimes n_0}+\iprod{  \mathcal{E}_{\varphi^\natural,\overline{Q}_{e}^{\natural} } ,n_0\otimes n_0})\Big)^2=\frac{4\mu\,^2}{(2\,\mu\,+\lambda)^2}\tr(  \mathcal{E}_{\varphi^\natural,\overline{Q}_{e}^{\natural} } )^2.\notag
	\end{align}

	\subsection{Homogenized quadratic curvature energy}\label{homcurghiba}
	In \cite{Ghiba2022}, the authors obtained the homogenized curvature energy for the following curvature energy 
	\begin{align}\label{curvenergy}
	W_{\text{curv}}(\Gamma^\natural[(\nabla_x\Theta)^\natural ]^{-1})&=\mu L_c^2\Big(b_1\norm{\sym \Gamma^\natural[(\nabla_x\Theta)^\natural ]^{-1}}^2+b_2\,\norm{\skew\Gamma^\natural[(\nabla_x\Theta)^\natural ]^{-1}}^2+b_3\tr(\Gamma^\natural[(\nabla_x\Theta)^\natural ]^{-1})^2\Big)\,,
	\end{align}
	as
	\begin{align}\label{homocurv}
	\nonumber	W_{\text{curv}}^{\text{hom}}(\mathcal{K}_{e,s})&=\mu L_c^2\Big(b_1\norm{\sym\mathcal{K}_{e,s}}^2+b_2\norm{\skew \mathcal{K}_{e,s}}^2-\frac{(b_1-b_2)^2}{2(b_1+b_2)}\norm{\mathcal{K}_{e,s}^Tn_0}^2+\frac{b_1b_3}{(b_1+b_3)}\tr(\mathcal{K}_{e,s})^2\Big)\\
	\nonumber&=\mu L_c^2\Big(b_1\norm{\sym\mathcal{K}_{e,s}^\parallel}^2+b_2\norm{\skew \mathcal{K}_{e,s}^\parallel}^2-\frac{(b_1-b_2)^2}{2(b_1+b_2)}\norm{\mathcal{K}_{e,s}^Tn_0}^2+\frac{b_1b_3}{(b_1+b_3)}\tr(\mathcal{K}_{e,s}^\parallel)^2+\frac{b_1+b_2}{2}\norm{\mathcal{K}_{e,s}^Tn_0}\Big)\\
	&=\mu L_c^2\Big(b_1\norm{\sym\mathcal{K}_{e,s}^\parallel}^2+b_2\norm{\skew \mathcal{K}_{e,s}^\parallel}^2+\frac{b_1b_3}{(b_1+b_3)}\tr(\mathcal{K}_{e,s}^\parallel)^2+\frac{2b_1b_2}{b_1+b_2}\norm{\mathcal{K}_{e,s}^\perp}\Big)\,,
	\end{align}
	where 
	$\mathcal{K}_{e,s}=(\Gamma_1|\Gamma_2|0)[(\nabla_x\Theta)^\natural ]^{-1}$ with the decomposition 
	\begin{align}
	X=X^\parallel+X^\perp, \qquad  \qquad \qquad X^\parallel\coloneqq{\rm A}_{y_0} \,X,  \qquad \qquad \qquad X^\perp\coloneqq(\id_3-{\rm A}_{y_0}) \,X,
	\end{align}
	for every matrix $X$.

\end{appendix}
\end{document}